\documentclass[a4paper]{article}

\usepackage[english]{babel}
\usepackage[T1]{fontenc}

\usepackage[a4paper,top=3cm,bottom=2cm,left=3cm,right=3cm,marginparwidth=2cm]{geometry}

\newcommand{\T}{\mathbb{T}}

\usepackage[colorinlistoftodos]{todonotes}
\usepackage[colorlinks=true, allcolors=blue]{hyperref}
\usepackage{makeidx}
\usepackage{amssymb}
\usepackage{url}
\usepackage[square,sort,comma,numbers]{natbib}
\usepackage{xcolor}
\usepackage{hyperref}
\usepackage{pdfcomment}
\usepackage{subcaption}
\usepackage{enumitem}
\usepackage{framed}
\usepackage{graphicx}
\usepackage{longtable}
\usepackage{amsfonts}
\usepackage{amsmath}
\usepackage{amsthm}
\usepackage{authblk}
\usepackage[noend,ruled,noline,linesnumbered]{algorithm2e}
\usepackage{setspace} 
\usepackage{fancyhdr}
\usepackage{graphicx} 

\newtheorem{definition}{Definition}[section]
\newtheorem{lemma}{Lemma}[section]
\newtheorem{theorem}{Theorem}[section]
\newtheorem{remark}{Remark}[section]
\newtheorem{proposition}{Proposition}[section]

\newtheorem{assumption}{Assumption}[section]
\newtheorem{example}{Example}[section]
\title{Observable Lyapunov Exponents for Globally Coupled Maps in the Mean-Field Limit}
\date{}
\author[1]{Masood Ahmad}
\author[1]{Matteo Tanzi}
\affil[1]{Department of Mathematics, King's College London, London, United Kingdom}

\begin{document}
\maketitle

\begin{abstract}
Large coupled systems may be chaotic at the microscopic level while exhibiting stable collective behaviour at the macroscopic level. Distinguishing between these two forms of instability has been a problem in the physics literature, and several approaches have been proposed. In this paper, we address this problem by introducing and studying \emph{observable Lyapunov exponents}, which measure the rate at which perturbations of initial conditions grow or decay when viewed through a chosen observable, rather than in the full phase space. For globally coupled maps, we consider finite-time observable Lyapunov exponents and define their mean-field counterpart by first taking the infinite-system limit and then the long-time limit. Our main result shows that, for a broad class of symmetric macroscopic observables, these exponents turn out to be determined by the linearisation of the self-consistent transfer operator, which is the nonlinear  operator governing the evolution of the population distribution. We then apply this result to weakly coupled uniformly expanding maps and prove that their mean-field Lyapunov exponents are negative near a stationary mean-field state. Thus, although the microscopic dynamics is chaotic,  perturbations decay at the macroscopic level. Our results provide a rigorous framework for distinguishing microscopic from macroscopic chaos in large interacting systems.
\end{abstract}

\section{Introduction}
\subsection{Globally Coupled Maps}
Globally Coupled Maps (GCMs) are mathematical models for large populations of interacting units. Each unit follows the same local dynamics and is coupled to all others through a mean-field term determined by the population average. Its evolution therefore combines its intrinsic dynamics with an interaction generated by the collective state of the population. GCMs are often studied in the thermodynamic limit, with the system's size limiting to infinity, where the system is described by the distribution of states rather than by tracking the individual units. The evolution of this distribution is governed by a self-consistent operator, which we define rigorously in Section~\ref{setting}.

GCMs provide simple yet paradigmatic examples of complex systems and are of considerable interest as dynamical systems in their own right: despite their elementary formulation, they exhibit a remarkably wide range of dynamical behaviours. Through numerical studies of globally coupled logistic maps, Kaneko~\cite{kaneko1989chaotic,kaneko1990clustering} identified several distinct dynamical regimes: \emph{synchronised}, \emph{ordered}, \emph{partially ordered}, and \emph{turbulent}. These regimes arise as the initial conditions, coupling strength, and parameter of the logistic family are varied, and are distinguished by the extent to which the units organise into synchronised clusters or remain apparently uncorrelated. A closely related phenomenon had previously been observed for coupled continuous-time oscillators in~\cite{wiesenfeld1989attractor}, where it was termed \emph{attractor crowding}: the number of attractors grows factorially with the system size.
% Furthermore, \cite{kaneko1989chaotic,shimada2000resolution} observed the phenomenon of \emph{posi-nega switching}, where units form two clusters of opposite sign and periodically exchange membership while the large-scale organisation of the system remains invariant. synchronisation was also investigated in \cite{cosenza1998synchronization}, using the same globally coupled framework as in \eqref{kaneko globally} but with local dynamics given by the logarithmic map $f(x)=a+\ln|x|$.\\
The most interesting regime for our purposes is the turbulent regime, where the states of the units evolve chaotically \cite{kaneko1990globally, kaneko1992mean, pikovsky1994globally, kaneko1995remarks, pikovsky1994collective, shibata1997heterogeneity}.

Globally coupled systems can exhibit chaos at two distinct levels: microscopic and macroscopic. Microscopic chaos is characterised by the conventional Lyapunov spectrum, which measures the growth of the distance between nearby initial conditions in the full phase space. Macroscopic chaos instead concerns the sensitive dependence from the initial condition of "global" behaviour, as measured for example by the mean-field. Whether the conventional Lyapunov spectrum captures the emergence of macroscopic chaos, or new approaches are needed, has so far remained unclear and has been object of study mostly in the physics literature (that we review below).

 Although these results provide important insights, they rely largely on numerical studies of specific systems and observables. The main contribution of this paper is to introduce a theoretical framework in which this discussion can be put on rigorous footing with  the introduction of the notion of an \emph{observable Lyapunov exponent}, which quantifies the rates at which observations along an orbit diverge from or converge towards those along a nearby perturbed orbit. For globally coupled systems, we focus on  mean-field observables, that capture global features of the system, and we rigorously relate the corresponding \emph{mean-field observable Lyapunov exponents} to the derivative of the self-consistent transfer operator, which is the nonlinear operator governing the evolution of the mean field in the thermodynamic limit. These exponents characterise macroscopic chaos by measuring the divergence of observable values along nearby trajectories, rather than that of the microscopic trajectories themselves. This distinction is fundamental: two trajectories may diverge exponentially while the values of a low-dimensional observable along them separate much more slowly, or not at all. Indeed, for uniformly expanding coupled systems, we prove that the mean-field observable Lyapunov exponents are negative, showing that microscopic chaos does not necessarily imply macroscopic chaos.

\paragraph{Acknowledgements.}
The authors thank Edmilson Roque dos Santos for helpful discussions. M.T. acknowledges support from EPSRC--FAPESP Grant No.~2023/13706 and EPSRC Grant No.~UKRI1021. During the preparation of this manuscript, generative AI tools were used for language editing/assistance with proofreading. All resulting suggestions were independently checked and edited by the authors, who take full responsibility for the contents of the manuscript.

\subsection{Microscopic vs Macroscopic Chaos in GCMs}

The distinction between microscopic and macroscopic chaos has  been investigated so far largely through numerical and heuristic approaches in the physics literature. The earliest studies of macroscopic chaos were carried out independently in \cite{shibata1998collective} and~\cite{cencini1999macroscopic}. Both works use finite-size Lyapunov exponent (FSLE) approaches.  \cite{shibata1998collective}  investigated heterogeneous (logistic map) GCMs and analysed collective behaviour through mean-field 
observables, such as
\begin{equation}
h(t)=\frac{1}{N}\sum_{i=1}^{N}h(x_i(t)).
\end{equation} while \cite{cencini1999macroscopic} studied both homogeneous (tent map) and heterogeneous (logistic map) GCMs, focusing  on macroscopic 
observables such as the mean (centre-of-mass variable)
\begin{equation}{\label{centerofmass}}
m(t)=\frac{1}{N}\sum_{i=1}^{N}x_i(t)
\end{equation}
% and the fluctuation variable
% \begin{equation}{\label{fluctuation}}
% \sigma(t)=
% \sqrt{
% \frac{1}{N}\sum_{i=1}^{N}x_i^2(t)
% -
% \left(\frac{1}{N}\sum_{i=1}^{N}x_i(t)\right)^2
% }, 
% \end{equation}
Both works brought evidence that macroscopic chaotic behaviour of these quantities cannot be revealed at the, level of microscopic dynamics. Instead, finite perturbations of the initial condition comparable to or greater than the collective fluctuation scale, $O(N^{-1/2})$ are required.

 A different viewpoint for detecting macroscopic chaos was introduced in \cite{takeuchi2009lyapunov}, and further developed in \cite{takeuchi2013collective}. The authors argue that macroscopic chaotic behaviour is encoded in the standard Lyapunov analysis through the structure of covariant Lyapunov vectors (CLVs). By examining the localisation properties of these vectors, they distinguished between \emph{microscopic modes}, which are localised and have only a few appreciably nonzero components, and \emph{collective modes}, in which most components are of comparable magnitude. The delocalisation is quantified using the inverse participation ratio (IPR):
$
Y_2^{(j)} = \left\langle \sum_i |\delta x_i^{(j)}|^4 \right\rangle_t,$
where the CLV is normalised such that $\sum_i |\delta x_i^{(j)}|^2=1$. For a localised mode, $Y_2^{(j)} = O(1)$, while for a collective delocalised mode, $Y_2^{(j)} \sim N^{-1}$. Using this criterion, they showed that Lyapunov exponents associated with collective modes converge to the leading exponent of the macroscopic Perron--Frobenius dynamics. This suggests that collective chaos is encoded within the conventional Lyapunov spectrum. However, as  already observed in that paper and as we make precise here, microscopic and macroscopic chaos need not coincide: coupled uniformly expanding systems may be chaotic at the microscopic level while exhibiting no macroscopic chaos.

 Attention has been given also to the study of the Lyapunov spectrum and its scaling with the dimension of the system. In~\cite{takeuchi2011extensive}, the authors argued that the Lyapunov spectrum of GCMs is not fully extensive; while an extensive band exists in the central part of the spectrum, sub-extensive bands persist at both ends even in the thermodynamic limit. Consequently, the largest Lyapunov exponent approaches its asymptotic value logarithmically:
$
\lambda_{\infty}-\lambda_{1}(N)\sim\frac{c}{\ln N}.
$
Revisiting this, Velasco, López and  Pazó \cite{velasco2021nonuniversal} borought instead evidence that the convergence of the largest Lyapunov exponent is strongly nonuniversal, with finite-size corrections potentially following a power-law scaling:
$
\lambda_{\infty}-\lambda_{1}(N)\sim\frac{c}{N^{\gamma}}, 0<\gamma\le1.
$
However, the Lyapunov spectrum characterises the instability of the full microscopic dynamics and does not, by itself, determine whether chaos emerges at the macroscopic level.

Another notable work studying Lyapunov exponents for coupled systems is \cite{koiller2010coupled} where the authors considered coupled-map networks over arbitrary finite graphs, focusing on the effect of network structure and coupling strength on the Lyapunov spectrum. For networks of circle diffeomorphisms or expanding maps, they rigorously prove that increasing the coupling strength produces a cascade of bifurcations in which Lyapunov exponents become negative and unstable directions are progressively stabilised. Their work therefore provides a geometric description of how coupling can suppress microscopic chaos.

\subsection{Setup and Main Results}
{\label{setting}}

\subsubsection{Observable Lyapunov exponents} We begin by recalling the definition of  Lyapunov exponents (see for example \cite{viana2014lectures} for a thorough treatment). Let $M$ be a Riemannian manifold and $F:M\to M$ be a $C^1(M,M)$ map, and denote 
$D_{\mathbf{x}}F^t$  the differential of the $t$-th iterate of $F$ at the point 
$\mathbf{x} \in M$. We are interested in the exponential growth rate of $\| D_{\mathbf{x}}F^t(\mathbf{v})\|$ for a vector $\mathbf{v} \in T_{\mathbf{x}}M \setminus \{0\}$
which is called Lyapunov exponent of $F$ at the point $\mathbf{x}$ in the direction $\mathbf{v}$ and is defined as
\begin{equation}\label{Eq:StLyap}
\lambda(\mathbf{x}, \mathbf{v}) := \lim_{t \to \infty} \frac{1}{t} \log \|D_{\mathbf{x}}F^t(\mathbf{v})\|
\end{equation}
when the limit exists\footnote{If the limit does not existence of limits, we use the upper and lower Lyapunov exponents,
defined by replacing the limit with $\limsup$ and $\liminf$ respectively}. Lyapunov exponents provide asymptotic rates of convergence and divergence of nearby orbits.

For GCMs in the large-$N$ regime, the high dimensionality of the phase space makes tracking individual trajectories and computing Lyapunov exponents  difficult; moreover, these microscopic quantities might not directly capture the macroscopic properties one is interested in. Instead, one may study the evolution of the system by measuring certain observables, which provide a lower-dimensional and macroscopic description of the underlying dynamics. This motivates the study of stability properties from the viewpoint of observables which are captured by the following notion of \emph{observable Lyapunov exponent}.  
 
\begin{definition}[Observable Lyapunov exponent]{\label{Lyapunovdefinition}}
Let $M$ be a Riemannian manifold, $F:M\to M$ a $C^1(M, M)$ map, and  $\psi:M\to\mathbb{R}$ be and an $C^1(M, \mathbb{R})$ observable. Define 
\begin{equation}{\label{Lyaponuvunderobservable}}
\lambda_\psi(\mathbf{x},\mathbf{v}) := \lim_{t \to \infty} \frac{1}{t} \log \big| D_{\mathbf{x}}(\psi \circ F^t)(\mathbf{v}) \big|,
\end{equation}
whenever this limit exists. We call $\lambda_\psi(\mathbf{x},\mathbf{v})$ an observable Lyapunov exponent.
\end{definition}

An observable Lyapunov exponent measures the exponential rate at which the values of a given observable along an orbit converge to or diverge from those along an infinitesimally perturbed nearby orbit. Notice that, by the chain rule, this is equivalent to the growth rate of a tangent vector under the linearised dynamics after it is "projected" with the differential of the observable $\psi$. The usual Lyapunov spectrum and the observable Lyapunov spectrum may or may not coincide. In Section~\ref{Lyapunov under observable}, we explore this relationship through examples and provide, in simple setups, criteria  for when the two spectra agree.
 
\subsubsection{GCMs and Mean-Field Observable Lyapunov Exponents} Our primary interest is in the study of observable Lyapunov exponents associated with macroscopic observables, such as the mean-field, for GCMs. Therefore, we now introduce the class of GCMs considered throughout this work and that can be also found for example in \cite{tanzi2022meanfieldcoupledsystemsselfconsistent}. 
Let M = \(\mathbb{T}=\mathbb{R}/\mathbb{Z}\)\footnote{We identify the torus $\mathbb{T}=\mathbb{R}/\mathbb{Z}$ with the circle represented by the interval $[0,1)$. Addition and subtraction on $\mathbb{T}$ are understood modulo $1$. Derivatives of maps defined on $\mathbb{T}$ are understood through their lifts to $\mathbb{R}$, and the mean value theorem is applied to the corresponding lift.} denote the one-dimensional torus. For each \(N\in\mathbb{N}\), we consider the system of globally coupled maps
$
F^{(N)}:\mathbb{T}^N\to\mathbb{T}^N,
$
defined by
$
F^{(N)}(\mathbf{x})
=
\bigl(F^{(N)}_1(\mathbf{x}),\dots,F^{(N)}_N(\mathbf{x})\bigr),
$
where \(\mathbf{x}=(x_1,\dots,x_N)\in\mathbb{T}^N\) and
\begin{equation}{\label{globallycoupledmap}}
F^{(N)}_k(\mathbf{x})
=
f(x_k)
+
\frac1N
\sum_{\substack{j\neq k}}^{N}
h(x_k,x_j)
\pmod{1},
\qquad k=1,\dots,N.
\end{equation}
Here \(f:\mathbb{T}\to\mathbb{T}\) prescribes the local dynamics and \(h:\mathbb{T}^2\to\mathbb{R}\) is the coupling function. We are often interested in the study of the evolution of macroscopic observables like mean fields:
\begin{equation}\label{empirical}
\psi^{(N)}(x_1,\ldots,x_N)
=
\frac{1}{N}\sum_{k=1}^{N}\varphi(x_k).
\end{equation}

Supposing that $\psi^{(N)}$ is differentiable, \(\mathbf{x}\in\mathbb{T}^N\), \(\mathbf{v}\in\mathbb{R}^N \textbackslash \{0\}\), and \(t\in\mathbb{N}\), define
\[
\chi_{\psi, N,t}(\mathbf{x},\mathbf{v})
:=
\frac{1}{t}
\log
\left|
D_{\mathbf{x}}\left(\psi^{(N)}\circ (F^{(N)})^t\right)(\mathbf{v})\right|.
\]
Considering that we want to study GCMs in the large $N$ limit, we have two main choices at this point: 

i) take the limit \(t\to\infty\) first, yielding the observable Lyapunov spectrum for fixed \(N\), and then \(N\to\infty\) to find what is a limit for the Lyapunov spectrum;

ii) or take the limit \(N\to\infty\)   first, followed by the limit \(t\to\infty\).

The two options i) and ii) above are, in general, likely to give different results.  If one is interested in the macroscopic properties of the system, ii) seems the more appropriate choice and leads to the following definition

\begin{definition}[Mean-Field Observable Lyapunov Exponents for GCMs] For $\mathbf{x}\in \mathbb T^{\mathbb N}$ and $\mathbf{v}\in \mathbb R^{\mathbb N}$
\begin{equation}{\label{Lyapunovther}}
\chi(\mathbf{x},\mathbf{v})
:=
\lim_{t\to\infty}
\lim_{N\to\infty}
\chi_{\psi, N,t}(\mathbf{x}^{(N)},\mathbf{v}^{(N)}),
\end{equation}
when the limit exists, and where $\mathbf{x}^{(N)},\,\mathbf{v}^{(N)}\in \mathbb R^N$ are the first $N$ components of $\mathbf{x}$ and $\mathbf{v}$ respectively.
\end{definition}

\noindent The main result of this paper, shows that there is an expression for \eqref{Lyapunovther} in terms of the differential of the self-consistent transfer operator that we now recall. 

\subsubsection{Self-consistent Transfer Operator and its Differential}  For a system of $N$ coupled maps as in \eqref{globallycoupledmap}, with states $x_1,\ldots,x_N$, the distribution of the states is described by the empirical measure
\[
\mu_N=\frac{1}{N}\sum_{i=1}^N\delta_{x_i}.
\]
If, as $N\to\infty$, the empirical measures converge to a limiting distribution $\mu$, then, under suitable assumptions, the evolution of the distribution in the infinite limit is governed by the nonlinear operator
\[
\mathcal{F}(\mu)=(f_\mu)_*\mu,
\]
where $(f_\mu)_*$ denotes the push-forward operator associated with the map
\begin{equation}\label{fmu}
f_\mu(x)
=
f(x)
+
\int_{\mathbb{T}}h(x,y)\,d\mu(y)
\pmod{1}.
\end{equation}
The operator $\mathcal{F}$ is known as the \emph{self-consistent transfer operator}. Thus, in the thermodynamic limit, the distribution of the system evolves through repeated applications of $\mathcal{F}$. Further details of this derivation can be found in Lemma~1 of \cite{galatolo2022self} and Section~3 of \cite{tanzi2022meanfieldcoupledsystemsselfconsistent}.

The differential of the self-consistent transfer operator has been used to study the stability of its fixed points, which represent stationary macroscopic states. For general GCMs, the self-consistent transfer operator can fail to be Fréchet differentiable when regarded as a self-map of a single Banach space $(B,\|\cdot\|)$; see \cite{castorrini2025differential}. A more suitable framework involves a strong space $(B_s,\|\cdot\|_s)$ and a weak space $(B_w,\|\cdot\|_w)$ satisfying
$
B_s\subseteq B_w,\,
\|\cdot\|_s\geq \|\cdot\|_w.
$
In this setting, perturbations are measured in the strong norm, while the remainder in the linear approximation is controlled in the weak norm.

Let $P_s\subset B_s$ and $P_w\subset B_w$ denote the corresponding sets of probability measures, and define the zero-mass subspaces
\[
V_s=\left\{\nu\in B_s:\nu(\mathbb{T})=0\right\},
\qquad
V_w=\left\{\nu\in B_w:\nu(\mathbb{T})=0\right\}.
\]
We consider the self-consistent transfer operator
$
\mathcal{F}:P_s\longrightarrow P_s.
$
We assume that a sufficiently small neighbourhood of $\mu^*\in P_s$ admits a Banach-manifold structure modelled on $V_s$, so that the tangent space at $\mu^*$ can be identified with $V_s$.

\begin{definition}[Fréchet differentiability in the strong-to-weak sense]
We say that $\mathcal{F}$ is \emph{Fréchet differentiable at}
$\mu^*\in P_s$ in the strong-to-weak sense if there exists a linear operator
$
\mathbf{D}_{\mu^*}\mathcal{F}:V_s\longrightarrow V_s,
$
bounded when $V_s$ is equipped with the strong norm and the codomain with the weak norm, such that
\[
\lim_{\substack{\nu\in V_s,\ \mu^*+\nu\in P_s\\
\|\nu\|_s\to0}}
\frac{
\left\|
\mathcal{F}(\mu^*+\nu)
-\mathcal{F}(\mu^*)
-\mathbf{D}_{\mu^*}\mathcal{F}(\nu)
\right\|_w
}{
\|\nu\|_s
}
=0.
\]
\end{definition}

A weaker notion is that of Gâteaux differentiability. Whereas Fréchet differentiability requires the linear approximation to hold uniformly with respect to the direction of perturbation, Gâteaux differentiability only requires the existence of directional derivatives.

\begin{definition}[Gâteaux differentiability]
We say that $\mathcal{F}$ is \emph{Gâteaux differentiable at}
$\mu^*\in P_s$ if there exists a bounded linear operator
$
D_{\mu^*}\mathcal{F}:V_s\longrightarrow V_s
$
such that, for every admissible direction $\nu\in V_s$,
\[
D_{\mu^*}\mathcal{F}(\nu)
=
\lim_{\varepsilon\to0}
\frac{
\mathcal{F}(\mu^*+\varepsilon\nu)-\mathcal{F}(\mu^*)
}{
\varepsilon
},
\]
where the limit is taken in a specified topology on the target space. The measure
$D_{\mu^*}\mathcal{F}(\nu)$ is the Gâteaux derivative of $\mathcal{F}$ at
$\mu^*$ in the direction $\nu$. If $\nu$ has density $\rho$, we use
$D_{\mu^*}\mathcal{F}(\rho)$ to denote the density of
$D_{\mu^*}\mathcal{F}(\nu)$. If $\mathcal{F}$ is Fréchet differentiable at
$\mu^*$, then it is also Gâteaux differentiable there, and
$
D_{\mu^*}\mathcal{F}
=
\mathbf{D}_{\mu^*}\mathcal{F}.
$
\end{definition}

\begin{remark}
Although $V_s$ is equipped with the $B_s$-norm, the definition of the
Gâteaux derivative need not require convergence in this norm. More generally,
the limit may be taken in any topology on the target space that implies weak
convergence of measures. This is sufficient for our purposes, since weak
convergence allows us to pass to the limit against admissible test functions.
In particular, if
\[
\frac{
\mathcal{F}(\mu^*+\varepsilon\nu)-\mathcal{F}(\mu^*)
}{
\varepsilon
}
\longrightarrow
D_{\mu^*}\mathcal{F}(\nu)
\]
in such a topology, then
\[
\lim_{\varepsilon\to0}
\int_{\mathbb{T}}\psi\,d\left(
\frac{
\mathcal{F}(\mu^*+\varepsilon\nu)-\mathcal{F}(\mu^*)
}{
\varepsilon
}
\right)
=
\int_{\mathbb{T}}\psi\,
d\left(D_{\mu^*}\mathcal{F}(\nu)\right)
\]
for every admissible test function $\psi$. This property will be used in
Lemma~\ref{relationtodifferential}.
\end{remark}

\subsubsection{Mean-Field Observables} We now introduce the class of observables for which our results hold and which generalises the empirical observables defined in \eqref{empirical}. For each $N\in\mathbb{N}$, let
$
\psi^{(N)}:\mathbb{T}^N\longrightarrow\mathbb{R}.
$

\begin{assumption}[Symmetry and regularity]
\label{ass:observable-symmetry}
For every $N\in\mathbb{N}$, the observable $\psi^{(N)}$ belongs to
$C^2(\mathbb{T}^N,\mathbb{R})$ and is permutation invariant; that is,
\[
\psi^{(N)}(x_1,\ldots,x_N)
=
\psi^{(N)}(x_{\sigma(1)},\ldots,x_{\sigma(N)})
\]
for every $(x_1,\ldots,x_N)\in\mathbb{T}^N$ and every permutation
$\sigma\in S_N$.
\end{assumption}

\begin{assumption}[Mean-field scaling]
\label{ass:observable-scaling}
There exists a constant $C>0$, independent of $N$ and of the coordinate
indices, such that
\[
\begin{aligned}
\|\partial_k\psi^{(N)}\|_\infty
&\leq \frac{C}{N},
&& k=1,\ldots,N,\\
\|\partial_{kk}\psi^{(N)}\|_\infty
&\leq \frac{C}{N},
&& k=1,\ldots,N,\\
\|\partial_\ell\partial_k\psi^{(N)}\|_\infty
&\leq \frac{C}{N^2},
&& \ell\neq k.
\end{aligned}
\]
Here $\|\cdot\|_\infty$ denotes the supremum norm.
\end{assumption}

Assumptions~\ref{ass:observable-symmetry} and
\ref{ass:observable-scaling} are natural for macroscopic observables.
Permutation invariance ensures that all particles contribute symmetrically,
while the scaling of the first derivatives expresses that the influence of
any single particle vanishes as $N\to\infty$. Similarly, mixed contributions
from two distinct particles are of order $N^{-2}$.

Given a probability measure $\mu$ on $\T$, define the effective one-coordinate
observable by averaging over all coordinates except the $k$-th:
\[
\psi_{k,\mu}^{(N)}(s)
=
\int_{\mathbb{T}^{N-1}}
\psi^{(N)}(x_1,\ldots,x_{k-1},s,x_{k+1},\ldots,x_N)
\prod_{j\neq k}d\mu(x_j).
\]
This quantity measures the dependence of the macroscopic observable on the
$k$-th coordinate when all remaining coordinates are distributed according
to $\mu$. By Assumption~\ref{ass:observable-symmetry},
$\psi_{k,\mu}^{(N)}$ is independent of the choice of $k$.

\begin{assumption}[Limiting one-coordinate fluctuation]
\label{ass:observable-limit}
For every probability measure $\mu$ on $\mathbb T$, there exists a function
$u:\mathbb{T}\to\mathbb{R}$ such that
\[
\sup_{s\in\mathbb{T}}
\left|
N\left(
\psi_{k,\mu}^{(N)}(s)
-
\int_{\mathbb{T}}\psi_{k,\mu}^{(N)}(r)\,d\mu(r)
\right)
-u(s)
\right|
\longrightarrow 0
\qquad\text{ as }N\to\infty.
\]
\end{assumption}

Assumption~\ref{ass:observable-limit} ensures that the leading-order
fluctuation generated by a single coordinate has a well-defined limit. If $\varphi\in C^2(\mathbb{T},\mathbb{R})$, the empirical average
observable defined in \eqref{empirical}
satisfies Assumptions~\ref{ass:observable-symmetry}--\ref{ass:observable-limit}.
Indeed, the first two assumptions are immediately verified, and for the last one 
\[
N\left(
\psi_{k,\mu}^{(N)}(s)
-
\int_{\mathbb{T}}\psi_{k,\mu}^{(N)}(r)\,d\mu(r)
\right)
=
\varphi(s)-\int_{\mathbb{T}}\varphi(r)\,d\mu(r),
\]
so that Assumption \ref{ass:observable-limit} is satisfied for
$
u(s)
=
\varphi(s)-\int_{\mathbb{T}}\varphi(r)\,d\mu(r).$

\subsubsection{Main Results} We now introduce the probabilistic setting and notation used throughout this section.
Let $\mu_0$ be a probability measure on $\mathbb{T}$ with continuously differentiable density
$\rho_0$. For each $N\in\mathbb{N}$, we consider the finite-dimensional probability
space
$
(\mathbb{T}^N,\mathcal{B}(\mathbb{T}^N),\mathbb{P}_N),
$
where
$
\mathbb{P}_N=\mu_0^{\otimes N}.
$
Let
\[
\mathbf{x}^{(N)}=(x_1,\ldots,x_N)\in\mathbb{T}^N
\]
be the random vector whose components are independent and identically distributed
with law $\mu_0$.
To study the limit as $N\to\infty$, we consider the infinite product probability
space
$
(\mathbb{T}^{\mathbb{N}},\mathcal{B}(\mathbb{T}^{\mathbb{N}}),\mathbb{P}),
$
where
$
\mathbb{P}=\mu_0^{\otimes\mathbb{N}}
$ and 
$
\mathbf{x}^{\mathbb{N}}=(x_1,x_2,\ldots)\in\mathbb{T}^{\mathbb{N}}
$
be distributed according to
$\mathbb{P}$.
For each $N\in\mathbb{N}$, let
\[
\mathbf{v}^{(N)}=(v_1,\ldots,v_N)\in\mathbb{R}^N
\]
be a sequence of vectors satisfying
$
\sup_k |v_k|<\infty,
$
and assume that the Cesàro mean of the sequence $(v_k)_{k\geq1}$ converges to
some $\bar v\in\mathbb{R}$, namely,
\[
\bar v
:=
\lim_{N\to\infty}
\frac1N\sum_{k=1}^N v_k .
\]
\begin{definition}
For $l=1,\ldots,t-1$, let $\mu_l$ denote the distribution of
$f_{\mu_0}^{[l]}(x_k)$, when $x_k$ is distributed as $\mu_0$, and where
\begin{equation}\label{Eq:f^[m]}
f_{\mu_0}^{[l]}:=f_{\mu_{l-1}}\circ\cdots\circ f_{\mu_0}.
\end{equation}
\end{definition}

\begin{theorem} {\label{Maintheorem1}}
Consider the globally coupled map defined in \eqref{globallycoupledmap},
 with $f \in C^2(\mathbb{T}, \mathbb{R})$ and $h \in C^2(\mathbb{T}^2, \mathbb{R})$  and let
$
\psi^{(N)}:\mathbb{T}^N\to\mathbb{R}
$ be an observable satisfying the assumptions \ref{ass:observable-symmetry}--\ref{ass:observable-limit}. Assume that the self-consistent transfer operator associated with the globally coupled map system is Gâteaux differentiable at $\mu_l$, for $l=0,\ldots,t-1$.
Then, for every fixed $t\in\mathbb{N}$, $\mathbb{P}$-almost surely, 
\begin{align*}
\lim_{N\to\infty}&
\chi_{\psi, N,t}(\mathbf{x}^{(N)}(\omega),\mathbf{v}^{(N)})
=\\ 
&\quad=\frac{1}{t}\log \Bigg(|\bar{v}|\lim_{N\to\infty}N \Big|\int_{\mathbb{T}} ( \psi_{k,t}^{(N)}(s)-\mathbb{E}_{\mu_t}({\psi_{k,t}^{(N)}(s)}))
D_{\mu_{t-1}}\mathcal{F} \circ \cdots \circ D_{\mu_0}\mathcal{F}(\rho_{0}'(s))ds \Big| \Bigg)
\end{align*} where $
\psi_{k,t}^{(N)}(s)
=
\int_{{\mathbb{T}^{N-1}}}
\psi(x_1,\ldots,x_{k-1},s,x_{k+1},\ldots,x_N)
\prod_{j\neq k} d\mu_t(x_j)$.
\end{theorem}

\begin{remark}
For the particular choice of the observable given by the empirical average as defined in \eqref{empirical}, the conclusion of Theorem \ref{Maintheorem1} reads
\[
\lim_{N\to\infty}
\chi_{\psi,N,t}(\mathbf{x}^{(N)}(\omega),\mathbf{v}^{(N)})
= \frac{1}{t}\log \Bigg(|\bar{v}|\int {\varphi(s)}
D \mathcal{F}_{\mu_{t-1}} \circ \cdots \circ D \mathcal{F}_{\mu_0} (\rho_0'(s))ds \Big| \Bigg).
\]
\end{remark}
Theorem~\ref{Maintheorem1} shows that, along directions $\mathbf v$ whose components have a convergent Cesàro mean, the mean-field observable Lyapunov exponents are determined by the linear cocycle generated by the derivatives \(D_{\mu_t}\mathcal F\) over the orbit \(\mu_t=\mathcal F^t\mu_0\) of the self-consistent transfer operator, acting on the derivative \(\rho_0'\) of the initial density. For \(\bar v\neq0\), the dependence on the perturbation direction enters only through the term \(t^{-1}\log|\bar v|\), which vanishes as \(t\to\infty\). Thus, after first taking \(N\to\infty\) and then \(t\to\infty\), each admissible sequence of mean-field observables has a single Lyapunov exponent, independent of $\mathbf v$. We expect this collapse to reflect the permutation symmetry of both the interactions and the observables; for non-symmetric networks, a richer spectrum of mean-field exponents should emerge, a question we leave for future work.

We now apply the preceding theorem to weakly coupled uniformly expanding maps and prove that their mean-field observable Lyapunov exponent is negative. This reflects the existence of an attracting stationary state in the thermodynamic limit $N\to\infty$. We first introduce
the functional-analytic framework in which the result is formulated. We
consider the nested normed spaces
\[
(C^2,\|\cdot\|_{C^2})
\subseteq
(C^1,\|\cdot\|_{C^1})
\subseteq
(C^0,\|\cdot\|_{C^0}).
\]
Here, $C^k$ denotes the class of measures that have densities in
$C^k(\mathbb T,\mathbb R)$\footnote{Throughout, $C^k(\mathbb T,\mathbb R)$ denotes the space of
functions from $\mathbb T$ to $\mathbb R$ having continuous derivatives up
to order $k$. We use $C^k$ to denote the class of measures on $\mathbb T$
that admit a density belonging to $C^k(\mathbb T,\mathbb R)$.}.
For a measure $\mu$ with density $g\in C^k$, we define
\[
\|\mu\|_{C^k}:=\|g\|_{C^k},
\qquad
\|g\|_{C^k}
=
\sum_{j=0}^{k}\sup_{x\in\mathbb T}|g^{(j)}(x)|.
\]
Moreover, we denote by $P(\mathbb T)$ the space of all probability
measures on $(\mathbb T,\mathcal B(\mathbb T))$. We then define
\[
P_k:=P(\mathbb T)\cap C^k,
\qquad k\in\{0,1\},
\]
and the zero-average spaces
\[
V_0:=\{\mu\in C^0:\mu(\mathbb T)=0\},
\qquad
V_1:=V_0\cap C^1.
\]
Let $f_\mu$ be defined as in \eqref{fmu}.

\begin{assumption}[Uniform Expansion and Regularity]
\label{ass:small-coupling}
Assume that $f\in C^4(\mathbb T,\mathbb T)$ is uniformly expanding, i.e.,
there exists $\lambda>1$ such that
\[
\inf_{x\in\mathbb T}|f'(x)|\geq\lambda,
\]
and that $h\in C^4(\mathbb T^2,\mathbb R)$.
\end{assumption}

Since
\[
f_\mu'(x)
=
f'(x)+\int_{\mathbb T}\partial_1h(x,y)\,d\mu(y),
\]
we have
\[
|f_\mu'(x)|
\geq
|f'(x)|-\|\partial_1h\|_\infty.
\]
Hence, by choosing $\|h\|_1$ sufficiently small, the family
$\{f_\mu\}_{\mu\in P_0}$ remains uniformly expanding, that is, there
exists $\lambda_0>1$ such that
\[
\inf_{\mu\in P_0}\inf_{x\in\mathbb T}
|f_\mu'(x)|
\geq\lambda_0.
\]
\begin{definition}[Local Basin of Attraction]
We say that $\mu^*$ has a \emph{local basin of attraction} if there exist
$\varepsilon>0$ and constants $C,\gamma>0$ such that
\[
U_\varepsilon(\mu^*)
:=
\left\{
\mu\in C^1:
\|\mu-\mu^*\|_{C^1}\leq\varepsilon
\right\}
\]
satisfies
\[
\|F^t(\mu)-\mu^*\|_{C^1}
\leq
Ce^{-\gamma t}\|\mu-\mu^*\|_{C^1},
\qquad
\mu\in U_\varepsilon(\mu^*),\quad t\in\mathbb N.
\]
In particular, every initial distribution in $U_\varepsilon(\mu^*)$
converges exponentially to $\mu^*$ under iteration of $F$.
\end{definition}
Under Assumption~\ref{ass:small-coupling}, and for sufficiently small $\delta$, there exists an attracting
fixed point $\mu^*\in C^3$ with a local basin of attraction, as established
in Lemma~\ref{unique attracting}.

\begin{theorem}
\label{Maintheorem2}
Suppose that assumptions \ref{ass:observable-symmetry}-\ref{ass:small-coupling} hold. There is $\delta>0$ sufficiently small such that if $\|h\|_{C^3}<\delta$, then there exists
$\mu^*\in C^3$ an attracting fixed point of the associated
self-consistent transfer operator. If the initial distribution
$\mu_0\in C^2$ lies in local basin of attraction of $\mu^*$, then
the upper observable Lyapunov exponent in the mean-field limit is negative;
more precisely,
\[
\limsup_{t\to\infty}
\left(
\lim_{N\to\infty}
\chi_{\psi,N,t}
\bigl(\mathbf{x}^{(N)}(\omega),\mathbf{v}^{(N)}\bigr)
\right)
<0,
\qquad\text{almost surely}.
\]
\end{theorem}
This theorem makes precise what we had already mentioned in the introduction, i.e. that microscopic chaos does not necessarily lead to macroscopic chaos.

 Section $\ref{Lyapunov under observable}$ below is independent of the rest of the paper and is devoted to examples 
in the finite-dimensional settings showing the relationship between Lyapunov spectrum and observed Lyapunov spectrum for certain observables and dynamical systems. This section may 
be skipped by readers interested only in the main results and their proofs, 
which are presented in Section \ref{Proof of the Main Theorems}.

\subsection{Lyapunov Spectrum vs Observable Lyapunov Spectrum}{\label{Lyapunov under observable}}
Recall the definitions of the Lyapunov exponent \eqref{Eq:StLyap}  and observable Lyapunov exponent \eqref{Lyaponuvunderobservable}. A natural question arises:
Under what conditions do we have 
\[
\lambda(\mathbf{x},\mathbf{v}) = \lambda_\psi(\mathbf{x},\mathbf{v}) \quad \text{for all } \mathbf{v} \in \mathbb{R}^d?
\]
This question was also considered in \cite{ott2003learning} in the setting of \emph{deterministic observables}, namely observables whose values evolve autonomously. Such observables are typically sufficiently high-dimensional to retain enough information about the underlying system for the observed dynamics to be governed by a closed deterministic rule. We instead are mostly interested in the case of observations with dimension much lower than that of the phase space -- as for mean-fields in GCMs. 

We begin with an example illustrating that the two can in fact be different, and we then show that under some mild conditions on the observable, the Lyapunov spectrum remains unchanged for the specific case of factor maps, i.e. whose Jacobian is diagonal. The systematic study of the relationship between Lyapunov exponents and observed Lyapunov exponents is beyond the scope of this paper and  will be the object  of future investigation by the authors.

\begin{example}
Consider the linear toral map
\[
F:\mathbb{T}^2\to\mathbb{T}^2,
\qquad
F(x_1,x_2)=(2x_1+x_2,\;3x_2)\pmod{1},
\]
where \(\mathbb{T}^2=\mathbb{R}^2/\mathbb{Z}^2\). The Jacobian is constant,
\[
DF=
\begin{pmatrix}
2&1\\
0&3
\end{pmatrix}.
\]
Hence the Lyapunov exponents of \(F\) are
$
\lambda_1=\log 3,$ and $\lambda_2=\log 2,$
with corresponding Oseledets subspaces
$
E_1=\operatorname{span}(1,1)^T,$ and
$E_2=\operatorname{span}(1,0)^T$.
Now consider the smooth observable
\[
\psi:\mathbb{T}^2\to\mathbb{R},
\qquad
\psi(x_1,x_2)=\cos\bigl(2\pi(x_2-x_1)\bigr).
\]
Its differential is
\[
D_{\mathbf{x}}\psi
=
2\pi\sin\bigl(2\pi(x_2-x_1)\bigr)
\left(
1,\,-1
\right).
\] 
For \(v\in E_1\), since \(E_1\) is invariant under \(DF\),
$D_{\mathbf{x}}F^tv=3^t v\in E_1.$
Consequently,
$
D_{F^t(\mathbf{x})}\psi\,D_{\mathbf{x}}F^tv=0
$
for every \(t\geq0\), and therefore
$
\lambda_\psi(\mathbf{x},\mathbf{v})=-\infty .
$
Thus, although the largest Lyapunov exponent of the toral map is
$
\lambda_1=\log3,
$
the corresponding observable Lyapunov exponent is
$
\lambda_\psi(x,v)=-\infty .
$
This shows that Lyapunov exponents obtained through an observable need not coincide with the usual Lyapunov exponents of the dynamical system, since the observable may annihilate an expanding Oseledets direction.
\end{example}

\noindent
We now investigate some sufficient conditions under which the Lyapunov spectrum of a map coincides with the Lyapunov spectrum under an observable. Throughout this section, we work under the following assumptions on the map \( F \) and the observable \( \psi \) that, roughly speaking, ensure that the kernel of the observable does not get close too frequently to the Oseledets directions.

\begin{assumption}\label{sameLyapunov}
\leavevmode
\begin{enumerate}
\item $F:\mathbb{T}^2\to\mathbb{T}^2$ is a $C^1(\mathbb{T}^2, \mathbb{T}^2)$ map admitting an invariant and ergodic probability measure \(\mu\) which is absolutely continuous with respect to Lebesgue measure, with bounded density.  
\item \(\psi:\mathbb{T}^2\to\mathbb{R}\) is a \(C^1(\mathbb{T}^2,\mathbb{R})\) function such that, for \(i=1,2\), there exist constants \(C_i>0\) and \(\alpha_i>0\) satisfying
\[
m\left(
\left\{\mathbf{x}\in\mathbb{T}^2:
|\partial_{x_i}\psi(\mathbf{x})|\le r
\right\}
\right)
\le C_i r^{\alpha_i},
\]
for all \(r>0\), where \(m\) denotes the Lebesgue measure on \(\mathbb{T}^2\).
\end{enumerate}
\end{assumption}

\begin{remark}
The sublevel-set condition in Assumption \ref{sameLyapunov} is not an artificial restriction. It is satisfied by a broad
class of observables. In particular, in  Appendix \ref{Appendix}. we show that this
assumption holds for several natural classes, including real analytic
observables with non-identically vanishing partial derivatives.
\end{remark}

\begin{proposition}[Lyapunov Spectrum Invariance]{\label{same spectrum}}

Suppose that \(F:\mathbb{T}^2\to\mathbb{T}^2\) and 
\(\psi:\mathbb{T}^2\to\mathbb{R}\) satisfy Assumptions $\ref{sameLyapunov}$.
Assume additionally that the Jacobian of \(F\) is diagonal for every 
\(\mathbf{x}=(x_1,x_2)\in\mathbb{T}^2\), i.e.,
\[
D_{\mathbf{x}}F=
\begin{pmatrix}
\partial_{x_1}F_1(\mathbf{x})&0\\
0&\partial_{x_2}F_2(\mathbf{x})
\end{pmatrix}.
\]
If the Lyapunov spectrum of \(F\) is simple, then for 
\(\mu\)-almost every \(\mathbf{x}\in\mathbb{T}^2\) and every non-zero tangent vector
\(\mathbf{v}\in\mathbb{R}^2\),
\[
\lambda(\mathbf{x},\mathbf{v})
=
\lambda_\psi(\mathbf{x},\mathbf{v}).
\]
\end{proposition}
We prove Proposition \ref{same spectrum} through a sequence of lemmas, and at the end, we combine these results to complete the proof of Proposition \ref{same spectrum}. Before the proof, we give a geometrical interpretation.
\paragraph{Geometric interpretation.}
In the diagonal case, the $i$-th Oseledets subspace is the coordinate
direction
$
E_i=\operatorname{span}\{e_i\}, i=1,2,
$
where $e_1=(1,0)$ and $e_2=(0,1)$ are the standard basis vectors of
$\mathbb{R}^2$, and
\[
\partial_{x_i}\psi(F^t(x))
=
\nabla\psi(F^t(x))\cdot e_i
\]
measures how strongly the observable detects this direction. If this quantity is zero, then $e_i$ lies in the kernel of the gradient, so the observable loses this direction. More generally, if it is very small, the observable is nearly insensitive to this direction. Thus, we want the Oseledets direction neither to lie in the kernel nor to approach the kernel exponentially fast along a typical orbit. This is expressed by
\[
\lim_{t\to\infty}
\frac{1}{t}\log
\left|\partial_{x_i}\psi(F^t(x))\right|
=0.
\]
The sublevel-set condition
\[
m\left(\left\{x:
|\partial_{x_i}\psi(x)|\le r
\right\}\right)
\le C_i r^{\alpha_i}
\]
provides a quantitative way to ensure this: it allows $\partial_{x_i}\psi$ to become arbitrarily small, but requires the measure of the set of points where it is very small to decrease sufficiently fast as $r \to 0$. A weaker condition is
\[
\log|\partial_{x_i}\psi|\in L^1(\mu),
\]
which applies to a larger class of observables and still gives subexponential decay along typical orbits. The sublevel-set condition is stronger but provides additional quantitative control.
\begin{lemma}\label{expdecay}
Let \(c<0\) and define
\[
\Omega_{t,c}^i:=
\left\{
\mathbf{x}\in\mathbb{T}^2:
|\partial_{x_i}\psi(\mathbf{x})|\le e^{ct}
\right\},
\]
for $i=1,2$. Then, under Assumption \ref{sameLyapunov}, for every probability measure \(\nu\) with bounded density with respect to \(m\), there exist constants \(C'>0\) and \(\gamma>0\) such that
$
\nu(\Omega_{t,c}^i)\le C'e^{-\gamma t}.
$
\end{lemma}

\begin{proof}
Since \(\nu\) has bounded density with respect to the Lebesgue measure,
there exists \(M>0\) such that
\[
d\nu=\rho(\mathbf{x})dm(\mathbf{x}),
\qquad
\rho(\mathbf{x})\le M  \qquad m\mbox{-a.e.}
\]
Therefore,
\[
\nu(\Omega_{t,c}^i)
=
\int_{\Omega_{t,c}}\rho(\mathbf{x})\,dm(\mathbf{x})
\le
M m(\Omega_{t,c}^i).
\]
By the definition of \(\Omega_{t,c}^i\) and the sublevel estimate with
\(r=e^{ct}\), we have
$
m(\Omega_{t,c}^i)
\le
C_i(e^{ct})^{\alpha_i}
=
C_ie^{\alpha_i ct}.
$
Hence,
\[
\nu(\Omega_{t,c}^i)
\le
MC_ie^{\alpha_i ct}.
\]
Since \(c<0\), setting
$
\gamma=-\alpha_i c>0
$
gives
$
\nu(\Omega_{t,c}^i)
\le
MC_ie^{-\gamma t}.
$
Thus, \(\nu(\Omega_{t,c}^i)\) decays exponentially as \(t\to\infty\).
\end{proof}

\begin{lemma}\label{lem:Yk-zero}
Fix $c < 0$ and define
\[
Y_c^i := \Bigl\{ \mathbf{x} \in \mathbb{T}^2 : \exists t_k \nearrow \infty \text{ such that }
\frac{1}{t_k} \log \bigl| \partial_{x_i}\psi(F^{t_k}(\mathbf{x})) \bigr| \le c \Bigr\}.
\]
Under Assumption \ref{sameLyapunov}, 
\(
\mu(Y_c^i) = 0.
\)
\end{lemma}
\begin{proof}
Let $\Gamma_t^i$ be defined by
\[
\Gamma_t^i := F^{-t}(\Omega_{t,c}^i)=\{ \mathbf{x} \in\mathbb{T}^2:|\partial_{x_i} \psi(F^{t}(\mathbf{x}))| \le e^{c t} \}.
\]
Using $F$-invariance of $\mu$, we have $\mu(\Gamma_t^i)=\mu(\Omega_{t,c}^i)$. By Lemma \ref{expdecay}, the measure of $\Gamma_t^i$ decays exponentially; the series
$\sum_{t=1}^\infty \mu(\Gamma_t^i)$ converges.
By the first Borel--Cantelli lemma,
\[\mu (\limsup_{t\to\infty}\Gamma_t^i) =0.\]
Moreover,
$
Y_c^i=\limsup_{t\to\infty}\Gamma_t^i,
$
because membership in infinitely many $\Gamma_t^i$ is equivalent to the
existence of a subsequence $(t_k)$ satisfying
$
\frac{1}{t_k}\log|\partial_{x_i}\psi(F^{t_k}(\mathbf{x}))|\le c .
$
Therefore,
$
\mu(Y_c^i)=0.
$
\end{proof}

\begin{lemma}{\label{limitlog0}}
Under the standing assumptions of Proposition \ref{same spectrum}, we have
\[
\lim_{t \to \infty} \frac{1}{t} \log \bigl| \partial_{x_i} \psi(F^t(\mathbf{x}))\bigr| = 0
\quad \text{for $\mu$-almost every } \mathbf{x} \in \mathbb{T}^2 \text{and } i=1,2.
\]
\end{lemma}
\begin{proof}
From Lemma~\ref{lem:Yk-zero}, for every $c<0$, there exists a set $E_c^i \subset \mathbb{T}^2$ with $\mu(E_c^i)=1$ such that for every $\mathbf{x}\in E_c^i$,
\[
\liminf_{t\to\infty}\frac{1}{t}\log \left|\partial_{x_i}\psi(F^t(\mathbf{x}))\right|\ge c.
\]
Now set $c=-\frac{1}{m}$, $m\in\mathbb{N}$, and define
\[
E^i := \bigcap_{m=1}^\infty E^i_{-1/m}.
\]
Then $\mu(E^i)=1$, and for every $\mathbf{x}\in E^i$,
\[
\liminf_{t\to\infty}\frac{1}{t}\log \left|\partial_{x_i}\psi(F^t(\mathbf{x}))\right|\ge -\frac{1}{m}
\quad \forall m\in\mathbb{N}.
\]
Hence,
\[
\liminf_{t\to\infty}\frac{1}{t}\log \left|\partial_{x_i}\psi(F^t(\mathbf{x}))\right|\ge 0
\quad \text{for almost every } \mathbf{x}\in \mathbb{T}^2.
\]
Since $\partial_{x_i} \psi$ is bounded, it follows that 
\[
\limsup_{t \to \infty} \frac{1}{t} \log |\partial_{x_i} \psi(F^{t} (\mathbf{x}))|  \le 0,
\] 
Combining the lower and upper bounds yields
\[
\lim_{t \to \infty} \frac{1}{t} \log|\partial_{x_i} \psi(F^{t} (\mathbf{x}))|  = 0.
\] 
\end{proof}
\noindent By taking $E=E^{1}\cap E^{2}$, we obtain a full-measure set on which the conclusion holds for both i=1,2.

\begin{proof}[\textbf{Proof of Proposition \ref{same spectrum}}] Let $\mathbf{v} = (v_1, v_2)$ with $v_1 \neq 0$. By definition, the Lyapunov exponent,
\[
\lambda(\mathbf{x},\mathbf{v}) = \lim_{t \to \infty} \frac{1}{t} \log \| D_{\mathbf{x}}F^t (\mathbf{v}) \|.
\]
If $F$ has a diagonal Jacobian, then
\[
D_{\mathbf{x}}F^t(\mathbf{v}) =
\begin{pmatrix}
\partial_{x_1} F_{1}^{t} v_1 \\[1mm]
\partial_{x_2} F_{2}^{t} v_2
\end{pmatrix},
~~\text{where}~~ \partial_{x_1}F_1^t
=
\frac{\partial F_1^t}{\partial x_1}(\mathbf{x}),
~~~\partial_{x_2}F_2^t
=
\frac{\partial F_2^t}{\partial{x_2}}(\mathbf{x}). \]
and
\[
\| D_{\mathbf{x}}F^t(\mathbf{v})\| = \sqrt{|\partial_{x_1}F_1^t v_1|^2 + |\partial_{x_2}F_2^t v_2|^2}.
\]
With out loss of generality we take $\lambda_1 > \lambda_2$ and $v_1 \neq 0$, the first term dominates and we obtain
\[
\lambda(\mathbf{x},\mathbf{v}) = \lim_{t \to \infty} \frac{1}{t} \log |\partial_{x_1}F_1^t(\mathbf{x}) v_1| = \lambda_1.
\]
Consider the observable
$\psi$
and compute the derivative along $\mathbf{v}$:
\[
D_{\mathbf{x}}(\psi \circ F^t) (\mathbf{v}) = \partial_{x_1} \psi(F^t(\mathbf{x}))\partial_{x_1}F_1^t(\mathbf{x})v_1+ \partial_{x_2}\psi(F^t(\mathbf{x}))  \partial_{x_2}F_2^t(\mathbf{x}) v_2.
\]
The observable Lyapunov exponent along $\mathbf{v}$ is
\[
\lambda_\psi(\mathbf{x},\mathbf{v}) = \lim_{t \to \infty} \frac{1}{t} \log \left|\partial_{x_1} \psi(F^t)\partial_{x_1}F_1^tv_1+ \partial_{x_2}\psi(F^t)  \partial_{x_2}F_2^t v_2\right|.
\]
Factor $\partial_{x_1}F_1^t(\mathbf x)v_1$ (this can be done for a sufficiently large \(t\) since $\lambda_1>-\infty$):
\[
\lambda_\psi(\mathbf{x},\mathbf{v}) = \lim_{t \to \infty} \frac{1}{t} \log \Big| \partial_{x_1}F_1^t v_1 \Big( \partial_{x_1} \psi(F^t) + \partial_{x_2} \psi(F^t) \frac{\partial_{x_2}F_2^tv_2}{\partial_{x_1} F_1^t v_1} \Big) \Big|.
\]
\[
\lambda_\psi(\mathbf{x},\mathbf{v}) = \lim_{t \to \infty} \frac{1}{t} \log |\partial_{x_1}F_1^t v_1| + 
\lim_{t \to \infty} \frac{1}{t} \log \left|  \partial_{x_1} \psi(F^t) +  \partial_{x_2}\psi(F^t) \frac{\partial_{x_2}F_2^tv_2}{\partial_{x_1} F_1^t v_1} \right|.
\]
The first term clearly gives
\[ \lim_{t \to \infty} \frac{1}{t} \log |\partial_{x_1}F_1^t(\mathbf{x})v_1| = \lambda_1.
\]
Since, by Lemma \ref{limitlog0} applied to both partial derivatives of
\(\psi\), for \(\mu\)-almost every \(\mathbf{x}\),
\[
\lim_{t\to\infty}
\frac1t
\log|\partial_{x_i}\psi(F^t(\mathbf{x}))|
=0,
\qquad i=1,2,
\]
we have
\[
\lim_{t\to\infty}
\frac1t
\log
\left|
\frac{\partial_{x_2}\psi(F^t)}
{\partial_{x_1}\psi(F^t)}
\right|
=
0.
\]
Moreover, since the spectrum is simple and
\(\lambda_1>\lambda_2\),  provided $v_2\neq 0$ (the case $v_2=0$ can be easily treated separately), we have
\[
\lim_{t\to\infty}
\frac1t
\log
\left|
\frac{\partial_{x_2}F_2^t v_2}
{\partial_{x_1}F_1^t v_1}
\right|
=
\lambda_2-\lambda_1<0 .
\]
Therefore,
\[
\lim_{t\to\infty}
\frac1t
\log
\left|
\frac{
\partial_{x_2}\psi(F^t)
\frac{\partial_{x_2}F_2^t v_2}
{\partial_{x_1}F_1^t v_1}
}
{\partial_{x_1}\psi(F^t)}
\right|
=
\lim_{t\to\infty}
\frac1t
\log
\left|
\frac{\partial_{x_2}\psi(F^t)}
{\partial_{x_1}\psi(F^t)}
\right|
+
\lim_{t\to\infty}
\frac1t
\log
\left|
\frac{\partial_{x_2}F_2^t v_2}
{\partial_{x_1}F_1^t v_1}
\right|
=
0+\lambda_2-\lambda_1<0.
\]
Hence,
\[
\frac{
\partial_{x_2}\psi(F^t)
\frac{\partial_{x_2}F_2^t v_2}
{\partial_{x_1}F_1^t v_1}
}
{\partial_{x_1}\psi(F^t)}
\longrightarrow0
\]
exponentially fast for \(\mu\)-almost every \(\mathbf{x}\). Consequently, writing 
\[
\partial_{x_1}\psi(F^t)
+
\partial_{x_2}\psi(F^t)
\frac{\partial_{x_2}F_2^t v_2}
{\partial_{x_1}F_1^t v_1}
=
\partial_{x_1}\psi(F^t)
\left(
1+
\frac{
\partial_{x_2}\psi(F^t)
\frac{\partial_{x_2}F_2^t v_2}
{\partial_{x_1}F_1^t v_1}
}
{\partial_{x_1}\psi(F^t)}
\right),
\]
the second factor converges to \(1\), and therefore
\[
\lim_{t\to\infty}
\frac1t
\log
\left|
\partial_{x_1}\psi(F^t)
+
\partial_{x_2}\psi(F^t)
\frac{\partial_{x_2}F_2^t v_2}
{\partial_{x_1}F_1^t v_1}
\right|
=
0.
\]
Combining all of the above
\[
\lim_{t\to\infty}
\frac1t
\log
|\partial_{x_1}F_1^t(\mathbf{x})v_1|
=
\lambda_1,
\]
and we obtain
\[
\lambda_\psi(\mathbf{x},\mathbf{v})
=
\lambda_1
=
\lambda(\mathbf{x},\mathbf{v}).
\]
If $v_1 = 0$ then,
\[
\lambda_\psi (\mathbf{x},\mathbf{v}) = \lim_{t\to\infty} \frac{1}{t} \log \big| \partial_{x_2} \psi(F^t) \, \partial_{x_2}F_2^tv_2 \big|
= \lambda_2=\lambda(\mathbf{x},\mathbf{v}).
\]
\end{proof}

\begin{remark}
The arguments in Proposition $\ref{same spectrum}$ extend to any finite-dimensional setting.
More precisely, let $F:\mathbb{T}^d\to\mathbb{T}^d$ and
$\psi:\mathbb{T}^d\to\mathbb{R}$ satisfy the  assumptions analogous to Assumption \ref{sameLyapunov}. Assume additionally that the Jacobian of $F$ is
diagonal and has a simple Lyapunov spectrum. Moreover, assume that the
sublevel estimate holds for every partial derivative
$\partial_{x_i}\psi$, $i=1,\ldots,d$. Then, for every non-zero tangent vector
the statements are valid with the same arguments.
\end{remark}

Now consider a general differentiable map \(F: M\rightarrow M\) on a Riemannian manifold \(M\) with a simple invariant Oseledets' splitting $TM=\bigoplus_{i=1}^d E_i$. One can observe that by using the chain rule,
\[
D_{\mathbf{x}}(\psi\circ F^t)(\mathbf{v})=\nabla\psi({F^t (\mathbf{x})}) D_{\mathbf{x}}F^t(\mathbf{v}),
\]
Here $\nabla\psi({F^t (\mathbf{x})}) \in (\mathbb{R}^d)^*$ is a covector (row vector) and $D_{\mathbf{x}}F^t(\mathbf{v}) \in \mathbb{R}^d$ is a vector, so the expression is the natural dual pairing. Let \(\mathbf{v_i}(\mathbf{x})\in E_i(\mathbf{x})\) be a unit Lyapunov vector corresponding to the Lyapunov exponent \(\lambda_i(\mathbf{x})\). Since \(E_i\) is one-dimensional,
\[
D_{\mathbf{x}}F^t(\mathbf{v_i(\mathbf{x})})
=
\|D_{\mathbf{x}}F^t(\mathbf{v_i(\mathbf{x}))}\|\mathbf{v_i}(F^t(\mathbf{x})),
\]
up to a sign, which disappears after taking absolute values.
Therefore,
\[
D_{\mathbf{x}}(\psi\circ F^t)(\mathbf{v}_i(\mathbf{x}))
=
\|D_{\mathbf{x}}F^t(v_i(x))\|
\,\nabla\psi({F^t (\mathbf{x})})
\mathbf{v_i}(F^t(\mathbf{x})).
\]
Hence
\[
\lambda_\psi(\mathbf{x},\mathbf{v_i(\mathbf{x})}) = 
\lim_{t\to\infty}
\frac1t
\log
\|D_{\mathbf{x}}F^t(\mathbf{v_i}(\mathbf{x}))\|
+
\lim_{t\to\infty}
\frac1t
\log
\left|
\nabla\psi(F^t(\mathbf{x}))
\mathbf{v_i}(F^t(\mathbf{x}))
\right|.\]
By the definition of the Lyapunov exponent,
\[
\lim_{t\to\infty}
\frac1t
\log
\|D_{\mathbf{x}}F^t(\mathbf{v_i}(\mathbf{x}))\|
=
\lambda_i(\mathbf{x}).
\]
Thus
\[
\lambda_\psi(\mathbf{x},\mathbf{v_i(\mathbf{x})})
=
\lambda_i(\mathbf{x})
+
\lim_{t\to\infty}
\frac1t
\log
\left|
\nabla\psi(F^t(\mathbf{x}))
\mathbf{v_i}(F^t(\mathbf{x}))
\right|.
\]
Defining
\[
h_i(\mathbf{x})
=
\left|
\nabla\psi(\mathbf{x}) \mathbf{v_i}(\mathbf{x})
\right|,
\]
we obtain
\[
\lambda_\psi(\mathbf{x},\mathbf{v_i(\mathbf{x})})
=
\lambda_i(\mathbf{x})
+
\lim_{t\to\infty}
\frac1t
\log h_i(F^t(\mathbf{x})).
\]
So, $\lambda_\psi(\mathbf{x},\mathbf{v_i(\mathbf{x})})=\lambda(\mathbf{x},\mathbf{v_i(\mathbf{x})})$ whenever
\[
\lim_{t\to\infty}
\frac1t
\log h_i(F^t(\mathbf{x}))
=
0.
\]

\begin{remark}
    Let $(M,\mathcal{B})$ be a measurable space, and let $F: M\to M$ be a measurable map admitting an invariant probability measure $\mu$. Let $h_i(\mathbf{x})=
\left|
\nabla\psi(\mathbf{x})\mathbf{v_i}(\mathbf{x})
\right|$ as above such that $h_i
:M\to (0,\infty)$ be a positive measurable function and $\log h_i \in L^1(M,\mu)$. Then 
\[
\lim_{t\to\infty}\frac{1}{t}\log h_i(F^t(x))=0
\quad \text{for } \mu\text{-almost every } x\in M.
\]
The above result follows from a standard application of the Borel--Cantelli
lemma, and the proof can be found, for example, in Viana's notes \cite[Lemma 3.7]{viana2014lectures}.
Verifying the condition $\log h_i \in L^1(M,\mu)$ is in general non-trivial. The map $x \mapsto \mathbf{v_i}(x)$ is only a measurable selection of Oseledets subspaces, defined $\mu$-a.e., and typically varies on a measurable Grassmannian bundle. Hence, quantities of the form
\[
h_i(x)=|\nabla\psi(x)\mathbf{v_i}(x)|
\]
inherit only measurable regularity even when $\psi$ is smooth. 
\end{remark}
\noindent While we do not address this question here, it would be interesting to determine for which classes of observables and dynamical systems the condition $(\log h_i \in L^1(M,\mu))$ holds. Determining 
natural sufficient conditions 
for log $h_i \in L^1(M,\mu)$ 
remains open in this setting. \newpage

\section{Proof of the main Theorems}{\label{Proof of the Main Theorems}}
\subsection{Proof of Theorem \ref{Maintheorem1}}
To simplify the notation, we omit the superscript \((N)\) throughout and write
\[
F_k^{(N)}=F_k,\qquad
\psi^{(N)}=\psi,\qquad
\mathbf{v}^{(N)}=\mathbf{v}~~\text{and}~~\mathbf{x}^{(N)}=\mathbf{x}.
\]
By the chain rule, the quantity $\chi_{\psi,N,t}(\mathbf{x}(\omega), \mathbf{v})$ can be written as

\begin{equation}
\chi_{\psi,N,t}(\mathbf{x}(\omega),\mathbf{v})
=
\frac{1}{t}
\log\left|
\nabla\psi\big(F^t(\mathbf{x}(\omega))\big)D_{\mathbf{x}(\omega)}F^t(\mathbf{v})
\right|,
\qquad
\mathbf{x}=(x_1,\ldots,x_N)\sim\mu_0^{\otimes N}.
\end{equation}
where
\[
F^t(\mathbf{x}(\omega))
=
\bigl(
F_1^t(\mathbf{x}(\omega)),\ldots,
F_N^t(\mathbf{x}(\omega))
\bigr),
\]
with
\[
F_k^t(\mathbf{x}(\omega))
=
f\bigl(F_k^{t-1}(\mathbf{x}(\omega))\bigr)
+
\frac{1}{N}\sum_{i\neq k}
h\bigl(
F_k^{t-1}(\mathbf{x}(\omega)),
F_i^{t-1}(\mathbf{x}(\omega))
\bigr),
\]
and $\nabla\psi(F^t(\mathbf{x}(\omega)))$ denotes the gradient of
$\psi$ evaluated at $F^t(\mathbf{x}(\omega))$.
The Jacobian matrix is
$$
D_{\mathbf{x}} F^t = [ {\partial F^t_k}/{\partial x_j} ]_{1 \le k,j \le N}~~
\text{and}~~
\nabla \psi(F^t(\mathbf{x}(\omega))) = \left( \partial_1 \psi(F^t(\mathbf{x}(\omega))), \dots, \partial_N \psi(F^t(\mathbf{x}(\omega))) \right).
$$
Below we will use the shorthand notation $\partial_j F_k^t := \partial F_k^t / \partial x_j$,
for $\mathbf{v} = (v_1,\dots,v_N) \in \mathbb{R}^{N}$, we get
\begin{equation*}
\chi_{\psi,N,t}(\mathbf{x}(\omega),\mathbf{v})
=
\frac{1}{t}
\log \Bigg|
\sum_{k=1}^N
\partial_k\psi\big(
F_1^t(\mathbf{x}(\omega)),\ldots,F_N^t(\mathbf{x}(\omega))
\big)
\Bigg(
\partial_kF_k^t(\mathbf{x}(\omega))v_k
+
\sum_{j\ne k}
\partial_jF_k^t(\mathbf{x}(\omega))v_j
\Bigg)
\Bigg|.
\end{equation*}

\begin{equation}\label{termsequation}
=
\frac{1}{t}\log\Bigg|
\underbrace{
\sum_{k=1}^N
\partial_k\psi\big(
F_1^t(\mathbf{x}(\omega)),\ldots,F_N^t(\mathbf{x}(\omega))
\big)
\partial_kF_k^t(\mathbf{x}(\omega))v_k
}_{\text{Term I}}
+
\underbrace{
\sum_{k=1}^N
\Bigg(
\partial_k\psi\big(
F_1^t(\mathbf{x}(\omega)),\ldots,F_N^t(\mathbf{x}(\omega))
\big)
\sum_{j\ne k}
\partial_jF_k^t(\mathbf{x}(\omega))v_j
\Bigg)
}_{\text{Term II}}
\Bigg|.
\end{equation}
Define
\begin{align*}
&G_{1,N}(\mathbf{x}(\omega))
=
\sum_{k=1}^N
\partial_k \psi\bigl(
F_1^{t}(\mathbf{x}(\omega)), \ldots,
F_N^{t}(\mathbf{x}(\omega))
\bigr)
\partial_k F_k^{t}(\mathbf{x}(\omega))v_k,\\&
G_{2,N}(\mathbf{x}(\omega))
=
\sum_{k=1}^N
\partial_k \psi\bigl(
F_1^{t}(\mathbf{x}(\omega)), \ldots,
F_N^{t}(\mathbf{x}(\omega))
\bigr)
\sum_{j\ne k}
\partial_j F_k^{t}(\mathbf{x}(\omega))\,v_j,\\&
H_{1,N}(\mathbf{x}(\omega))
=
\sum_{k=1}^{N}
\partial_k \psi\left(
f_{\mu_0}^{[t]}(x_1(\omega)), \ldots,
f_{\mu_0}^{[t]}(x_N(\omega))
\right)
(f_{\mu_0}^{[t]})'(x_k(\omega))v_k,\\&
H_{2,N}(\mathbf{x}(\omega))
=
\bar{v}\sum_{k=1}^{N}
\partial_k \psi\left(
f_{\mu_0}^{[t]}(x_1(\omega)), \ldots,
f_{\mu_0}^{[t]}(x_N(\omega))
\right)
g_t(x_k(\omega)).
\end{align*}
where, for fixed $s \in \mathbb{T}$, we define
\begin{align*}
g_t(s)
&= \lim_{N\rightarrow\infty}
\int_{\mathbb{T}^{N-1}}
\sum_{j\ne k}
\partial_j F_k^t
(x_1,\ldots,x_{k-1},s,x_{k+1},\ldots,x_N)
\prod_{j\ne k} d\mu_0(x_j)\\
&=
\lim_{N\to\infty}
\mathbb{E}
\left[
\sum_{j\neq k}\partial_j F_k^t(x_1,...,x_N)
\,\middle|\, x_k=s
\right].
\end{align*}
The expectation is with respect to the reference measure \(\mathbb P=\mu_0^{\mathbb N}\) and the limit is assumed to exist uniformly in $s$. By the exchangeability of the globally coupled system, this quantity is independent of the choice of index $k$, and defines a function $g_t:\mathbb{T}\to\mathbb{R}$.
Our main goal is to find
\begin{equation}
\lim_{N\to\infty}\frac{1}{t}\log
\big|(G_{1,N}(\mathbf{x}(\omega))+G_{2,N}(\mathbf{x}(\omega)))\big|.
\end{equation}

\noindent The proof of the main theorem is carried out through a sequence of
auxiliary lemmas. Lemmas \ref{epsilonnet}--\ref{g_t(x_k)} are used to
prove that, almost surely,
\[
\lim_{N\to\infty}
\left|
(G_{1,N}(\mathbf{x}(\omega))+G_{2,N}(\mathbf{x}(\omega)))
-
(H_{1,N}(\mathbf{x}(\omega))+H_{2,N}(\mathbf{x}(\omega)))
\right|
=0.
\]
In Lemma \ref{h1h2}, we prove that, almost surely,
\[
\lim_{N\to\infty}
\left|
(H_{1,N}(\mathbf{x}(\omega))+H_{2,N}(\mathbf{x}(\omega)))
-
\mathbb{E}\big[
H_{1,N}(\mathbf{x}(\omega))+H_{2,N}(\mathbf{x}(\omega))
\big]
\right|
=0,~~ \text{where}
\]
\[
\begin{aligned}
\mathbb{E}\!\left[
H_{1,N}(\mathbf{x}(\omega))
+
H_{2,N}(\mathbf{x}(\omega))
\right]
&=\int_{\mathbb{T}^{N}}
\left[
H_{1,N}(x_1,\ldots,x_N)
+
H_{2,N}(x_1,\ldots,x_N)
\right]
\,d\mu_0(x_1)\cdots d\mu_0(x_N).
\end{aligned}
\]
Here, $x_1,\ldots,x_N$ are dummy integration variables; in particular,
they denote the coordinates of the point of $\mathbb{T}^N$ being
integrated over, rather than the random variables
$x_1(\omega),\ldots,x_N(\omega)$.
Lemma \ref{exph1h2} computes the expectation
\[
\mathbb{E}\big[
H_{1,N}(\mathbf{x}(\omega))+H_{2,N}(\mathbf{x}(\omega))
\big].
\]
Finally, in Lemma \ref{relationtodifferential}, we relate this
expectation to the differential of the self-consistent transfer
operator. The proof of the theorem is then obtained by combining the
conclusions of these lemmas.

\begin{lemma}\label{epsilonnet}
There exists a set $\Omega_0$ of probability one such that, for every
$\omega\in\Omega_0$,
\[
\max_{1\le k\le N}
\sup_{s\in\mathbb T}
\left|
\frac1N\sum_{j\neq k}h(s,x_j(\omega))
-
\frac{N-1}{N}\int_{\mathbb T}h(s,y)\,d\mu_0(y)
\right|
\longrightarrow0
\]
as $N\to\infty$.
Moreover,
\[
\sup_{1\leq k\leq N}
\left|
F_k(\mathbf{x}(\omega))-f_{\mu_0}(x_k(\omega))
\right|
\longrightarrow 0
\qquad\text{as }N\to\infty .
\]
\end{lemma}

\begin{proof}
Fix $k\in\mathbb N$ and $s\in\mathbb T$. Define $Y_j(\omega):=h(s,x_j(\omega)),
 j\neq k.$
Then $\{Y_j\}_{j\neq k}$ are bounded i.i.d. random variables with
\[
\mathbb E[Y_j]
=
\int_{\mathbb T}h(s,y)\,d\mu_0(y).
\]
Hence, by Hoeffding's inequality, for every $\varepsilon>0$,
\[
\mathbb P\!\left(
\left|
\frac1N\sum_{j\neq k}Y_j
-
\frac{N-1}{N}\mathbb E[Y_j]
\right|
>\varepsilon
\right)
\le
2e^{-c\varepsilon^2N},
\]
for some constant $c>0$ independent of $N$, $k$, and $s$ -- from the regularity of $h$. To obtain uniformity in $s$, cover $\mathbb T$ by
$M=\lceil\varepsilon^{-1}\rceil$ arcs of length at most $\varepsilon$,
and let $s_1,\dots,s_M$ be representative points from each arc. For every $m\in\{1,...,M\}$, define
\[
A_m^k
=
\left\{
\omega:
\left|
\frac1N\sum_{j\neq k}h(s_m,x_j(\omega))
-
\frac{N-1}{N}
\int_{\mathbb T}h(s_m,y)\,d\mu_0(y)
\right|
>
\varepsilon
\right\}.
\]
By the union bound,
\[
\mathbb P\!\left(
\bigcup_{m=1}^M A_m^k
\right)
\le
2Me^{-c\varepsilon^2N}.
\]
Since $h$ is Lipschitz in its first variable, there exists a constant
$K>0$ such that, whenever $|s-s_m|\le\varepsilon$,
\[
\left|
\frac1N\sum_{j\neq k}h(s,x_j(\omega))
-
\frac1N\sum_{j\neq k}h(s_m,x_j(\omega))
\right|
\le
K\varepsilon,
\]
and
\[
\left|
\int_{\mathbb T}h(s,y)\,d\mu_0(y)
-
\int_{\mathbb T}h(s_m,y)\,d\mu_0(y)
\right|
\le
K\varepsilon.
\]
Therefore,
\[
\sup_{s\in\mathbb T}
\left|
\frac1N\sum_{j\neq k}h(s,x_j(\omega))
-
\frac{N-1}{N}
\int_{\mathbb T}h(s,y)\,d\mu_0(y)
\right|
\le
(1+2K)\varepsilon
\]
whenever $\omega\notin\bigcup_{m=1}^MA_m^k$.
Hence,
\[
\mathbb P\!\left(
\sup_{s\in\mathbb T}
\left|
\frac1N\sum_{j\neq k}h(s,x_j(\omega))
-
\frac{N-1}{N}
\int_{\mathbb T}h(s,y)\,d\mu_0(y)
\right|
>
(1+2K)\varepsilon
\right)
\le
2Me^{-c\varepsilon^2N}.
\]
Taking the union over $k=1,\dots,N$ gives
\[
\mathbb P\!\left(
\max_{1\le k\le N}
\sup_{s\in\mathbb T}
\left|
\frac1N\sum_{j\neq k}h(s,x_j(\omega))
-
\frac{N-1}{N}
\int_{\mathbb T}h(s,y)\,d\mu_0(y)
\right|
>
(1+2K)\varepsilon
\right)
\le
2NMe^{-c\varepsilon^2N}.
\]
Since
\[
\sum_{N=1}^{\infty}
2NMe^{-c\varepsilon^2N}
<
\infty,
\]
the Borel--Cantelli lemma implies that, for every $\varepsilon>0$, there exists a  set $\Omega_\varepsilon$ of dull measure such that, for every $\omega\in\Omega_\varepsilon$, there exists $N_0(\omega,\varepsilon)$ satisfying
\[
\max_{1\le k\le N}
\sup_{s\in\mathbb T}
\left|
\frac1N\sum_{j\neq k}h(s,x_j(\omega))
-
\frac{N-1}{N}
\int_{\mathbb T}h(s,y)\,d\mu_0(y)
\right|
\le
(1+2K)\varepsilon
\]
for all $N\ge N_0(\omega,\varepsilon)$. Finally, let
\[
\Omega_0
=
\bigcap_{l=1}^{\infty}
\Omega_{1/l}.
\]
Since $\mathbb P(\Omega_0)=1$, the desired almost sure convergence follows. The convergence
$$\sup_{1\leq k\leq N}
\left|
F_k(\mathbf{x}(\omega))-f_{\mu_0}(x_k(\omega))
\right|
\longrightarrow 0
\qquad\text{as }N\to\infty$$
is immediate from the definition of $F_k$ and the uniform convergence established above.
\end{proof}
\begin{remark}
Throughout the paper, whenever concentration inequalities such as
Hoeffding's inequality or McDiarmid's inequality are applied to
quantities indexed by $k$, we use the same union bound and
Borel--Cantelli argument as in Lemma \ref{epsilonnet}. Therefore, the
resulting convergence holds uniformly over all indices
$k=1,\ldots,N$ on a single probability-one set. In particular, such
convergences will be understood in the sense that 
\[
\max_{1\leq k\leq N}|A_{N,k}(\mathbf{x})|
\longrightarrow0
\quad\text{almost surely}.
\]
\end{remark}
\noindent We suppress the $\omega$-dependence from that point onward, with the understanding that $x_i(\omega)$ denotes the random realisation in almost-sure statements, whereas $x_i$ appearing under the integration over $\mathbb T^N$ is simply the integration variable. Moreover, all limits will be understood as \emph{\(\mathbb P\)-almost sure convergence} unless otherwise specified. Recall the definition in \eqref{Eq:f^[m]}.

\begin{lemma}\label{mutmut}
For any $t\geq1$, there exists a set $\Omega_0$ of probability one such that,
for every $\omega\in\Omega_0$,
\[
\max_{1\leq k\leq N}
\left|
F_k^t(\mathbf{x})-f_{\mu_0}^{[t]}(x_k)
\right|
\longrightarrow 0
\quad\text{as }N\to\infty .
\]
\end{lemma}
\begin{proof}
We proceed by induction on $t$. For $t=1$, the result follows from Lemma \ref{epsilonnet}. Suppose the claim holds for $t=m$, i.e.,
\[
\delta_m^{(N)}:=\max_{1\leq k\leq N}
\left|
F_k^t(\mathbf{x})-f_{\mu_0}^{[t]}(x_k)
\right|\rightarrow 0.
\]
Notice that,
\begin{equation}\label{m general}
F_k^{m+1}(\mathbf{x})
=
f(F_k^{m}(\mathbf{x}))
+\frac{1}{N}\sum_{i\neq k}
h(F_k^{m}(\mathbf{x}),F_i^{m}(\mathbf{x})).
\end{equation}
Since $h\in C^1(\mathbb T^2,\mathbb R)$, it is Lipschitz. Let
$
L=\sup_{(x,y)\in\mathbb T^2}\|D h(x,y)\|.
$
Then,
\begin{align*}
\frac1N\sum_{i\neq k}
\left|
h(F_k^{m}(\mathbf{x}),F_i^{m}(\mathbf{x}))
-
h(f_{\mu_0}^{[m]}(x_k),f_{\mu_0}^{[m]}(x_i))
\right|
&\leq
\frac{L}{N}\sum_{i\neq k}
\left(
|F_k^{m}(\mathbf{x})-f_{\mu_0}^{[m]}(x_k)|
+
|F_i^{m}(\mathbf{x})-f_{\mu_0}^{[m]}(x_i)|
\right)\\
&\le 2L\delta_{m}^{(N)}
\end{align*}
and therefore, 
\[
\max_{1\le k\le N}\frac1N\sum_{i\neq k}
\left|
h(F_k^{m}(\mathbf{x}),F_i^{m}(\mathbf{x}))
-
h(f_{\mu_0}^{[m]}(x_k),f_{\mu_0}^{[m]}(x_i))
\right|
\longrightarrow0 .
\]
Analogously, \(\max_{1\le k\le N}|(f\bigl(F_k^m(x)\bigr)
-
f\bigl(f_{\mu_0}^{[m]}(x_k)\bigr)|\rightarrow 0.\)
Hence, it remains to determine the limit of
\[
f(f_{\mu_0}^{[m]}(x_k))
+
\frac1N\sum_{i\neq k}
h(f_{\mu_0}^{[m]}(x_k),f_{\mu_0}^{[m]}(x_i)).
\]
Since the convergence in Lemma \ref{epsilonnet} is uniform in the first
argument, we may apply it to the function
\[
\widetilde{h}(s,x)
:=
h\bigl(s,f_{\mu_0}^{[m]}(x)\bigr).
\]
Moreover, since the variables $(x_i)_{i\geq1}$ are i.i.d. with
law $\mu_0$, the variables
$
f_{\mu_0}^{[m]}(x_i), i\geq1,
$
are i.i.d. with law
$
\mu_m=(f_{\mu_0}^{[m]})_*\mu_0.
$
Therefore, uniformly in $s\in\mathbb{T}$,
\[
\frac{1}{N}\sum_{i\neq k}
h\bigl(s,f_{\mu_0}^{[m]}(x_i)\bigr)
-
\int_{\mathbb{T}} h(s,y)\,d\mu_m(y)
\longrightarrow 0.
\]
Taking
$
s=f_{\mu_0}^{[m]}(x_k)
$
and using the uniformity in $s$, we obtain, uniformly in $k$,
\[
\frac{1}{N}\sum_{i\neq k}
h\bigl(f_{\mu_0}^{[m]}(x_k),
f_{\mu_0}^{[m]}(x_i)\bigr)
\longrightarrow
\int_{\mathbb{T}}
h\bigl(f_{\mu_0}^{[m]}(x_k),y\bigr)\,d\mu_m(y).
\]
Therefore, recalling \eqref{m general}
\[
F_k^{m+1}(\mathbf{x})
\longrightarrow
f(f_{\mu_0}^{[m]}(x_k))
+
\int_{\mathbb T}
h(f_{\mu_0}^{[m]}(x_k),y)\,d\mu_m(y).
\]
uniformly in k and and hence
\[
F_k^{m+1}(\mathbf{x})
\longrightarrow
f_{\mu_0}^{[m+1]}(x_k).
\]
This completes the induction.
\end{proof}

 From now on, we will omit the notation $\mathbf{x}$ from $F_k^t(\mathbf{x})$ whenever there is no ambiguity. Any exceptions to this convention will be stated explicitly. In the following lemmas, we use the notation $O(\cdot)$ with constants
that may depend on the fixed time $t$, but are independent of the system
size $N$, the indices $j,k$, and the point $\mathbf{x}$ (where applicable).
All such estimates are therefore uniform with respect to $N$, $j$, and
$k$ for each fixed $t$.

\begin{lemma}{\label{partialj}}
For any fixed $t \ge 1$, 
\[
\partial_j F_k^t=
\begin{cases}
O(1), & j=k,\\
O(1/N), & j \neq k.
\end{cases}
\]
\end{lemma}

\begin{proof}

Since $f\in C^1(\mathbb{T}, \mathbb{R})$ and $h\in C^1(\mathbb{T}^2, \mathbb{R})$, their derivatives are bounded. We have
\[
\partial_j F_k = \frac{1}{N}\,\partial_2 h(x_k, x_j) = O\!\left(\frac{1}{N}\right),\text{for}~j \neq k
~\text{and}~
\partial_k F_k = f'(x_k) + \frac{1}{N} \sum_{i \neq k} \partial_1 h(x_k, x_i) = O(1).
\]
Assume the statement holds for $t=m$, i.e.,
\[\partial_j F_k^m =
\begin{cases}
O(1), & j=k,\\
O(1/N), & j \neq k.
\end{cases}
\]
Then for $t=m+1$, recalling \eqref{m general}
\[
\partial_j F_k^{m+1} = f'(F_k^m) \partial_j F_k^m + \frac{1}{N} \sum_{i \neq k} \Big[ \partial_1 h(F_k^m,F_i^m) \partial_j F_k^m + \partial_2 h(F_k^m,F_i^m) \partial_j F_i^m \Big].
\]
and 
\[
\partial_k F_k^{m+1} = f'(F_k^m) \partial_k F_k^m + \frac{1}{N} \sum_{i \neq k} \Big[ \partial_1 h(F_k^m,F_i^m) \partial_k F_k^m + \partial_2 h(F_k^m,F_i^m) \partial_k F_i^m \Big].
\]
Using the induction hypotheses, for $j\neq k$ 
\[
\partial_j F_k^{m+1} = O(1/N)+(1/N) \Big((N-1)O(1/N)+O(1)+(N-2)O(1/N)\Big)=O(1/N).
\]
\[
\partial_k F_k^{m+1} = O(1)+(1/N) \Big((N-1)O(1)+(N-1)O(1/N)\Big)=O(1).
\]
\end{proof}

\begin{lemma}\label{mutmut1}
For any $t\geq1$, there exists a set $\Omega_0$ of probability one such that,
for every $\omega\in\Omega_0$,
\[
\max_{1\leq k\leq N}
\left|
\partial_k F_k^t(\mathbf{x})
-
(f_{\mu_0}^{[t]})'(x_k)
\right|
\longrightarrow 0
\quad\text{as }N\to\infty .
\]
\end{lemma}

\begin{proof}
We proceed by induction on $t$. For $t=1$, we have
\[
\partial_kF_k
=
f'(x_k)+\frac1N\sum_{i\neq k}\partial_1h(x_k,x_i).
\]
By Lemma \ref{epsilonnet},
\[
\max_{1\le k\le N}\left|\frac1N\sum_{i\neq k}\partial_1h(x_k,x_i)
-
\int_{\mathbb T}\partial_1h(x_k,y)\,d\mu_0(y)\right|\rightarrow 0.
\]
Therefore,
\[
\max_{1\le k\le N}\left|\partial_kF_k(\mathbf{x})
-
\left(f'(x_k)+\int_{\mathbb T}\partial_1h(x_k,y)\,d\mu_0(y)\right)\right|\rightarrow 0,
\]
and recall that $f'(x_k)+\int_{\mathbb T}\partial_1h(x_k,y)\,d\mu_0(y)=(f_{\mu_0})'(x_k)$. Now assume that for $t=m$,
\[
\max_{1\leq k\leq N}
\left|
\partial_kF_k^{m}(\mathbf{x})
-
(f_{\mu_0}^{[m]})'(x_k)
\right|
\longrightarrow0 .
\]
Using the chain rule,
\[
\partial_kF_k^{m+1}
=
\left(
f'(F_k^m)
+
\frac1N\sum_{i\neq k}
\partial_1h(F_k^m,F_i^m)
\right)
\partial_kF_k^m
+
\frac1N\sum_{i\neq k}
\partial_2h(F_k^m,F_i^m)
\partial_kF_i^m .
\]
By Lemma \ref{partialj}, for $i\neq k$,
\[
\partial_kF_i^m=O(1/N),
\]
and hence the second term converges to zero. It remains to study the first term. By the induction hypothesis and Lemma
\ref{mutmut},
\[
F_k^m(\mathbf{x})\longrightarrow f_{\mu_0}^{[m]}(x_k),
\qquad
\partial_kF_k^m(\mathbf{x})
\longrightarrow
(f_{\mu_0}^{[m]})'(x_k),
\]
uniformly in $k$. Therefore, using the same argument as in Lemma
\ref{epsilonnet}, one obtains
\[
\max_{1\leq k\leq N}\left|\partial_kF_k^{m+1}
-
\left(
f'(f_{\mu_0}^{[m]}(x_k))
+
\int_{\mathbb T}
\partial_1h
(f_{\mu_0}^{[m]}(x_k),f_{\mu_0}^{[m]}(y))
\,d\mu_0(y)
\right)
(f_{\mu_0}^{[m]})'(x_k)\right|\longrightarrow 0
\]
By the definition of $f_{\mu_m}$, the term in parentheses is
$
f_{\mu_m}'\!\left(f_{\mu_0}^{[m]}(x_k)\right).
$
Hence,
\[
f_{\mu_m}'\!\left(f_{\mu_0}^{[m]}(x_k)\right)
\left(f_{\mu_0}^{[m]}\right)'(x_k)
=
\left(f_{\mu_0}^{[m+1]}\right)'(x_k).
\]
\end{proof}

\begin{lemma}\label{partiall}
For any $t\geq 1$ and $l,j,k$,
\begin{equation}\label{MixedDerivatives}
\partial_l\partial_j F_k^t=
\begin{cases}
O(1/N), & l=j,\ j\neq k,\\
O(1/N), & l\neq j,\ j=k,\\
O(1/N), & l=k,\ j\neq k,\\
O(1/N^2), & l\neq j,\ j\neq k,\ l\neq k.
\end{cases}
\end{equation}
\end{lemma}

\begin{proof}
We argue by induction on $t$. For $t=1$,
\[
F_k=f(x_k)+\frac1N\sum_{i\neq k}h(x_k,x_i).
\]
Hence, if $l=j\neq k$,
\[
\partial_l\partial_jF_k
=\frac1N\partial_2^2h(x_k,x_j)=O(1/N),
\]
while, if $l=k\neq j$ or $j=k\neq l$,
\[
\partial_l\partial_jF_k
=\frac1N\partial_1\partial_2h(x_k,x_j)=O(1/N).
\]
If $l,j,k$ are pairwise distinct, then
$\partial_l\partial_jF_k=0$. Thus $\eqref{MixedDerivatives}$
holds for $t=1$. Assume that \eqref{MixedDerivatives} holds for $t=m$. Since
$F_k^{m+1}=F_k\circ F^m$, the chain rule gives
\[
\partial_l\partial_jF_k^{m+1}
=f''(F_k^m)\partial_lF_k^m\partial_jF_k^m
+f'(F_k^m)\partial_l\partial_jF_k^m+\frac1N\sum_{i\neq k}
\Big[
\partial_1^2h\,\partial_lF_k^m\partial_jF_k^m
+\partial_1\partial_2h\,\partial_lF_i^m\partial_jF_k^m\]
\[+\partial_1\partial_2h\,\partial_lF_k^m\partial_jF_i^m
+\partial_2^2h\,\partial_lF_i^m\partial_jF_i^m
+\partial_1h\,\partial_l\partial_jF_k^m
+\partial_2h\,\partial_l\partial_jF_i^m
\]
where the derivatives of $h$ are evaluated at
$(F_k^m,F_i^m)$. By Lemma~\ref{partialj},
\[
\partial_rF_q^m=
\begin{cases}
O(1),&r=q,\\
O(1/N),&r\neq q.
\end{cases}
\]
Since $f,h\in C^2$, their derivatives are uniformly bounded.
If $l=j\neq k$, then
$\partial_lF_k^m,\partial_jF_k^m=O(1/N)$ and
$\partial_l\partial_jF_k^m=O(1/N)$. Moreover,
\[
\frac1N\sum_{i\neq k}
\partial_lF_i^m\partial_jF_i^m=O(1/N),
\]
because only the term $i=j$ is $O(1)$ before multiplication by
$1/N$. All other terms are $O(1/N)$ or smaller, and hence
\[
\partial_l\partial_jF_k^{m+1}=O(1/N).
\]
If $j=k\neq l$, then
$\partial_jF_k^m=O(1)$,
$\partial_lF_k^m=O(1/N)$, and
$\partial_l\partial_jF_k^m=O(1/N)$. Since $i\neq k$,
$\partial_jF_i^m=O(1/N)$, and therefore all terms in the sum are
$O(1/N)$ or smaller. Thus
\[
\partial_l\partial_jF_k^{m+1}=O(1/N).
\]
The case $l=k\neq j$ is analogous, except that
$\partial_lF_k^m=O(1)$ and $\partial_jF_k^m=O(1/N)$. The only
potentially $O(1/N)$ contributions in the sum arise from the
indices $i=j$; hence
\[
\partial_l\partial_jF_k^{m+1}=O(1/N).
\]
Finally, suppose that $l,j,k$ are pairwise distinct. Then
\[
\partial_lF_k^m,\partial_jF_k^m=O(1/N),
\qquad
\partial_l\partial_jF_k^m=O(1/N^2).
\]
In each sum, the only possible $O(1)$ first derivatives occur for
$i=l$ or $i=j$, and the corresponding terms are multiplied by
$1/N$ together with another factor $O(1/N)$. Similarly, by the
induction hypothesis,
\[
\frac1N\sum_{i\neq k}
\partial_l\partial_jF_i^m=O(1/N^2),
\]
since the exceptional indices $i=l,j$ contribute $O(1/N)$ before
the prefactor $1/N$. Consequently,
\[
\partial_l\partial_jF_k^{m+1}=O(1/N^2).
\]
This completes the induction.
\end{proof}

\begin{lemma}\label{MCinequility}
For $\mathbf x=(x_1,...,x_N)$ and $\mathbf v=(v_1,...,v_N)$, define
\[
S_{t,k}(\mathbf{x},\mathbf{v})
=
\sum_{j\neq k}\partial_jF_k^t(\mathbf{x})v_j .
\]
Then, there exists a set $\Omega_0$ of probability one such that, for every
$\omega\in\Omega_0$,
\[
\sup_{s\in\mathbb{T}}
\max_{1\leq k\leq N}
\left|
S_{t,k}(x_1,...,x_{k-1},s,x_{k+1},...,x_N,\mathbf{v})
-
\mathbb{E}
\left[
S_{t,k}(\mathbf x,\mathbf{v})\mid x_k=s
\right]
\right|
\longrightarrow 0
\qquad\text{as }N\to\infty .
\]
\end{lemma}

\begin{proof}
Fix $s\in\mathbb{T}$. Since $x_k=s$ is fixed, we regard
$S_{t,k}$ as a function of the remaining random variables $(x_j)_{j\neq k}$.
For any $\ell\neq k$, changing only the $\ell$-th coordinate gives, by the
mean value theorem,
\[
\left|
S_{t,k}(x_1,\dots,x_\ell,\dots,x_N,\mathbf v)
-
S_{t,k}(x_1,\dots,x_\ell',\dots,x_N,\mathbf v)
\right| \leq
\sum_{j\neq k}
\sup_{\mathbf{x}\in\mathbb{T}^N}
|\partial_\ell\partial_jF_k^t(\mathbf{x})|
\,|x_\ell-x_\ell'|\,|v_j|.
\]
Since $(v_j)$ is uniformly bounded and Lemma \ref{partiall} gives uniform
bounds on the second derivatives, there exists $C_t>0$, independent of
$N,j,k$, such that
\[
\left|
S_{t,k}(x_1,\dots,x_\ell,\dots,x_N,\mathbf v)
-
S_{t,k}(x_1,\dots,x_\ell',\dots,x_N,\mathbf v)
\right|
\leq \frac{C_t}{N}.
\]
Hence $S_{t,k}$ satisfies the bounded differences condition. The variance
proxy in McDiarmid's inequality satisfies
\[
\sum_{\ell\neq k}\left(\frac{C_t}{N}\right)^2
\leq \frac{C_t^2}{N}.
\]
Therefore, for every $\varepsilon>0$,
\[
\mathbb{P}\left(
\left|
S_{t,k}(x_1,...,x_{k-1},s,x_{k+1},...,x_N,\mathbf v)
-
\mathbb{E}
\left[
S_{t,k}(\mathbf{x},\mathbf v)\mid x_k=s
\right]
\right|>\varepsilon
\right)
\leq 2e^{-c_t\varepsilon^2N},
\]
where $c_t>0$ is independent of $k$ and $s$. Taking the union bound over
$k$ gives
\[
\mathbb{P}\left(
\max_{1\leq k\leq N}
\left|
S_{t,k}(x_1,...,x_{k-1},s,x_{k+1},...,x_N,\mathbf v)
-
\mathbb{E}
\left[
S_{t,k}(\mathbf{x},\mathbf v)\mid x_k=s
\right]
\right|>\varepsilon
\right)
\leq 2Ne^{-c_t\varepsilon^2N}.
\]
The right-hand side is summable in $N$, and therefore the Borel--Cantelli
lemma gives almost sure convergence for every fixed $s$. Finally, the extension from fixed $s$ to uniform convergence in
$s\in\mathbb{T}$ follows by the same $\varepsilon$-net argument used in
Lemma \ref{epsilonnet}. Thus, there exists a single set $\Omega_0$ of
probability one such that, for every $\omega\in\Omega_0$,
\[
\sup_{s\in\mathbb{T}}
\max_{1\leq k\leq N}
\left|
S_{t,k}(\mathbf{x},\mathbf v)
-
\mathbb{E}
\left[
S_{t,k}(\mathbf{x},\mathbf v)\mid x_k=s
\right]
\right|
\longrightarrow0 .
\]
\end{proof}

\begin{lemma}\label{g_t(x_k)}
For every $t\geq1$, there exists a continuous function
$g_t:\mathbb T\to\mathbb R$ such that
\[
\sup_{1\leq k\leq N}\sup_{s\in\mathbb T}
\left|
\mathbb E\!\left[
R_{t,k}\mid x_k=s
\right]-g_t(s)
\right|
\longrightarrow0,
\qquad N\to\infty,
\]
where
$
R_{t,k}:=\sum_{j\ne k}\partial_jF_k^t.
$
For $t=1$ the function $g_t$ is given by
$
g_1(s)
=
\int_{\mathbb T}
\partial_2h(s,y)\,d\mu_0(y),
$
and, for $t\geq2$,
\[
g_t(s)=
f'_{\mu_{t-1}}
\!\left(f_{\mu_0}^{[t-1]}(s)\right)
g_{t-1}(s)
+
\int_{\mathbb T}
\partial_2h
\!\left(
f_{\mu_0}^{[t-1]}(s),
f_{\mu_0}^{[t-1]}(y)
\right)
(f_{\mu_0}^{[t-1]})'(y)
\,d\mu_0(y)\]
\[+\int_{\mathbb T}
\partial_2h
\!\left(
f_{\mu_0}^{[t-1]}(s),
f_{\mu_0}^{[t-1]}(y)
\right)
g_{t-1}(y)
\,d\mu_0(y).
\]
Moreover,
\[
\sup_{1\leq k\leq N}\sup_{s\in\mathbb T}
\left|
\mathbb E\!\left[S_{t,k}
\,\middle|\,x_k=s
\right]
-\bar v\,g_t(s)
\right|
\longrightarrow0.
\]
\end{lemma}
\begin{proof}
We prove the result by induction on $t$. For $t=1$,
\[
R_{1,k}
=
\sum_{j\ne k}\partial_jF_k
=
\frac1N\sum_{j\ne k}\partial_2h(x_k,x_j),
\]
and therefore
\[
\mathbb E[R_{1,k}\mid x_k=s]
=
\frac{N-1}{N}
\int_{\mathbb T}
\partial_2h(s,y)\,d\mu_0(y)
\to
\int_{\mathbb T}
\partial_2h(s,y)\,d\mu_0(y)
=:g_1(s).
\]
Assume that, for some $m\geq1$,
\[
|\mathbb E[R_{m,k}\mid x_k=s]-g_m(s)| \to 0
\qquad\text{for every }s\in\mathbb T.
\]
By Lemma~\ref{MCinequility}, applied with $v_j=1$ for $j\ne k$,
\[
|R_{m,k}(x_1,...,x_{k-1},s,x_{k+1},...,x_N)
-
\mathbb E[R_{m,k}\mid x_k=s]|
\to 0.
\]
Hence, by the induction hypothesis,
\begin{equation}{\label{lRmk}}
|R_{m,k}(x_1,...,x_{k-1},s,x_{k+1},...,x_N) - g_m(s)| \to 0.
\end{equation}
\noindent For $t=m+1$, the chain rule gives
\[
R_{m+1,k}
=
\left(
f'(F_k^m)
+\frac1N\sum_{i\ne k}
\partial_1h(F_k^m,F_i^m)
\right)R_{m,k}
+
\frac1N\sum_{i\ne k}
\partial_2h(F_k^m,F_i^m)\,
\partial_iF_i^m\]

\[+\frac1N\sum_{i\ne k}
\partial_2h(F_k^m,F_i^m)R_{m,i}
-
\frac1N\sum_{i\ne k}
\partial_2h(F_k^m,F_i^m)\,
\partial_kF_i^m.
\]
By Lemma~\ref{partialj},
\[
\partial_kF_i^m=O(N^{-1}),
\qquad i\ne k,
\]
and hence the last term is $O(N^{-1})$ uniformly in $k$. By Lemmas~\ref{mutmut} and~\ref{MCinequility}, together with \eqref{lRmk}, $R_{m+1,k}(x_1,...,x_{k-1},s,x_{k+1},...,x_N)$ converges to
\[\left(
f'(f_{\mu_0}^{[m]}(s))
+
\int_{\mathbb T}
\partial_1h
\bigl(
f_{\mu_0}^{[m]}(s),
f_{\mu_0}^{[m]}(y)
\bigr)
\,d\mu_0(y)
\right)g_m(s)
+
\int_{\mathbb T}
\partial_2h
\bigl(
f_{\mu_0}^{[m]}(s),
f_{\mu_0}^{[m]}(y)
\bigr)
(f_{\mu_0}^{[m]})'(y)
\,d\mu_0(y)
\]
\[+\int_{\mathbb T}
\partial_2h
\bigl(
f_{\mu_0}^{[m]}(s),
f_{\mu_0}^{[m]}(y)
\bigr)
g_m(y)
\,d\mu_0(y).
\]
uniformly in k. Moreover, by Lemma~\ref{partialj} ,
\[
\left|
R_{m+1,k}(x_1,\ldots,x_{k-1},s,x_{k+1},\ldots,x_N)
\right|
\leq C
\]
uniformly in $N$ and $k$. By Dominated convergene theorem \[
|\mathbb E[R_{m+1,k}\mid x_k=s]-g_{m+1}(s)| \to 0
\qquad\text{uniformly in k for every }s\in\mathbb T.
\] 
It remains to upgrade the convergence to uniform convergence in $s$. We have
\[
\left|
\frac{d}{ds}
\mathbb{E}[R_{t,k}\mid x_k=s]
\right|
=
\left|
\sum_{j\ne k}
\mathbb{E}\!\left[
\partial_k\partial_j F_k^t
\mid x_k=s
\right]
\right|\leq
\sum_{j\ne k}
\mathbb{E}\!\left[
\left|\partial_k\partial_j F_k^t\right|
\mid x_k=s
\right].
\]
By Lemma \ref{partiall}, there exists a constant $C_0>0$, independent of
$N$, $k$, and $s$, such that
\[
\left|\partial_k\partial_j F_k^t\right|
\leq \frac{C_0}{N},
\qquad j\ne k.
\]
Therefore,
\[
\left|
\frac{d}{ds}
\mathbb{E}[R_{t,k}\mid x_k=s]
\right|
\leq
\sum_{j\ne k}\frac{C_0}{N}
=
\frac{N-1}{N}C_0
\leq C_0.
\]
Hence the family
\[
\left\{
s\mapsto\mathbb{E}[R_{t,k}\mid x_k=s]
\right\}_{N,k}
\]
is equi-Lipschitz in $s$. Since $g_t$ is continuous on the compact space $\mathbb{T}$, it is uniformly
continuous. Together with the pointwise convergence, which is already
uniform in $k$, the equi-Lipschitz property implies uniform convergence
in $s$. Thus
\[
\sup_{1\leq k\leq N}\sup_{s\in\mathbb{T}}
\left|
\mathbb{E}[R_{t,k}\mid x_k=s]-g_t(s)
\right|
\longrightarrow 0.
\]
Using linearity of expectation and the fact that
$\mathbb E[\partial_jF_k^t\mid x_k=s]$
is independent of $j$, we obtain 
\[
\lim_{N\to\infty}
\mathbb E\!\left[
\sum_{j\ne k}\partial_jF_k^t\,v_j
\,\middle|\,
x_k=s
\right]
=
\lim_{N\to\infty}
(N-1)
\mathbb E[\partial_jF_k^t\mid x_k=s]
\left(
\frac1{N-1}
\sum_{j\ne k}v_j
\right)
=
\bar v\,
\lim_{N\to\infty}
\mathbb E\!\left[
\sum_{j\ne k}\partial_jF_k^t
\,\middle|\,
x_k=s
\right].
\]
Since
\[
\sup_{1\leq k\leq N}
\left|
\frac1{N-1}\sum_{j\ne k}v_j-\bar v
\right|
\longrightarrow0,
~~~\text{and}~~~
\sup_{1\leq k\leq N}\sup_{s\in\mathbb T}
\left|
\mathbb E[R_{t,k}\mid x_k=s]-g_t(s)
\right|
\longrightarrow0,
\]
we obtain
\[
\sup_{1\leq k\leq N}\sup_{s\in\mathbb T}
\left|
\mathbb E\!\left[S_{t,k}
\,\middle|\,x_k=s
\right]
-\bar v\,g_t(s)
\right|
\longrightarrow0.
\]
This completes the proof.
\end{proof}

 Recall
\[
H_{1,N}(x_1,\ldots,x_N)
=
\sum_{k=1}^{N}
\partial_k \psi\!\left(
f_{\mu_0}^{[t]}(x_1),\ldots,f_{\mu_0}^{[t]}(x_N)
\right)
(f_{\mu_0}^{[t]})'(x_k)v_k,
\]
and
\[
H_{2,N}(x_1,\ldots,x_N)
=
\bar{v}\sum_{k=1}^{N}
\partial_k \psi\!\left(
f_{\mu_0}^{[t]}(x_1),\ldots,f_{\mu_0}^{[t]}(x_N)
\right)
g_t(x_k).
\]
\begin{lemma}{\label{h1h2}}
\[
|H_{i,N}-\mathbb{E}[H_{i,N}]|
\xrightarrow[N\to\infty]{a.s.} 0,
\qquad i=1,2,
\]
where the expectation is taken with respect to the joint law of
\((x_1,\ldots,x_N)\).
\end{lemma}
\begin{proof}
We  show that \(H_{1,N}\) (and similarly \(H_{2,N}\)) satisfies the bounded differences property, i.e.,
\[
\bigl|H_i(x_1,\ldots,x_\ell,\ldots,x_N)-H_i(x_1,\ldots,x_\ell',\ldots,x_N)\bigr|
= O\!\left(\frac{1}{N}\right), \quad i=1,2.
\]
The result then follows from McDiarmid’s inequality and the Borel–Cantelli lemma by the same procedure as in Lemma \ref{MCinequility} and Lemma \ref{g_t(x_k)}. 
To verify the bounded-differences condition, fix \(1\leq \ell\leq N\) and replace only the
\(\ell\)-th coordinate \(x_\ell\) by \(x_\ell'\). Then

\[
\Bigl|
H_1(x_1,\ldots,x_\ell,\ldots,x_N)
-
H_1(x_1,\ldots,x_\ell',\ldots,x_N)
\Bigr|
\leq
\Bigl|
\partial_\ell\psi(\mathbf{f}_{\mu_0}^{t}(\mathbf{x}))\,(f_{\mu_0}^{[t]})'(x_\ell)
-
\partial_\ell\psi(\mathbf{f}_{\mu_0}^{t}(\mathbf{x'}))\,(f_{\mu_0}^{[t]})'(x_\ell')
\Bigr||v_l|\]
\[
+
\sum_{k\neq \ell}
\Bigl|
\partial_k\psi(\mathbf{f}_{\mu_0}^{t}(\mathbf{x}))-\partial_k\psi(\mathbf{f}_{\mu_0}^{t}(\mathbf{x'}))
\Bigr|
|(f_{\mu_0}^{[t]})'(x_k)||v_k|,
\]
where
$\mathbf{f}_{\mu_0}^{t}(\mathbf{x})
=
\bigl(
f_{\mu_0}^{[t]}(x_1),\ldots,f_{\mu_0}^{[t]}(x_\ell),\ldots,f_{\mu_0}^{[t]}(x_N)
\bigr)$
~and~
$\mathbf{f}_{\mu_0}^{t}(\mathbf{x'})=
\bigl(
f_{\mu_0}^{[t]}(x_1),\ldots,f_{\mu_0}^{[t]}(x_\ell'),\ldots,f_{\mu_0}^{[t]}(x_N)
\bigr).$
Applying the triangle inequality to the first term yields
\[
\Bigl|
\partial_\ell\psi(\mathbf{f}_{\mu_0}^{t}(\mathbf{x}))\,(f_{\mu_0}^{[t]})'(x_\ell)
-
\partial_\ell\psi(\mathbf{f}_{\mu_0}^{t}(\mathbf{x'}))\,(f_{\mu_0}^{[t]})'(x_\ell')\Big| \]
\[\leq
\Bigl|
\partial_\ell\psi(\mathbf{f}_{\mu_0}^{t}(\mathbf{x}))-\partial_\ell\psi(\mathbf{\mathbf{f}_{\mu_0}^{t}(\mathbf{x'})})
\Bigr|
\,|(f_{\mu_0}^{[t]})'(x_\ell')|
+
|\partial_\ell\psi(\mathbf{f}_{\mu_0}^{t}(\mathbf{x}))|
\,\Bigl|
(f_{\mu_0}^{[t]})'(x_\ell)-(f_{\mu_0}^{[t]})'(x_\ell')
\Bigr|.
\]
By the mean value theorem,

\[
\Bigl|
\partial_\ell\psi(\mathbf{f}_{\mu_0}^{t}(\mathbf{x}))-\partial_\ell\psi(\mathbf{f}_{\mu_0}^{t}(\mathbf{x'}))
\Bigr|
\leq
\|\partial_\ell^2\psi\|_\infty
\,|f_{\mu_0}^{[t]}(x_\ell)-f_{\mu_0}^{[t]}(x_\ell')|,
\]
and therefore
\[
\Bigl|
\partial_\ell\psi(\mathbf{f}_{\mu_0}^{t}(\mathbf{x}))-\partial_\ell\psi(\mathbf{f}_{\mu_0}^{t}(\mathbf{x'}))
\Bigr|
\,| (f_{\mu_0}^{[t]})'(x_\ell')|
=
O\!\left(\frac1N\right).
\]
Moreover,
\[
|\partial_\ell\psi(\mathbf{f}_{\mu_0}^{t}(\mathbf{x}))|
\Bigl|
(f_{\mu_0}^{[t]})'(x_\ell)-(f_{\mu_0}^{[t]})'(x_\ell')
\Bigr|
=
O\!\left(\frac1N\right).
\]
For \(k\neq \ell\), another application of the mean value theorem gives
\[
\Bigl|
\partial_k\psi(\mathbf{f}_{\mu_0}^{t}(\mathbf{x}))-\partial_k\psi(\mathbf{f}_{\mu_0}^{t}(\mathbf{x'}))
\Bigr|
\leq
\|\partial_\ell \partial_k\psi\|_\infty
\,|f_{\mu_0}^{[t]}(x_\ell)-f_{\mu_0}^{[t]}(x_\ell')|
=
O\!\left(\frac1{N^2}\right).
\]
Hence
\[
\sum_{k\neq \ell}
\Bigl|
\partial_k\psi(F)-\partial_k\psi(F')
\Bigr|
\,| (f_{\mu_0}^{[t]})'(x_k)|
\leq
N\,O\!\left(\frac1{N^2}\right)
=
O\!\left(\frac1N\right).
\]
Combining the estimates, we obtain
\[
\Bigl|
H_1(x_1,\ldots,x_\ell,\ldots,x_N)
-
H_1(x_1,\ldots,x_\ell',\ldots,x_N)
\Bigr|
=
O\!\left(\frac1N\right).
\]
Therefore \(H_1\) satisfies the bounded-differences condition.
The proof for
\[
H_2(x_1,\ldots,x_N)
=
\sum_{k=1}^{N}
\partial_k \psi\!\left(
f_{\mu_0}^{[t]}(x_1),\ldots,f_{\mu_0}^{[t]}(x_N)
\right)
g_t(x_k)
\]
is identical. Replacing \(({f_{\mu_0}^{[t]}})'(x_k)\) by \(g_t(x_k)\)\footnote{
For each fixed $t$, $g_t$ is bounded and Lipschitz. This follows
inductively from its recursive definition, since $f,h\in C^2$ and all
terms defining $g_t$ are compositions and integrals of $C^1$ functions.
}  in the above argument yields

\[
\Bigl|
H_2(x_1,\ldots,x_\ell,\ldots,x_N)
-
H_2(x_1,\ldots,x_\ell',\ldots,x_N)
\Bigr|
=
O\!\left(\frac1N\right).
\]
\end{proof}

\begin{lemma}\label{exph1h2}
Let
\[
\psi_{k,t}(s)
=
\int_{\mathbb{T}^{N-1}}
\psi(x_1,\ldots,x_{k-1},s,x_{k+1},\ldots,x_N)
\prod_{j\neq k}d\mu_t(x_j),
\]
where
$
\mu_t=(f_{\mu_0}^{[t]})_\ast\mu_0,$
~and~$\overline{v_N}=\frac1N\sum_{k=1}^N v_k .
$
Then
\[
\mathbb{E}[H_1+H_2]
=
-N\int_{\mathbb T}
\psi_{k,t}(s)
\left[
(f_{\mu_0}^{[t]})_\ast
\left(
\rho_0(s)(f_{\mu_0}^{[t]})'(s)\overline{v_N}
+
{g_t(s)\rho_0(s)}\bar v
\right)
\right]'ds .
\]
\end{lemma}

\begin{proof}
Since \(\psi\) is permutation invariant and the coordinates have
identical laws, the corresponding unweighted terms have the same
expectation. Therefore,
\[
\mathbb{E}[H_1+H_2]
=
N \int \partial_k \psi\!\left(
f_{\mu_0}^{[t]}(x_1),\ldots,f_{\mu_0}^{[t]}(x_N)
\right)
\left((f_{\mu_0}^{[t]})' (x_k)\overline{v_N}+g_t(x_k)\bar{v}\right)
\prod_{j=1}^N d\mu_0(x_j).\]
Writing \(\rho_0\) for the density of \(\mu_0\), we obtain
\[
\mathbb{E}[H_1+H_2]=
N \int \partial_{k} \psi\left(
f_{\mu_0}^{[t]}(x_1),\ldots,f_{\mu_0}^{[t]}(x_N)
\right)
\left(\rho_0(x_k)(f_{\mu_0}^{[t]})' (x_k)\overline{v_N}+\rho_0(x_k)g_t(x_k)\bar{v}\right)
dx_k
\prod_{j\neq k} d\mu_0(x_j).
\]
By the change of variable formula and recognising the push forward measure \(\mu_t = (f_{\mu_0}^{[t]})_\ast \mu_0\), this becomes
\[
\mathbb{E}[H_1+H_2]=
N \int \partial_{k} \psi\left(
x_1,\ldots,s,\dots,x_N
\right)
\left[
(f_{\mu_0}^{[t]})_\ast
\left(
\rho_0(s)(f_{\mu_0}^{[t]})'(s)\overline{v_N}
+
{g_t(s)\rho_0(s)}\bar v
\right)
\right]
ds
\prod_{j\neq k} d\mu_t(x_j).
\]
By integration by parts,
\[
\mathbb{E}[H_1+H_2]=
-N \int  \psi\left(
x_1,\ldots,s,\ldots,x_N
\right)
\left[
(f_{\mu_0}^{[t]})_\ast
\left(
\rho_0(s)(f_{\mu_0}^{[t]})'(s)\overline{v_N}
+
{g_t(s)\rho_0(s)}\bar v
\right)
\right]'
ds
\prod_{j\neq k} d\mu_t(x_j).
\]

\[
\mathbb{E}[H_1+H_2]
=
-N
\int_{\mathbb{T}}
\left(
\int_{\mathbb{T}^{N-1}}
\psi(x_1,\ldots,s,\ldots,x_N)
\prod_{j\neq k} d\mu_t(x_j)
\right)
\left[
(f_{\mu_0}^{[t]})_\ast
\left(
\rho_0(s)(f_{\mu_0}^{[t]})'(s)\overline{v_N}
+
{g_t(s)\rho_0(s)}\bar v
\right)
\right]'
ds.
\]
\end{proof}

\noindent For the next lemma, we use the following convention throughout. Let $\nu$ be a measure with density $\rho_0'$, i.e. $d\nu=\rho_0'(s)\,ds.$ We denote by
\[
h_m
=
D\mathcal{F}_{\mu_{m-1}}\circ\cdots\circ D\mathcal{F}_{\mu_0}(\nu),
\]
the measure obtained after applying the sequence of differential operators to $\nu$. When we write
\[
h_m(s)=
\left(
D\mathcal{F}_{\mu_{m-1}}\circ\cdots\circ D\mathcal{F}_{\mu_0}(\rho_0')
\right)(s),
\]
we mean the density notation corresponding to the measure $h_m$. Moreover, for a map $f$ (which may represent $f_\mu$, $f_{\mu_0}^{[m]}$, or any other iterate), we use $f_*$ to denote the push-forward operator. When applied to a measure, $f_*\mu$ denotes the push-forward measure. When applied to a density $q$, the notation $f_*q$ denotes the density of the push-forward of the measure with density $q$, i.e.
\[
f_*q \quad\text{denotes the density of}\quad f_*(q(s)\,ds).
\]
\begin{lemma}\label{relationtodifferential}
For any $C^\infty$ function $\varphi:\mathbb{T}^1\to\mathbb{R}$,
\[
\int
\varphi(s)
\left(
D\mathcal{F}_{\mu_{t-1}}\circ\cdots\circ D\mathcal{F}_{\mu_0}
(\rho_0')
\right)(s)\,ds
=
\int
\varphi(s)
\left[
(f_{\mu_0}^{[t]})_\ast
\left(
\rho_0(s)(f_{\mu_0}^{[t]})'(s)
+
g_t(s)\rho_0(s)
\right)
\right]'ds .
\]
\end{lemma}

\begin{proof}
We prove the statement by induction on $t$. For the base case $t=1$, we first establish the desired identity. Define
\[
h_1
:=
D\mathcal{F}_{\mu_0}(\nu)
=
\lim_{\varepsilon\to 0}
\frac{
(f_{\mu_0+\varepsilon\nu})_*(\mu_0+\varepsilon\nu)
-
(f_{\mu_0})_*\mu_0
}{\varepsilon},
\]
where $\nu$ is the measure with density $\rho_0'$. Hence,
\[
h_1
=
\lim_{\varepsilon\to 0}
\frac{
(f_{\mu_0+\varepsilon\nu})_*\mu_0
-
(f_{\mu_0})_*\mu_0
}{\varepsilon}
+
(f_{\mu_0})_*\nu.
\]
Testing against $\varphi$, we obtain
\[
\int \varphi(s)
D\mathcal{F}_{\mu_0}(\rho_0')(s)\,ds
=
\lim_{\varepsilon\to 0}
\frac{1}{\varepsilon}
\int
\Big(
(\varphi\circ f_{\mu_0+\varepsilon\nu})(s)
-
(\varphi\circ f_{\mu_0})(s)
\Big)
\rho_0(s)\,ds
+
\int
(\varphi\circ f_{\mu_0})(s)\rho_0'(s)\,ds.
\]
Moreover,
\[
\frac{
(\varphi\circ f_{\mu_0+\varepsilon\nu})(s)
-
(\varphi\circ f_{\mu_0})(s)
}{\varepsilon}
\longrightarrow
(\varphi'\circ f_{\mu_0})(s)
\int h(s,y)\rho_0'(y)\,dy.
\]
Therefore,
\begin{equation}\label{t=1}
\int
\varphi(s)
D\mathcal{F}_{\mu_0}(\rho_0')(s)\,ds
=
\int
(\varphi\circ f_{\mu_0})(s)\rho_0'(s)\,ds
+
\int
(\varphi'\circ f_{\mu_0})(s)
\left(
\int h(s,y)\rho_0'(y)\,dy
\right)
\rho_0(s)\,ds.
\end{equation}
Using integration by parts in the first term, and integration by
parts with respect to $y$ in the second term, we obtain
\[
\int
\varphi(s)
D\mathcal{F}_{\mu_0}(\rho_0')(s)\,ds
=
-\int
(\varphi\circ f_{\mu_0})'(s)\rho_0(s)\,ds
-
\int
(\varphi'\circ f_{\mu_0})(s)
\left(
\int
\partial_2h(s,y)\rho_0(y)\,dy
\right)
\rho_0(s)\,ds.
\]
Hence,
\[
\int
\varphi(s)
D\mathcal{F}_{\mu_0}(\rho_0')(s)\,ds
=
-\int
(\varphi'\circ f_{\mu_0})(s)
(f_{\mu_0})'(s)\rho_0(s)\,ds
-
\int
(\varphi'\circ f_{\mu_0})(s)
\left(
\int
\partial_2h(s,y)\rho_0(y)\,dy
\right)
\rho_0(s)\,ds.
\]
Using the definition of the push-forward, this becomes
\[
\begin{aligned}
\int
\varphi(s)
D\mathcal{F}_{\mu_0}(\rho_0')(s)\,ds
={}&
-\int
\varphi'(s)
\left[
(f_{\mu_0})_*
\left(
\rho_0(s)(f_{\mu_0})'(s)
\right)
\right]ds
\\
&-
\int
\varphi'(s)
\left[
(f_{\mu_0})_*
\left(
\left(
\int
\partial_2h(s,y)\rho_0(y)\,dy
\right)
\rho_0(s)
\right)
\right]ds.
\end{aligned}
\]
Applying integration by parts once again, we obtain
\[
\begin{aligned}
\int
\varphi(s)
D\mathcal{F}_{\mu_0}(\rho_0')(s)\,ds
={}&
\int
\varphi(s)
\left[
(f_{\mu_0})_*
\left(
\rho_0(s)(f_{\mu_0})'(s)
\right)
\right]'ds
\\
&+
\int
\varphi(s)
\left[
(f_{\mu_0})_*
\left(
\left(
\int
\partial_2h(s,y)\rho_0(y)\,dy
\right)
\rho_0(s)
\right)
\right]'ds.
\end{aligned}
\]
This proves the statement for $t=1$.
Now assume that the statement holds for some $m\geq 1$. We prove that
it also holds for $m+1$. Recall that $h_m$ denotes the measure, while
$h_m(s)$ denotes its density. By the induction hypothesis,
\[
h_m(s)
=
\left[
(f_{\mu_0}^{[m]})_\ast
\left(
\rho_0(s)(f_{\mu_0}^{[m]})'(s)
+
g_m(s)\rho_0(s)
\right)
\right]'.
\]
The computation leading to \eqref{t=1} yields the following one-step
formula: if $\eta$ is a measure with density $q$, then
\[
\int
\varphi(s)
D\mathcal{F}_{\mu}(\eta)(s)\,ds
=
\int
(\varphi\circ f_{\mu})(s)q(s)\,ds
+
\int
(\varphi'\circ f_{\mu})(s)
\left(
\int h(s,y)q(y)\,dy
\right)
d\mu(s).
\]
Applying this identity with $\mu=\mu_m$ and $\eta=h_m$, we obtain
\begin{equation}\label{indindiff}
\begin{aligned}
\int
\varphi(s)
D\mathcal{F}_{\mu_m}(h_m)(s)\,ds
={}&
\int
(\varphi\circ f_{\mu_m})(s)h_m(s)\,ds
\\
&+
\int
(\varphi'\circ f_{\mu_m})(s)
\left(
\int h(s,y)h_m(y)\,dy
\right)
d\mu_m(s).
\end{aligned}
\end{equation}

Substituting the expression for $h_m$ given by the induction hypothesis
into \eqref{indindiff}, we obtain
\[
\begin{aligned}
\int
\varphi(s)
D\mathcal{F}_{\mu_m}(h_m)(s)\,ds
={}&
\int
(\varphi\circ f_{\mu_m})(s)
\left[
(f_{\mu_0}^{[m]})_\ast
\left(
\rho_0(s)(f_{\mu_0}^{[m]})'(s)
\right)
\right]'ds
\\
&+
\int
(\varphi\circ f_{\mu_m})(s)
\left[
(f_{\mu_0}^{[m]})_\ast
\left(
g_m(s)\rho_0(s)
\right)
\right]'ds
\\
&+
\int
(\varphi'\circ f_{\mu_m})(s)
\left(
\int
h(s,y)
\left[
(f_{\mu_0}^{[m]})_\ast
\left(
\rho_0(y)(f_{\mu_0}^{[m]})'(y)
\right)
\right]'dy
\right)
d\mu_m(s)
\\
&+
\int
(\varphi'\circ f_{\mu_m})(s)
\left(
\int
h(s,y)
\left[
(f_{\mu_0}^{[m]})_\ast
\left(
g_m(y)\rho_0(y)
\right)
\right]'dy
\right)
d\mu_m(s).
\end{aligned}
\]
Using integration by parts, we get
\[
\begin{aligned}
\int
\varphi(s)
D\mathcal{F}_{\mu_m}(h_m)(s)\,ds
={}&
-\int
(\varphi\circ f_{\mu_m})'(s)
\left[
(f_{\mu_0}^{[m]})_\ast
\left(
\rho_0(s)(f_{\mu_0}^{[m]})'(s)
\right)
\right]ds
\\
&-
\int
(\varphi\circ f_{\mu_m})'(s)
\left[
(f_{\mu_0}^{[m]})_\ast
\left(
g_m(s)\rho_0(s)
\right)
\right]ds
\\
&-
\int
(\varphi'\circ f_{\mu_m})(s)
\left(
\int
\partial_2h(s,y)
\left[
(f_{\mu_0}^{[m]})_\ast
\left(
\rho_0(y)(f_{\mu_0}^{[m]})'(y)
\right)
\right]dy
\right)
d\mu_m(s)
\\
&-
\int
(\varphi'\circ f_{\mu_m})(s)
\left(
\int
\partial_2h(s,y)
\left[
(f_{\mu_0}^{[m]})_\ast
\left(
g_m(y)\rho_0(y)
\right)
\right]dy
\right)
d\mu_m(s).
\end{aligned}
\]
Since
$
\mu_m=(f_{\mu_0}^{[m]})_\ast\mu_0,
$
using the change-of-variables formula and the chain rule, the preceding
equation becomes
\[
\begin{aligned}
\int
\varphi(s)
D\mathcal{F}_{\mu_m}(h_m)(s)\,ds
={}&-\int
\varphi'(s)
\left[
(f_{\mu_0}^{[m+1]})_\ast
\left(
\rho_0(s)(f_{\mu_0}^{[m+1]})'(s)
\right)
\right]ds
\\
&-
\int
\varphi'(s)
\left[
(f_{\mu_0}^{[m+1]})_\ast
\left(
f'_{\mu_m}(f_{\mu_0}^{[m]}(s))
g_m(s)\rho_0(s)
\right)
\right]ds
\\
&-
\int
(\varphi'\circ f_{\mu_0}^{[m+1]})(s)
\left(
\int
\partial_2h
\bigl(
f_{\mu_0}^{[m]}(s),
f_{\mu_0}^{[m]}(y)
\bigr)
(f_{\mu_0}^{[m]})'(y)\rho_0(y)
\,dy
\right)
\rho_0(s)\,ds
\\
&-
\int
(\varphi'\circ f_{\mu_0}^{[m+1]})(s)
\left(
\int
\partial_2h
\bigl(
f_{\mu_0}^{[m]}(s),
f_{\mu_0}^{[m]}(y)
\bigr)
g_m(y)\rho_0(y)
\,dy
\right)
\rho_0(s)\,ds.
\end{aligned}
\]
Using the definition of the push-forward, we get
\[
\begin{aligned}
\int
\varphi(s)
D\mathcal{F}_{\mu_m}(h_m)(s)\,ds
={}&-\int
\varphi'(s)
\left[
(f_{\mu_0}^{[m+1]})_\ast
\left(
\rho_0(s)(f_{\mu_0}^{[m+1]})'(s)
\right)
\right]ds
\\
&-
\int
\varphi'(s)
\left[
(f_{\mu_0}^{[m+1]})_\ast
\left(
f'_{\mu_m}(f_{\mu_0}^{[m]}(s))
g_m(s)\rho_0(s)
\right)
\right]ds
\\
&-
\int
\varphi'(s)
\left[
(f_{\mu_0}^{[m+1]})_\ast
\left(
\rho_0(s)
\int
\partial_2h
\bigl(
f_{\mu_0}^{[m]}(s),
f_{\mu_0}^{[m]}(y)
\bigr)
(f_{\mu_0}^{[m]})'(y)\rho_0(y)
\,dy
\right)
\right]ds
\\
&-
\int
\varphi'(s)
\left[
(f_{\mu_0}^{[m+1]})_\ast
\left(
\rho_0(s)
\int
\partial_2h
\bigl(
f_{\mu_0}^{[m]}(s),
f_{\mu_0}^{[m]}(y)
\bigr)
g_m(y)\rho_0(y)
\,dy
\right)
\right]ds.
\end{aligned}
\]
By the definition of $g_{m+1}$ from Lemma \ref{g_t(x_k)},
\[
g_{m+1}(s)
=
f'_{\mu_m}(f_{\mu_0}^{[m]}(s))g_m(s)
+
\int
\partial_2h
\bigl(
f_{\mu_0}^{[m]}(s),
f_{\mu_0}^{[m]}(y)
\bigr)
(f_{\mu_0}^{[m]})'(y)\rho_0(y)\,dy
+
\int
\partial_2h
\bigl(
f_{\mu_0}^{[m]}(s),
f_{\mu_0}^{[m]}(y)
\bigr)
g_m(y)\rho_0(y)\,dy.
\]
Consequently,
\[
\begin{aligned}
\int
\varphi(s)
D\mathcal{F}_{\mu_m}(h_m)(s)\,ds
=
-\int
\varphi'(s)
\left[
(f_{\mu_0}^{[m+1]})_\ast
\left(
\rho_0(s)(f_{\mu_0}^{[m+1]})'(s)
+
g_{m+1}(s)\rho_0(s)
\right)
\right]ds.
\end{aligned}
\]
Finally, applying integration by parts once more, we obtain
\[
\begin{aligned}
\int
\varphi(s)
D\mathcal{F}_{\mu_m}(h_m)(s)\,ds
=
\int
\varphi(s)
\left[
(f_{\mu_0}^{[m+1]})_\ast
\left(
\rho_0(s)(f_{\mu_0}^{[m+1]})'(s)
+
g_{m+1}(s)\rho_0(s)
\right)
\right]'ds.
\end{aligned}
\]
Therefore, the statement holds for $t=m+1$. By induction, the result
follows for every $t\geq 1$.
\end{proof}

\begin{proof}[\textbf{Proof of Theorem \ref{Maintheorem1}}:]
Having established the necessary lemmas, we return to the problem of determining $$\frac{1}{t}\log\lim_{N\to\infty} \big|(G_{1, N}(\mathbf{x})+G_{2, N}(\mathbf{x}))\big|$$
First, we will prove that 
\begin{equation}{\label{proofmaingoal}}
\lim_{N\to\infty} \Big|(G_{1,N}(\mathbf{x)}+G_{2, N}(\mathbf{x}))-(H_{1,N}(\mathbf{x})+H_{2,N}(\mathbf{x}))\Big|=0
\end{equation}
$$\lim_{N\to\infty} \big|(G_{1,N}(\mathbf{x})+G_{2, N}(\mathbf{x}))-(H_{1,N}(\mathbf{x})+H_{2,N}(\mathbf{x}))\big| \leq \lim_{N\to\infty} \big|(G_{1,N}(\mathbf{x}) - H_{1, N}(\mathbf{x}))|+|(G_{2,N}(\mathbf{x})-H_{2,N}(\mathbf{x}))\big|$$
Let $ F^{t}(\mathbf{x})
=
\bigl(F_1^{t}(\mathbf{x}),\ldots,F_N^{t}(\mathbf{x})\bigr)$
and
$\mathbf{f}_{\mu_0}^{t}(\mathbf{x})
=
\bigl(f_{\mu_0}^{[t]}(x_1),\ldots,f_{\mu_0}^{[t]}(x_N)\bigr).
$
$$
\big|(G_{1,N}(\mathbf{x})-H_{1,N}(\mathbf{x}))\big|\leq \sum_{k=1}^N
\big|\partial_k \psi(F^t(\mathbf{x}))\partial_{k}F_k^t(\mathbf{x})v_k
-\partial_k \psi(\mathbf{f}_{\mu_0}^{t}(\mathbf{x}))(f_{\mu_0}^{[t]})'(x_k)v_k
\big|$$
By triangular inequity
$$|(G_{1,N}(\mathbf{x})-H_{1,N}(\mathbf{x}))|\leq \sum_{k=1}^N
\big|
\partial_k \psi(F^t(\mathbf{x})
-
\partial_k \psi(\mathbf{f}_{\mu_0}^{t}(\mathbf{x})
\big||\partial_{k}F_k^t(\mathbf{x})||v_k| +
\sum_{k=1}^N
\big|
\partial_{k}F_k^t(\mathbf{x})-f{'}_\mu^{[t]}(x_k)
\big||\partial_k \psi(\mathbf{f}_{\mu_0}^{t}(\mathbf{x})||v_k|.
$$
\noindent By the mean value theorem, we have
\[
\left|
\partial_k \psi(F^t(\mathbf{x})
-
\partial_k \psi(\mathbf{f}_{\mu_0}^{t}(\mathbf{x})
\right|
\le
\|\partial_k^2\|_{\infty}|F_k^t(\mathbf{x})-f_{\mu_0}^{[t]}(x_l)| +
\sum_{l \neq k}^N \|\partial_l\partial_k \psi\|_{\infty}|F_l^t(\mathbf{x})-f_{\mu_0}^{[t]}(x_l)|.
\]
Using $\|\partial_l \partial_k \psi\|_{\infty} = O(1/N^2)$ for $l \neq k$, and  $\|\partial_k^2 \psi \|_{\infty}= O(1/N)$,
\begin{equation}{\label{pmainproof1}}
|(G_{1,N}(\mathbf{x})-H_{1,N}(\mathbf{x}))| \leq \frac{C}{N}
\left(
\sum_{k=1}^N |F_k^t(\mathbf{x})-f_{\mu_0}^{[t]}(x_k)|
+
\sum_{k=1}^N |\partial_{k}F_k^t(\mathbf{x})-(f_{\mu_0}^{[t]})'(x_k)|
\right),
\end{equation}
which converges to $0$ as $N \to \infty$ by Lemma \ref{mutmut} and \ref{mutmut1}. 
Also
\begin{equation}{\label{pmainproof2}}
|(G_{2,N}(\mathbf{x})-H_{2,N}(\mathbf{x}))| \leq \frac{C}{N}
\left(
\sum_{k=1}^N |F_k^t(\mathbf{x})-f_{\mu_0}^{[t]}(x_k)|
+
\sum_{k=1}^N \Big|\sum_{j \neq k}\partial_{j}F_k^t(\mathbf{x})v_j-\bar{v}g_t(x_k)\Big|
\right),
\end{equation}
 converges to $0$ as $N \to \infty$ by Lemma \ref{MCinequility} and \ref{g_t(x_k)}.
 Combining \eqref{pmainproof1} and \eqref{pmainproof2}, we obtain \eqref{proofmaingoal}.
From Lemma \ref{h1h2}
\begin{equation}{\label{proofmaingoal2}}
\lim_{N\to\infty} \big|(H_{1,N}(\mathbf{x})+H_{2, N}(\mathbf{x}))-\mathbb{E}(H_{1,N}(\mathbf{x})+H_{2,N}(\mathbf{x}))\big|=0
\end{equation}
From \eqref{proofmaingoal} and \eqref{proofmaingoal2}
\begin{equation}{\label{proofmaingoa3}}
\lim_{N\to\infty} \big|(G_{1,N}(\mathbf{x})+G_{2, N}(\mathbf{x}))-\mathbb{E}(H_{1,N}(\mathbf{x})+H_{2,N}(\mathbf{x}))\big|=0
\end{equation}
Using Lemma \ref{exph1h2}
\[
\lim_{N\to\infty}\mathbb{E}[H_{1,N}(\mathbf{x})+H_{2,N}(\mathbf{x})]
=
- \lim_{N\to\infty}N \int \psi_{k,t}(s)
\left[
(f_{\mu_0}^{[t]})_\ast
\left(
\rho_0(s)(f_{\mu_0}^{[t]})'(s)\overline{v_N}
+
g_t(s)\rho_0(s)\bar{v}
\right)
\right]'.
\].

wherea
\[
\psi_{k,t}(s)
=
\int
\psi(x_1,\ldots,x_{k-1},\, s,\, x_{k+1},\ldots,x_N)
\prod_{j\ne k} d\mu_t(x_j).
\]
By noticing that the integrand contains the push-forward of a density with zero average, we can write
\[
\lim_{N\to\infty}\mathbb{E}[H_{1,N}(\mathbf{x})+H_{2,N}(\mathbf{x})]
=
- \bar{v}\lim_{N\to\infty}N \int ( \psi_{k,t}(s)-\mathbb{E}_{\mu_t}({\psi_{k,t}(s)}))
\left[
(f_{\mu_0}^{[t]})_\ast
\left(
\rho_0(s)(f_{\mu_0}^{[t]})'(s)
+
g_t(s)\rho_0(s)
\right)
\right]'ds.
\]
So,
\[\lim_{N\to\infty}\mathbb{E}_{\mu_t}[H_{1,N}(\mathbf{x})+H_{2,N}(\mathbf{x})]
=
- \bar{v}\lim_{N\to\infty}N \int ( \psi_{k,t}(s)-\mathbb{E}_{\mu_t}({\psi_{k,t}(s)}))
\left[
(f_{\mu_0}^{[t]})_\ast
\left(
\rho_0(s)(f_{\mu_0}^{[t]})'(s)
+
g_t(s)\rho_0(s)
\right)
\right]'ds.
\]
From Lemma \ref{relationtodifferential}
\[\lim_{N\to\infty}\mathbb{E}[H_{1,N}(\mathbf{x})+H_{2,N}(\mathbf{x})]
=
- \bar{v}\lim_{N\to\infty}N \int ( \psi_{k,t}(s)-\mathbb{E}_{\mu_t}({\psi_{k,t}(s)}))
D \mathcal{F}_{\mu_{t-1}} \circ \cdots \circ D \mathcal{F}_{\mu_0} (\rho_0')(s)ds.
\]
From \eqref{proofmaingoa3} 
$$\lim_{N\to\infty}|(G_{1,N}(\mathbf{x})+G_{2,N}(\mathbf{x})|=  | \bar{v}| \lim_{N\to\infty}N \Bigg|\int ( \psi_{k,t}(s)-\mathbb{E}_{\mu_t}({\psi_{k,t}(s)}))
D \mathcal{F}_{\mu_{t-1}} \circ \cdots \circ D \mathcal{F}_{\mu_0} (\rho_0')(s)ds \Bigg|$$

\begin{remark}{\label{0(1/N)}}
Note that 
    $\|\psi_{k,t}(s)-\mathbb{E}_{\mu_t}({\psi_{k,t}(s)})\|_{\infty}=O(1/N)$ independent of t, we deduce $$ \lim_{N\to\infty} |(G_{1,N}(\mathbf{x})+G_{2,N}(\mathbf{x}))| \leq  \bigl( C |\bar{v}|\int 
|D \mathcal{F}_{\mu_{t-1}} \circ \cdots \circ D \mathcal{F}_{\mu_0} (\rho_0'(s))|ds \bigr)$$
which gives us 
\[
\limsup_{t\to\infty}\lim_{N\to\infty}\chi_{\psi,N,t}(\mathbf{x},\mathbf{v})
\leq
\limsup_{t\to\infty}
\frac{1}{t}
\log
\left\|
D_{\mu_{t-1}}\mathcal F\circ\cdots\circ D_{\mu_0}\mathcal F(\rho_0')
\right\|_{L^1}.
\]
\end{remark}

\end{proof}

\subsection{Proof of Theorem \ref{Maintheorem2}}

Before proving the main result, we state many auxiliary lemmas whose proofs are provided in Appendix \ref{Appendix}.
Let \(\mu^*\) be a fixed point of \(\mathcal{F}\), and assume that \(\mu^*\) has a strictly positive density \(\rho^*\). By choosing \(\varepsilon>0\) sufficiently small, every measure sufficiently close to \(\mu^*\) in the \(C^0\) (and therefore also on the \(C^1\)) topology has a strictly positive density. Define the neighbourhoods
\[
U_0=
\left\{
\tilde{\mu}\in P_0:
\|\mu^*-\tilde{\mu}\|_{C^0}<\varepsilon
\right\},
\]
and
\[
U_1=
\left\{
\bar{\mu}\in P_1:
\|\mu^*-\bar{\mu}\|_{C^1}<\varepsilon
\right\}.
\]
Then \(U_0\) and \(U_1\) are Banach manifolds. Moreover, their tangent spaces at any point are given by
\[
T_{\tilde{\mu}}U_0=V_0,\qquad
T_{\tilde{\mu}}U_1=V_1 .
\]

\begin{lemma}{\label{unique attracting}}
Let $f_\mu$ be the family of maps defined in Theorem \ref{Maintheorem2}. The corresponding self-consistent operator admits an attracting fixed point  $\mu^* \in C^3$ with density strictly greater than zero and has a local basin of attraction.
\end{lemma}
\begin{proof}
We adapt the invariant-cone construction of
\cite[Theorem 4.1]{castorrini2026stability} to the present additive
coupling. For sufficiently small weak coupling, the associated
self-consistent transfer operator maps a suitable closed, convex cone
of strictly positive \(C^2\) densities into itself, with the defining
cone estimates uniform in the coupling parameter. By the Schauder fixed
point theorem, it therefore admits a fixed point
\(\mu^*=\rho^*\,dx\) with \(\rho^*>0\). The uniform cone estimates give a uniform $C^2$ bound on the fixed density for all kernels $h$ satisfying $\|h\|_{C^3}\leq\delta$, for sufficiently small $\delta$. Since \(f_{\mu^*}\) is a \(C^4\)
uniformly expanding map, standard invariant-density regularity for
uniformly expanding maps implies \(\rho^*\in C^3\). Finally, the
uniform estimates required for local attraction allow
\cite[Lemma 6 and Corollary 7]{castorrini2025differential} to be applied,
once Assumptions~\ref{Fixed point}--\ref{Boundedness} are verified.
We establish these assumptions in Appendix~\ref{Appendix}.
\end{proof}

\begin{lemma}{\label{different frechet}}
Let $\eta \in C^2 \cap U_0$ be a measure in the local basin of attraction of the fixed point $\mu^*$. Define the map
\[
\mathbb{L}:U_0\to C^1,
\qquad
\mathbb{L}(\tilde{\mu}):=(f_{\tilde{\mu}})_*\eta.
\]
Then \(\mathbb{L}\) is Fréchet differentiable at $\eta$. Moreover, its derivative
\[
D_{\eta} \mathbb{L}:V_0\to V_1
\]
satisfies
\[
\|D_\eta \mathbb{L}\|_{V_0\to V_1}
\leq
C\|h\|_{C^2}\|\eta\|_{C^2}.
\]
In particular, for every \(\nu\in V_0\),
\[
\|D_\eta\mathbb{L}(\nu)\|_{C^1}
\leq
C\|h\|_{C^2}\|\eta\|_{C^2}\|\nu\|_{C^0}.
\]
Furthermore, the Fr\'echet derivative of the self-consistent operator $\mathcal{F}$ at $\eta$ exists and is  given by
\[
D_{\eta}\mathcal{F}
=
(f_\eta)_*
+
D_{\eta}\mathbb{L}.
\]
\end{lemma}
\begin{proof}
    See Appendix \ref{Appendix}.
\end{proof}
\begin{lemma}\label{contractionfor theorem}

Let $\mu_0\in C^2$ be an initial distribution, and suppose that $\mu_0$
lies in the local basin of attraction of the fixed point $\mu^*$. Then there
exist $\delta>0$, $J,\tau\in\mathbb N$, and $q_0<1$, independent of
$h$ and $j$, such that, for every $\|h\|_{C^2}\leq\delta$,
\[
\sup_{j\geq J}
\left\|
\left(
f_{\mu_{j+\tau-1}}\circ\cdots\circ f_{\mu_j}
\right)_*
\right\|_{V_1\to V_1}
\leq q_0,
\]
where $\mu_j=\mathcal F^j(\mu_0)$ denotes the $j$-th iterate of $\mu_0$
under the self-consistent transfer operator.
\end{lemma}
\begin{proof}
    See Appendix \ref{Appendix}.
\end{proof}

\begin{proof}[\textbf{Proof of Theorem \ref{Maintheorem2}}]
From Remark \ref{0(1/N)}
\[
\limsup_{t\to\infty}\lim_{N\to\infty}\chi_{\psi,N,t}(\mathbf{x},\mathbf{v})
\leq
\limsup_{t\to\infty}
\frac{1}{t}
\log
\left\|
D_{\mu_{t-1}}\mathcal F\circ\cdots\circ D_{\mu_0}\mathcal F(\rho_{0}')
\right\|_{L^1}.
\]
Therefore, it is enough to prove that the Lyapunov exponent of the cocycle is strictly negative, which is guaranteed if there exist $J,\tau\in\mathbb{N}$ and
$q<1$ such that
\[
\sup_{j\geq J}
\left\|
D_{\mu_{j+\tau-1}}\mathcal F\circ\cdots\circ
D_{\mu_j}\mathcal F
\right\|_{V_1\to V_1}
\leq q.
\]
For $\nu\in V_1$, define
\[
A_m^j(\nu)
:=
D_{\mu_{j+m-1}}\mathcal{F}\circ\cdots\circ
D_{\mu_j}\mathcal{F}(\nu),
\]
and
\[
d_m^j(\nu)
:=
\left\|
(f_{\mu_{j+m-1}}\circ\cdots\circ f_{\mu_j})_*\nu
-
A_m^j(\nu)
\right\|_{V_1}.
\]
For $m=1$, using the decomposition
\[
D_{\mu_j}\mathcal{F}(\nu)
=
(f_{\mu_j})_*\nu+D_{\mu_j}\mathbb{L}(\nu),
\]
we obtain
\[
d_1^j(\nu)
=
\left\|
(f_{\mu_j})_*\nu-D_{\mu_j}\mathcal{F}(\nu)
\right\|_{V_1}
=
\left\|D_{\mu_j}\mathbb{L}(\nu)\right\|_{V_1}.
\]
By the estimate on $D_{\mu_j}\mathbb{L}$,
\[
d_1^j(\nu)
\leq
C\|h\|_{C^2}\|\mu_j\|_{C^2}\|\nu\|_{V_0} \leq
C\|h\|_{C^2}\|\mu_j\|_{C^2}\|\nu\|_{V_1}.
\]
For $m>1$, since
\[
D_{\mu_{j+m-1}}\mathcal{F}
\bigl(A_{m-1}^j(\nu)\bigr)
=
(f_{\mu_{j+m-1}})_*A_{m-1}^j(\nu)
+
D_{\mu_{j+m-1}}\mathbb{L}
\bigl(A_{m-1}^j(\nu)\bigr),
\]
we obtain
\[
d_m^j(\nu)
=
\Bigl\|
(f_{\mu_{j+m-1}})_*
\bigl[
(f_{\mu_{j+m-2}}\circ\cdots\circ f_{\mu_j})_*\nu
-
A_{m-1}^j(\nu)
\bigr]
-
D_{\mu_{j+m-1}}\mathbb{L}
\bigl(A_{m-1}^j(\nu)\bigr)
\Bigr\|_{V_1}.
\]
Therefore,
\[
d_m^j(\nu)
\leq
\|(f_{\mu_{j+m-1}})_*\|_{V_1\to V_1}
d_{m-1}^j(\nu)
+
\left\|
D_{\mu_{j+m-1}}\mathbb{L}
\bigl(A_{m-1}^j(\nu)\bigr)
\right\|_{V_1}.
\]
Using the uniform bound on the pushforward operator and the estimate on
$D_\mu\mathbb{L}$, we obtain
\[
d_m^j(\nu)
\leq
C d_{m-1}^j(\nu)
+
C\|h\|_{C^2}
\|\mu_{j+m-1}\|_{C^2}
\|A_{m-1}^j(\nu)\|_{V_0}.
\]
For $1\leq m\leq\tau$, where $\tau$ is fixed, the derivative
products are bounded. Thus, there exists $C_{m-1}>0$, independent of $j$ and of the
kernel $h$ satisfying $\|h\|_{C^3}\leq\delta$, such that
\[
\|A_{m-1}^j(\nu)\|_{V_0}
\leq C_{m-1}\|\nu\|_{V_1}.
\]
Indeed, the required bounds on the finite products of the
pushforward operators follow from the uniform expansion and the
uniform $C^3$ bounds on the maps $f_{\mu_j}$. %Thus, there exists $C_{m-1}>0$,
%independent of $j$, such that
%\[
%\|A_{m-1}^j(\nu)\|_{V_0}
%\leq 
%C_{m-1}\|\nu\|_{V_1}.
%\]
Together with the uniform $C^2$ bound \footnote{By the same approach used to establish the $C^1$ Lasota--Yorke
inequality in Assumption~\ref{Sequential Lasota–Yorke Inequalities},
and using the higher regularity $f\in C^4(\mathbb T,\mathbb T)$ and
$h\in C^4(\mathbb T^2,\mathbb R)$, the corresponding $C^2$ estimate holds.
In particular, there exist constants $A_2,B_2>0$ and $\alpha_2<1$,
independent of the sequence $\{\mu_i\}$, such that
\[
\left\|
(f_{\mu_n}\circ\cdots\circ f_{\mu_1})_*\nu
\right\|_{C^2}
\leq
A_2\alpha_2^n\|\nu\|_{C^2}
+
B_2\|\nu\|_{C^1}.
\]
Together with the uniform $C^1$ bound on the orbit $(\mu_j)_{j\geq0}$,
this yields a constant $M_2>0$, independent of $j$, such that
$
\sup_{j\geq0}\|\mu_j\|_{C^2}\leq M_2.
$
Hence, for $j\geq J$ and $1\leq m\leq\tau$,
$
\|\mu_{j+m-1}\|_{C^2}\leq M_2.
$}
\[
\sup_{j\geq J}\|\mu_j\|_{C^2}\leq M_2,
\]
this gives
\[
d_m^j(\nu)
\leq
C d_{m-1}^j(\nu)
+
C_m\|h\|_{C^2}\|\nu\|_{V_1},
\]
where $C_m>0$ is independent of $j$, $\nu$, and
$h$ satisfying $\|h\|_{C^3}\leq\delta$.
An induction over $m$ yields
\[
d_m^j(\nu)
\leq K_m\|h\|_{C^2}\|\nu\|_{V_1},
\]
where $K_m>0$ depends only on the uniform bounds above (and on
$m\leq\tau$), and is therefore independent of $j$, $\nu$, and
$h$ satisfying $\|h\|_{C^3}\leq\delta$.
Consequently,
\[
\left\|
D_{\mu_{j+m-1}}\mathcal{F}\circ\cdots\circ
D_{\mu_j}\mathcal{F}
-
(f_{\mu_{j+m-1}}\circ\cdots\circ f_{\mu_j})_*
\right\|_{V_1\to V_1}
=
\sup_{\nu\in V_1\setminus\{0\}}
\frac{d_m^j(\nu)}{\|\nu\|_{V_1}}
\leq
K_m\|h\|_{C^2}.
\]
In particular, for $m=\tau$,
\[
\sup_{j\geq J}
\left\|
D_{\mu_{j+\tau-1}}\mathcal{F}\circ\cdots\circ
D_{\mu_j}\mathcal{F}
-
(f_{\mu_{j+\tau-1}}\circ\cdots\circ f_{\mu_j})_*
\right\|_{V_1\to V_1}
\leq
K_\tau\|h\|_{C^2}.
\]
By Lemma \ref{contractionfor theorem}, with $J,\tau$, and $q_0<1$ chosen uniformly for
$\|h\|_{C^3}\leq\delta$, we have
\[
\sup_{j\geq J}
\left\|
\left(f_{\mu_{j+\tau-1}}\circ\cdots\circ f_{\mu_j}\right)_*
\right\|_{V_1\to V_1}
\leq q_0.
\]
Hence,
\[
\begin{aligned}
\sup_{j\geq J}
\left\|
D_{\mu_{j+\tau-1}}\mathcal F\circ\cdots\circ D_{\mu_j}\mathcal F
\right\|_{V_1\to V_1}
&\leq q_0+K_\tau\|h\|_{C^2}.
\end{aligned}
\]
Choose $\delta>0$ sufficiently small such that
\[
q_0+K_\tau\delta<1.
\]
Since
\[
\|h\|_{C^2}\leq\|h\|_{C^3}\leq\delta,
\]
we obtain
\[
\sup_{j\geq J}
\left\|
D_{\mu_{j+\tau-1}}\mathcal F\circ\cdots\circ D_{\mu_j}\mathcal F
\right\|_{V_1\to V_1}
\leq
q_0+K_\tau\|h\|_{C^2}
\leq
q_0+K_\tau\delta<1.
\]
Since this holds for every $j\geq J$, the cocycle
\[
D_{\mu_{t-1}}\mathcal F\circ\cdots\circ D_{\mu_0}\mathcal F
\]
can be decomposed into a finite initial segment and a sequence of
contracting blocks of length $\tau$. The finite initial segment contributes
only a multiplicative constant, while the contracting blocks yield
exponential decay of the operator norm. Therefore,
\[
\limsup_{t\to\infty}
\frac{1}{t}
\log
\left\|
D_{\mu_{t-1}}\mathcal F
\circ\cdots\circ
D_{\mu_0}\mathcal F
\right\|_{V_1\to V_1}
< 0.
\]
Combining this uniform operator norm bound with the continuous embedding from $V_1$ to $L^1$ for the vector $\rho'$, we conclude that
\[
\limsup_{t\to\infty}\lim_{N\to\infty}\chi_{\psi,N,t}(\mathbf{x},\mathbf{v})
\leq
\limsup_{t\to\infty}
\frac{1}{t}
\log
\left\|
D_{\mu_{t-1}}\mathcal F\circ\cdots\circ D_{\mu_0}\mathcal F(\rho_0')
\right\|_{L^1}
< 0,
\]
which completes the proof.
\end{proof}

\appendix
\section{Appendix}\label{Appendix}

Before proving Lemmas \ref{unique attracting}, \ref{different frechet}, and \ref{contractionfor theorem}, we verify that the family of maps
introduced in Theorem $\ref{Maintheorem2}$ satisfies the assumptions listed in 
\cite[Section 2]{castorrini2025differential}. Once these assumptions are
checked, the statements of the lemmas follow from them. We now state the assumptions and verify that they hold in our setting.
\begin{assumption}[\textbf{Fixed point}] \label{Fixed point}
   There exist $\mu^* \in C^3$  such that $\mathcal{F}(\mu^*)= \mu^*$.
   \end{assumption}
\begin{proof}
 This follows from Lemma \ref{unique attracting}      
\end{proof}.
\begin{assumption}[\textbf{Uniform Boundedness}]{\label{Uniform Boundedness}}
 There exists a constant $C$ such that for each $\mu \in U_{0}$, the operator norms satisfy:
 \begin{equation}
\|(f_\mu)_{*}\|_{C^0 \to C^0} \leq C, \quad \|\ (f_\mu)_{*}\|_{C^1 \to C^1} \leq C, \quad \|(f_\mu)_{*}\|_{C^2 \to C^2} \leq C.
\end{equation}   
\end{assumption}
\begin{proof}
Let \(\nu \in C^k\) be a measure with density \(\rho \in C^k\), where
\(k=0,1,2\). The Perron--Frobenius operator associated with \(f_\mu\) is
given by
\[
\|(f_\mu)_*\nu\|_{C^k}=\|\mathcal L_\mu\rho\|_{C^k}.
\]
Since \(U_0\) is a sufficiently small neighbourhood of \(\mu^*\) and the
map \(\mu\mapsto f_\mu\) depends continuously on \(\mu\), all maps
\(f_\mu\), \(\mu\in U_0\), are homotopic to \(f_{\mu^*}\). Hence their
degrees are identical. Denoting this common degree by \(d\), we have
\[
\deg(f_\mu)=d,\qquad \mu\in U_0 .
\]
Therefore, for every \(x\in\mathbb{T}\), the set \(f_\mu^{-1}(x)\) contains
exactly \(d\) elements.
Using the explicit formula for the Perron--Frobenius operator,
\[
(\mathcal L_\mu \rho)(x)
=
\sum_{y\in f_\mu^{-1}(x)}
\frac{\rho(y)}{|f_\mu'(y)|} .
\]
Since \(f_\mu'\) has constant sign on \(\mathbb T\), the differentiated
formulas below differ from the orientation-preserving case only by a
common sign, which is irrelevant for the \(C^k\)-norm estimates.
Differentiating, we obtain
\[
(\mathcal L_\mu\rho)'(x)
=
\sum_{y\in f_\mu^{-1}(x)}
\left[
\frac{\rho'(y)}{(f_\mu'(y))^2}
-
\frac{\rho(y)f_\mu''(y)}{(f_\mu'(y))^3}
\right].
\]
Similarly,
\[
(\mathcal L_\mu\rho)''(x)
=
\sum_{y\in f_\mu^{-1}(x)}
\left[
\frac{\rho''(y)}{(f_\mu'(y))^3}
-
3\frac{\rho'(y)f_\mu''(y)}{(f_\mu'(y))^4}
+
\rho(y)
\frac{3(f_\mu''(y))^2-f_\mu'(y)f_\mu'''(y)}
     {(f_\mu'(y))^5}
\right].
\]
By the uniform expansion assumption, there exists \(\lambda>1\) such that
\[
\inf_{\mu\in U_0}\inf_{x\in\mathbb{T}} |f_\mu'(x)|\geq \lambda .
\]
Moreover,
\[
\sup_{\mu\in U_0}
\left(
\|f_\mu''\|_{C^0}
+
\|f_\mu'''\|_{C^0}
\right)<\infty .
\]
Since every sum contains exactly \(d\) terms, the above expressions are
uniformly bounded in \(\mu\). Hence there exists \(C>0\), independent of
\(\mu\in U_0\), such that
\[
\|\mathcal L_\mu\rho\|_{C^k}
\leq C\|\rho\|_{C^k},
\qquad k=0,1,2 .
\]
Therefore,
\[
\|(f_\mu)_*\|_{C^k\to C^k}\leq C,
\qquad k=0,1,2.
\]
\end{proof}

\begin{assumption}[\textbf{Sequential Lasota--Yorke Inequalities}]{\label{Sequential Lasota–Yorke Inequalities}}
There exist \(\alpha<1\) and \(A,B>0\) such that, for every
\(n\ge1\), any sequence \(\{\mu_i\}_{i=1}^n\) with
\(\mu_i\in\mathbf{P}_0\), and any \(\nu\in C^1(\mathbb T)\),
\begin{equation}{\label{lasota1}}
\|(f_{\mu _n} \circ \dots \circ f_{\mu_1})_{*}\nu\|_{C^0} \leq   B \| \nu\|_{C^0}.
\end{equation}
\begin{equation}{\label{lasota2}}
\|(f_{\mu _n} \circ \dots \circ f_{\mu_1})_{*}\nu\|_{C^1} \leq A \alpha^n \| \nu\|_{C^1} + B \| \nu\|_{C^0}.
\end{equation} 
\end{assumption}
\begin{proof}
Let
\[
F_n=f_{\mu_n}\circ\cdots\circ f_{\mu_1},
\]
and $\varphi\in C^1$ be the density of $\nu$. Denote by $\mathcal{L}_{n}$ the Perron--Frobenius operator associated with $F_n$, namely
\[
\mathcal{L}_n\varphi(x)
=
\sum_{F_n(y)=x}\frac{\varphi(y)}{|F_n'(y)|}.
\]
Then
\[
\|(F_n)_*\nu\|_{C^1}
=
\|\mathcal{L}_n\varphi\|_{C^1}.
\]
Since
\[
(\mathcal{L}_n\varphi)'(x)
=
\mathcal{L}_n\left(\frac{\varphi'}{F_n'}\right)(x)
-
\mathcal{L}_n\!\left(\varphi\frac{F_n''}{(F_n')^2}\right)(x),
\]
we obtain
\begin{align*}
\|\mathcal{L}_n\varphi\|_{C^1}
&=
\|\mathcal{L}_n\varphi\|_{C^0}
+
\|(\mathcal{L}_n\varphi)'\|_{C^0}\\
&\le
\|\mathcal{L}_n\varphi\|_{C^0}
+
\left\|\mathcal{L}_n\!\left(\frac{\varphi'}{F_n'}\right)\right\|_{C^0}
+
\left\|\mathcal{L}_n\!\left(\varphi\frac{F_n''}{(F_n')^2}\right)\right\|_{C^0}.
\end{align*}
Define
\[
C_n
:=
\sup_x
\sum_{F_n(y)=x}
\frac1{|F_n'(y)|}.
\]
Then
\[
\|\mathcal{L}_n\psi\|_{C^0}
\le
C_n\|\psi\|_{C^0}
\]
for every continuous function $\psi$. Hence
\begin{align*}
\|\mathcal{L}_n\varphi\|_{C^1}
&\leq
C_n\|\varphi\|_{C^0}
+
C_n\left\|\frac{\varphi'}{F_n'}\right\|_{C^0}
+
C_n\left\|\varphi\frac{F_n''}{(F_n')^2}\right\|_{C^0}.
\end{align*}
Since the family $\{f_{\mu_i}\}$ is uniformly expanding and $C^2$,
It follows that, for every composition
\[
F_n=f_{\mu_n}\circ\cdots\circ f_{\mu_1},
~~~\text{we have}~~~
|F_n'(x)|\ge \lambda^n
\]
and there exists a constant $B_0>0$, independent of $n$ and of
$\mu_1,\dots,\mu_n$, such that
\[
|F_n'(x)|\ge \lambda^n,
\qquad
\left\|\frac{F_n''}{(F_n')^2}\right\|_{C^0}\le B_0,
\]
we obtain
\[
\left\|\frac{\varphi'}{F_n'}\right\|_{C^0}
\le
\lambda^{-n}\|\varphi'\|_{C^0},
~~\text{and}~~
\left\|\varphi\frac{F_n''}{(F_n')^2}\right\|_{C^0}
\le
B_0\|\varphi\|_{C^0}.
\]
Therefore
\[
\|\mathcal{L}_n\varphi\|_{C^1}
\le
C_n\lambda^{-n}\|\varphi'\|_{C^0}
+
C_n(1+B_0)\|\varphi\|_{C^0}.
\]
Since $\|\varphi'\|_{C^0}\le \|\varphi\|_{C^1}$, we get
\[
\|\mathcal{L}_n\varphi\|_{C^1}
\le
C_n\lambda^{-n}\|\varphi\|_{C^1}
+
C_n(1+B_0)\|\varphi\|_{C^0}.
\]
It remains to show that $C_n$ is bounded independently of $n$.
Let $\{h_j\}$ denote the inverse branches of $F_n$. Then
\[
C_n
=
\sup_x \sum_j |h_j'(x)|
~~\text{and}~~
\frac{d}{dx}\log |h_j'(x)|
=
-\frac{F_n''(h_j(x))}
{\big(F_n'(h_j(x))\big)^2}.~~
\frac{d}{dx}\log |h_j'(x)|
=
-\frac{F_n''}{(F_n')^2},
\]
and hence
\[
\left|
\frac{d}{dx}\log |h_j'(x)|
\right|
\le B_0.
\]
Integrating  from $z$ to $x$ yields the bounded distortion estimate
\[
|h_j'(x)|
\le
e^{B_0}|h_j'(z)|
\qquad
\forall x,z.
\]
Therefore
\[
\sum_j |h_j'(x)|
\le
e^{B_0}\sum_j |h_j'(z)|.
\]
Integrating with respect to $z$ over the torus,
\[
\sum_j |h_j'(x)|
\le
e^{B_0}
\int \sum_j |h_j'(z)|\,dz.
\]
Since the images of the inverse branches partition the Torus,
\[
\int \sum_j |h_j'(z)|\,dz
=
\sum_j \int |h_j'(z)|\,dz
=
\sum_j |h_j(\mathbb{T})|
=
1.
\]
Hence
\[
C_n\le e^{B_0},
\]
uniformly in $n$. Thus, taking
\[
A=e^{B_0},\qquad
B=e^{B_0}(1+B_0),\qquad
\alpha=\lambda^{-1},
\]
with \(0<\alpha<1\), the desired inequality follows.
\end{proof}

\begin{assumption}[\textbf{Convergence to Equilibrium}]{\label{Convergence to Equilibrium}}
    For the operator $(f_{\mu^*})_{*}$ associated with the fixed point, there exists a sequence $a_n \to 0$ such that for all $\mu \in {V}_1$:
\begin{equation}
\|(f^{t}_{\mu^*})_{*}\mu\|_{C^1 } \leq a_t \|\mu\|_{C^1}.
\end{equation}
\begin{proof}
    The map $f_{\mu^*}$ is a uniformly expanding $C^2$ map,  the space $C^1$ is compactly embedded in $C^0$, and the transfer operator $(f_{\mu^*})_{*}$ satisfies a Lasota-Yorke inequality; it holds that for the operator associated with the fixed point, there exists a sequence $a_{t} \to 0$ such that for all $\mu \in V_{1}$:
\begin{equation}
\|(f_{\mu^*}^{t})_{*} \mu\|_{C^1} \leq a_{t} \|\mu\|_{C^1}
\end{equation}
For details, see
\cite[Theorem 30]{galatolo2022statisticalpropertiesdynamicsintroduction}.
\end{proof}
\end{assumption}
\begin{assumption}[\textbf{Regularity of the Operator Family}]
{\label{Regularity of the Operator Family}}
\begin{enumerate}[label=(\roman*)]

\item There exists $C_0 > 0$ such that for any $\mu_1, \mu_2 \in U_1$
and the fixed point $\mu^* \in C^2$,
\begin{equation}{\label{regularity2}}
\|(f_{\mu_1})_{*}\mu^* -(f_{\mu_2})_{*}\mu^*\|_{C^1}
\leq
C_0 \|h\|_{C^2}\|\mu^*\|_{C^2}
\|\mu_1 - \mu_2\|_{C^0}.
\end{equation}
\item There exists $C_1 > 0$ such that for any $\mu \in C^1$
and $\nu_1, \nu_2 \in U_0$,
\begin{equation}{\label{regularity2}}
\|(f_{\nu_1})_{*}\mu - (f_{\nu_2})_{*}\mu\|_{C^0}
\leq
C_1 \|h\|_{C^2}\|\mu\|_{C^1}
\|\nu_1 - \nu_2\|_{C^0}.
\end{equation}

\end{enumerate}
\end{assumption}

\begin{proof}
Let $\rho$ denote the density of $\mu^*$ and, for $j=1,2$, let
$g_{j,i}=f_{\mu_j,i}^{-1}$ be the inverse branches of $f_{\mu_j}$.
Then, using the Frobenius operator,
\[
\mathcal L_{\mu_j}\rho(x)
=
\sum_i
\frac{\rho(g_{j,i}(x))}
     {|f_{\mu_j}'(g_{j,i}(x))|}.
\]
Moreover,
\[
(\mathcal L_{\mu_j}\rho)'(x)
=
\sum_i
\left(
\frac{\rho'(g_{j,i}(x))}
     {(f_{\mu_j}'(g_{j,i}(x)))^2}
-
\frac{\rho(g_{j,i}(x))f_{\mu_j}''(g_{j,i}(x))}
     {(f_{\mu_j}'(g_{j,i}(x)))^3}
\right).
\]
Hence
$$\|\mathcal L_{\mu_1}\rho-\mathcal L_{\mu_2}\rho\|_{C^1}
 \le
\left\|
\sum_i
\left(
\frac{\rho(g_{1,i})}{f_{\mu_1}'(g_{1,i})}-
\frac{\rho(g_{2,i})}{f_{\mu_2}'(g_{2,i})}
\right)\right\|_{\infty}+
\left\|
\sum_i
\left(
\frac{\rho'(g_{1,i})}{(f_{\mu_1}'(g_{1,i}))^2}
-
\frac{\rho'(g_{2,i})}{(f_{\mu_2}'(g_{2,i}))^2}
\right)
\right\|_\infty
$$
$$+\left\|
\sum_i
\left(
\frac{\rho(g_{1,i})f_{\mu_1}''(g_{1,i})}
{(f_{\mu_1}'(g_{1,i}))^3}
-\frac{\rho(g_{2,i})f_{\mu_2}''(g_{2,i})}{(f_{\mu_2}'(g_{2,i}))^3}
\right)
\right\|_\infty$$

For the first term, we write
\[
\frac{\rho(g_{1,i})}{f_{\mu_1}'(g_{1,i})}
-
\frac{\rho(g_{2,i})}{f_{\mu_2}'(g_{2,i})}
=
\frac{\rho(g_{1,i})-\rho(g_{2,i})}
{f_{\mu_1}'(g_{1,i})}
+
\rho(g_{2,i})
\left(
\frac{1}{f_{\mu_1}'(g_{1,i})}
-
\frac{1}{f_{\mu_2}'(g_{2,i})}
\right).
\]
Denote the first, second and third terms above by $A$, $B$, and $C$,
respectively. Using the uniform expansion, we obtain
$$
\|A\|
\leq
\sum_i
\left\|
\frac{\rho(g_{1,i})-\rho(g_{2,i})}
{f_{\mu_1}'(g_{1,i})}
\right\|_\infty+
\sum_i
\left\|
\rho(g_{2,i})
\left(
\frac{1}{f_{\mu_1}'(g_{1,i})}
-
\frac{1}{f_{\mu_2}'(g_{2,i})}
\right)
\right\|_\infty $$$$\leq
\sum_i
\frac{\|\rho'\|_\infty}{\lambda}
\|g_{1,i}-g_{2,i}\|_\infty
+
\sum_i
\frac{\|\rho\|_\infty}{\lambda^2}
\|f_{\mu_1}'(g_{1,i})
-
f_{\mu_2}'(g_{2,i})\|_\infty.
$$
We first prove the required estimates on the inverse branches.
Since
$
f_\mu(x)
=
f(x)+\int h(x,y)\,d\mu(y),
$
we have
\[
f_{\mu_1}(x)-f_{\mu_2}(x)
=
\int h(x,y)\,d(\mu_1-\mu_2)(y).
\]
Hence
$
\|f_{\mu_1}-f_{\mu_2}\|_{C^1}
\leq
\|h\|_{C^1}\|\mu_1-\mu_2\|_{C^0}.
$
Let $g_{j,i}$ be the $i$-th inverse branch of $f_{\mu_j}$, so that
\[
f_{\mu_j}(g_{j,i}(x))=x,
\qquad j=1,2.
\]
Then
$
f_{\mu_1}(g_{1,i}(x))
=
f_{\mu_2}(g_{2,i}(x)).
$
Therefore
$
f_{\mu_1}(g_{1,i}(x))
-
f_{\mu_1}(g_{2,i}(x))
=
f_{\mu_2}(g_{2,i}(x))
-
f_{\mu_1}(g_{2,i}(x)).
$
By uniform expansion and the mean value theorem\footnote{
Throughout the proof, we choose compatible lifts of the maps
$f_{\mu_j}:\mathbb{T}\to\mathbb{T}$ to maps
$F_{\mu_j}:\mathbb{R}\to\mathbb{R}$ of the same degree.
The inverse branches $g_{j,i}$ are understood as the corresponding
lifted inverse branches on a fundamental interval. All estimates
involving differences of inverse branches are performed on these lifts.
},
we get
\[
\lambda |g_{1,i}(x)-g_{2,i}(x)|
\leq
\|f_{\mu_1}-f_{\mu_2}\|_{C^0}.
\]
Thus
\[
\|g_{1,i}-g_{2,i}\|_{\infty}
\leq
\frac{1}{\lambda}
\|f_{\mu_1}-f_{\mu_2}\|_{C^0}
\leq
\frac{1}{\lambda}
\|h\|_{C^1}
\|\mu_1-\mu_2\|_{C^0}.
\]
Next,
$$
\|f_{\mu_1}'\circ g_{1,i}
-
f_{\mu_2}'\circ g_{2,i}\|_\infty
\leq
\|f_{\mu_1}'\circ g_{1,i}
-
f_{\mu_1}'\circ g_{2,i}\|_\infty
\quad+
\|f_{\mu_1}'\circ g_{2,i}
-
f_{\mu_2}'\circ g_{2,i}\|_\infty.
$$
Since the family is uniformly $C^2$, there exists $M>0$ such that
\[
\sup_{\mu\in U_1}\|f_\mu''\|_\infty\leq M.
\]
Hence $f_\mu'$ is uniformly Lipschitz and
\[
\|f_{\mu_1}'\circ g_{1,i}
-
f_{\mu_1}'\circ g_{2,i}\|_\infty
\leq
M\|g_{1,i}-g_{2,i}\|_\infty.
\]
Also,
\[
\|f_{\mu_1}'\circ g_{2,i}
-
f_{\mu_2}'\circ g_{2,i}\|_\infty
\leq
\|f_{\mu_1}'-f_{\mu_2}'\|_\infty
\leq
\|h\|_{C^1}\|\mu_1-\mu_2\|_{C^0}.
\]
Combining the estimates gives
\[
\|f_{\mu_1}'\circ g_{1,i}
-
f_{\mu_2}'\circ g_{2,i}\|_\infty
\leq
C\|h\|_{C^1}\|\mu_1-\mu_2\|_{C^0},
\]
where $C$ is a generic constant independent of $\mu_1,\mu_2$.
Therefore,
\[
\|A\|
\leq
C\|h\|_{C^1}\|\rho\|_{C^1}
\|\mu_1-\mu_2\|_{C^0}.
\]
The terms $B$ and $C$ are estimated in the same way, using the
uniform $C^2$ bounds on the family $\{f_\mu\}$.
For the second term, we write
\[
\frac{\rho'(g_{1,i})}{(f_{\mu_1}'(g_{1,i}))^2}
-
\frac{\rho'(g_{2,i})}{(f_{\mu_2}'(g_{2,i}))^2}
\]
as the sum of the difference in the numerators and the difference in
the denominators. Using the estimate on the inverse branches and the
uniform $C^2$ bounds, we obtain
\[
\|B\|
\leq
C\|h\|_{C^1}\|\rho\|_{C^2}
\|\mu_1-\mu_2\|_{C^0}.
\]
Similarly, for the third term,
\[
\frac{\rho(g_{1,i})f_{\mu_1}''(g_{1,i})}
{(f_{\mu_1}'(g_{1,i}))^3}
-
\frac{\rho(g_{2,i})f_{\mu_2}''(g_{2,i})}
{(f_{\mu_2}'(g_{2,i}))^3}
\]
is decomposed into differences of $\rho$, $f_\mu''$, and the
denominators. Since
\[
\|f_{\mu_1}''-f_{\mu_2}''\|_\infty
\leq
\|h\|_{C^2}\|\mu_1-\mu_2\|_{C^0},
\]
we obtain
\[
\|C\|
\leq
C\|h\|_{C^2}\|\rho\|_{C^2}
\|\mu_1-\mu_2\|_{C^0}.
\]
Combining the three estimates yields
\[
\|\mathcal L_{\mu_1}\rho
-
\mathcal L_{\mu_2}\rho\|_{C^1}
\leq
C_0\|h\|_{C^2}
\|\rho\|_{C^2}
\|\mu_1-\mu_2\|_{C^0}.
\]
Since $\rho$ is the density of $\mu^*$, this proves
\[
\|(f_{\mu_1})_*\mu^*
-
(f_{\mu_2})_*\mu^*\|_{C^1}
\leq
C_0\|h\|_{C^2}
\|\mu^*\|_{C^2}
\|\mu_1-\mu_2\|_{C^0}.
\]
For \eqref{regularity2}, let $\mu\in C^1$ and let $\rho$ denote its density.
Since the family $\{f_\nu:\nu\in U_0\}$ is uniformly bounded in
$C^2$, there exists $M>0$ such that
\[
\sup_{\nu\in U_0}\|f_\nu''\|_\infty\leq M.
\]
Thus $f_\nu'$ is uniformly Lipschitz. Moreover,
$
\|f_{\nu_1}-f_{\nu_2}\|_{C^1}
\leq
\|h\|_{C^1}\|\nu_1-\nu_2\|_{C^0}.
$
Applying the same inverse-branch estimates as above, but only at the
$C^0$ level, gives
\[
\|\mathcal L_{\nu_1}\rho-\mathcal L_{\nu_2}\rho\|_{C^0}
\leq
C_1\|h\|_{C^2}\|\rho\|_{C^1}
\|\nu_1-\nu_2\|_{C^0}.
\]
Since $\rho$ is the density of $\mu$, we obtain
\[
\|(f_{\nu_1})_*\mu-(f_{\nu_2})_*\mu\|_{C^0}
\leq
C_1\|h\|_{C^2}\|\mu\|_{C^1}
\|\nu_1-\nu_2\|_{C^0}.
\]
This proves \eqref{regularity2}.
\end{proof}

\begin{assumption}[\textbf{Differentiability}]{\label{Differentiability}}
Let $\eta \in C^2 \cap U_0$ be a measure. The map $
\mathbb{L} : \mathbf{U}_0 \to \mathbf{C^1}
$
defined by
$
\mathbb{L}(\tilde{\mu}) = (f_{\tilde{\mu}})_{*}\eta$
is Fréchet differentiable at ${\eta}$
means that there exists a bounded linear operator
\[
D_{\eta}\mathbb{L} : V_0 \to V_1
\]
such that   
\begin{equation}{\label{Differentiabilityequ}}
\lim_{\substack{\|\nu\|_{C^0}\to 0\\ \nu\in V_0}}
\frac{
\left\| 
(f_{\eta+\nu})_{*} \eta - (f_{\eta})_{*}\eta - D_{\eta}\mathbb{L}(\nu)
\right\|_{C^1}
}{
\|\nu\|_{C^0}
}
=0.
\end{equation}
\end{assumption}
\begin{proof}
We first show that $\mathbb{L}$ is Gâteaux differentiable at $\eta$. Let
$\nu\in V_{0}$ (the space of zero-average measures). Consider

\[
D_{\eta}(\nu)
=
\lim_{\varepsilon\to0}
\frac{
(f_{\eta+\varepsilon\nu})_{*}\eta
-
(f_{\eta})_{*}\eta
}{\varepsilon}.
\]
Let $\eta(x)$ denote the density of $\eta$ and let
$\varphi \in C^{\infty}(\mathbb{T}, \mathbb{R})$. Then

\[
\int \varphi(x)
\Bigl(
(f_{\eta+\varepsilon\nu})_{*}\eta(x)
-
(f_{\eta})_{*}\eta(x)
\Bigr)\,dx
=
\int \eta(x)
\Bigl(
\varphi(f_{\eta+\varepsilon\nu}(x))
-
\varphi(f_{\eta}(x))
\Bigr)\,dx.
\]
Define
\[
G_x(r)=\varphi(f_{\eta+r\varepsilon\nu}(x)).
\]
Then, by the fundamental theorem of calculus,
\[
G_x(1)-G_x(0)=\int_0^1 G_x'(r)\,dr .
\]
Now,
\[
\frac{d}{dr}
\varphi(f_{\eta+r\varepsilon\nu}(x))
=
\varepsilon
\varphi'(f_{\eta+r\varepsilon\nu}(x))
\int_{\mathbb T}h(x,y)d\nu(y).
\]
Therefore,
\[
\int_{\mathbb{T}} \eta(x)
\left(
\varphi(f_{\eta+\varepsilon\nu}(x))
-
\varphi(f_{\eta}(x))
\right)dx=
\varepsilon
\int_{\mathbb T}
\int_0^1
\eta(x)
\varphi'(f_{\eta+r\varepsilon\nu}(x))
\left(
\int_{\mathbb T}h(x,y)d\nu(y)
\right)
\,dr\,dx .
\]
The integrand is continuous on the compact set 
\(\mathbb T\times[0,1]\), and hence it is integrable. 
Therefore, by Fubini's theorem, we can interchange the order of integration:
\[
\int_{\mathbb{T}}\varphi(x)
\left(
(f_{\eta+\varepsilon\nu})_{*}\eta(x)
-
(f_{\eta})_{*}\eta(x)
\right)dx
=
\varepsilon
\int_0^1
\int_{\mathbb{T}}
\eta(x)
\varphi'(f_{\eta+r\varepsilon\nu}(x))
\left(
\int_{\mathbb T}h(x,y)d\nu(y)
\right)
dx\,dr.
\]
Using the chain rule with respect to the variable $x$ and integrating by parts, 
\[
\int\varphi(x)
\left(
(f_{\eta+\varepsilon\nu})_{*}\eta(x)
-
(f_{\eta})_{*}\eta(x)
\right)dx
= -\varepsilon
\int_0^1
\int_{\mathbb{T}}
\varphi(x)
(f_{\eta+r\varepsilon\nu})_{*}
\left(
\frac{
\eta(x)
\left(
\int_{\mathbb T}h(x,y)d\nu(y)
\right)
}
{f'_{\eta+r\varepsilon\nu}(x)}
\right)'
dxdr.
\]
Therefore,
\[
\frac{
(f_{\eta+\varepsilon\nu})_{*}\eta
-
(f_{\eta})_{*}\eta
}{\varepsilon}
=-\int_0^1
(f_{\eta+r\varepsilon\nu})_{*}
\left(
\frac{
\eta(x)
\left(
\int_{\mathbb T}h(x,y)d\nu(y)
\right)
}
{f'_{\eta+r\varepsilon\nu}(x)}
\right)'dr
\]
Moreover,
\[
\sup_{r\in[0,1]}
\|f_{\eta+r\epsilon\nu}-f_\eta\|_{C^3}
\leq C|\epsilon|\|\nu\|_{C^0}
\longrightarrow 0.
\]
Hence
$
f'_{\eta+r\epsilon\nu}\longrightarrow f'_\eta
\quad\text{in }C^2(\mathbb T),
$
uniformly for \(r\in[0,1]\).
Therefore
\[
\left(
\frac{
\eta(x)\displaystyle\int_{\mathbb T}h(x,y)d\nu(y)
}
{f'_{\eta+r\varepsilon\nu}(x)}
\right)'
\longrightarrow
\left(
\frac{
\eta(x)\displaystyle\int_{\mathbb T}h(x,y)d\nu(y)
}
{f'_{\eta}(x)}
\right)'
\]
in \(C^1(\mathbb T)\), uniformly for \(r\in[0,1]\). This follows from
the \(C^3\) convergence above and the explicit \(C^1\)
transfer-operator formula, since the inverse branches and their first
two derivatives converge uniformly in \(r\). Thus
\[
\sup_{r\in[0,1]}
\left\|
(f_{\eta+r\epsilon\nu})_*
\left(
\frac{\eta(x)\int_{\mathbb T}h(x,y)\,d\nu(y)}
{f'_{\eta+r\epsilon\nu}(x)}
\right)'
-
(f_\eta)_*
\left(
\frac{\eta(x)\int_{\mathbb T}h(x,y)\,d\nu(y)}
{f'_\eta(x)}
\right)'
\right\|_{C^1}
\longrightarrow 0.
\]
Hence,
\[
\left\|
\int_0^1
(f_{\eta+r\varepsilon\nu})_*
\left(
\frac{
\eta(x)\displaystyle\int_{\mathbb T}h(x,y)d\nu(y)
}
{f'_{\eta+r\varepsilon\nu}(x)}
\right)'dr
-
\int_0^1
(f_{\eta})_*
\left(
\frac{
\eta(x)\displaystyle\int_{\mathbb T}h(x,y)d\nu(y)
}
{f'_{\eta}(x)}
\right)'dr
\right\|_{C^{1}}
\longrightarrow0 .
\]
Since the second integrand is independent of $r$,
it implies,
\[
\frac{
(f_{\eta+\varepsilon\nu})_*\eta
-
(f_{\eta})_*\eta
}{\varepsilon}
\longrightarrow
-
(f_{\eta})_*
\left(
\frac{
\eta(x)\displaystyle\int_{\mathbb T}h(x,y)d\nu(y)
}
{f'_{\eta}(x)}
\right)'
\]
in $C^{1}(\mathbb T)$.
The candidate derivative is given by
\begin{equation}{\label{Exactdiffformula}}
D_{\eta}\mathbb{L}(\nu)
=
-(f_{\eta})_{*}
\left(
\frac{
\eta(x)\displaystyle\int_{\mathbb T}h(x,y)d\nu(y)
}
{f'_{\eta}(x)}
\right)' .
\end{equation}
The operator $D_{\eta}\mathbb{L}$ is linear in $\nu$, since the dependence on
$\nu$ appears only through the linear term
\[
\nu\mapsto \int_{\mathbb T}h(x,y)d\nu(y).
\]
Moreover, by the regularity of $h$ and $\eta\in C^2$, and the uniform
expansion of $f_\eta$, there exists $C>0$ such that
\[
\|D_{\eta}\mathbb{L}(\nu)\|_{C^1}
\leq
C\|h\|_{C^2}\|\eta\|_{C^2}\|\nu\|_{C^0}.
\]
Therefore,
\[
D_{\eta}\mathbb{L}:V_0\to V_1
\]
is a bounded linear operator.
It remains to prove the Fréchet differentiability condition. Let
\[
R(\nu)=
(f_{\eta+\nu})_{*}\eta-(f_{\eta})_{*}\eta-D_{\eta}\mathbb{L}(\nu).
\]
Using the integral representation obtained above,
\[
R(\nu)=
-\int_0^1
\left[
(f_{\eta+r\nu})_{*}
\left(
\frac{
\eta(x)\displaystyle\int_{\mathbb T}h(x,y)d\nu(y)}
{f'_{\eta+r\nu}(x)}
\right)'
-
(f_{\eta})_{*}
\left(
\frac{
\eta(x)\displaystyle\int_{\mathbb T}h(x,y)d\nu(y)}
{f'_{\eta}(x)}
\right)'
\right]dr .
\]
Define
$
H_\nu(x):=\int_{\mathbb T}h(x,y)\,d\nu(y).
$
By the regularity of \(h\),
$
\|H_\nu\|_{C^2}\leq C\|\nu\|_{C^0}.
$
The remainder can be written as

$$
R(\nu)
=
-\int_0^1
\left[
(f_{\eta+r\nu})_*
\left(
\frac{\eta H_\nu}{f'_{\eta+r\nu}}
\right)'
-
(f_\eta)_*
\left(
\frac{\eta H_\nu}{f'_\eta}
\right)'
\right]\,dr.
$$
We split the integrand as
$$
(f_{\eta+r\nu})_*
\left(
\frac{\eta H_\nu}{f'_{\eta+r\nu}}
\right)'
-
(f_\eta)_*
\left(
\frac{\eta H_\nu}{f'_\eta}
\right)'
={}
\left[(f_{\eta+r\nu})_*-(f_\eta)_*\right]
\left(
\frac{\eta H_\nu}{f'_{\eta+r\nu}}
\right)'
+
(f_\eta)_*
\left[
\left(
\frac{\eta H_\nu}{f'_{\eta+r\nu}}
\right)'
-
\left(
\frac{\eta H_\nu}{f'_\eta}
\right)'
\right].
$$
Since
$$
f_{\eta+r\nu}(x)-f_\eta(x)
=
r\int_{\mathbb T}h(x,y)\,d\nu(y),
$$
we have
$$
\sup_{r\in[0,1]}
\|f_{\eta+r\nu}-f_\eta\|_{C^3}
\leq C\|\nu\|_{C^0}.
$$
Moreover,
$$
\left\|
\left(
\frac{\eta H_\nu}{f'_{\eta+r\nu}}
\right)'
\right\|_{C^2}
\leq C\|H_\nu\|_{C^2}
\leq C\|\nu\|_{C^0},
$$
uniformly for \(r\in[0,1]\). Hence, by the continuity estimate for the transfer operator,
$$
\left\|
\left[(f_{\eta+r\nu})_*-(f_\eta)_*\right]
\left(
\frac{\eta H_\nu}{f'_{\eta+r\nu}}
\right)'
\right\|_{C^1}
\leq
C
\|f_{\eta+r\nu}-f_\eta\|_{C^3}
\left\|
\left(
\frac{\eta H_\nu}{f'_{\eta+r\nu}}
\right)'
\right\|_{C^2}
\leq
C\|\nu\|_{C^0}^2.
$$
For the second term, uniform expansion implies that
\(f'_{\eta+r\nu}\) and \(f'_\eta\) are uniformly bounded away from zero for \(\|\nu\|_{C^0}\) sufficiently small. Consequently,

$$
\left\|
\frac{1}{f'_{\eta+r\nu}}
-
\frac{1}{f'_\eta}
\right\|_{C^2}
\leq
C
\|f_{\eta+r\nu}-f_\eta\|_{C^3}
\leq C\|\nu\|_{C^0}.
$$
Together with
$$
\|H_\nu\|_{C^2}\leq C\|\nu\|_{C^0},
$$
the product and differentiation estimates give
$$
\left\|
\left(
\frac{\eta H_\nu}{f'_{\eta+r\nu}}
\right)'
-
\left(
\frac{\eta H_\nu}{f'_\eta}
\right)'
\right\|_{C^1}
\leq C\|\nu\|_{C^0}^2.
$$
Since \((f_\eta)_*:C^2\to C^1\) is bounded, it follows that
$$
\left\|
(f_\eta)_*
\left[
\left(
\frac{\eta H_\nu}{f'_{\eta+r\nu}}
\right)'
-
\left(
\frac{\eta H_\nu}{f'_\eta}
\right)'
\right]
\right\|_{C^1}
\leq C\|\nu\|_{C^0}^2.
$$
Therefore, uniformly for \(r\in[0,1]\),
$$
\left\|
(f_{\eta+r\nu})_*
\left(
\frac{\eta H_\nu}{f'_{\eta+r\nu}}
\right)'
-
(f_\eta)_*
\left(
\frac{\eta H_\nu}{f'_\eta}
\right)'
\right\|_{C^1}
\leq C\|\nu\|_{C^0}^2.
$$
Hence
$$
\|R(\nu)\|_{C^1}
\leq
\int_0^1 C\|\nu\|_{C^0}^2\,dr
=
C\|\nu\|_{C^0}^2.
$$
Consequently,
$$
\frac{
\|(f_{\eta+\nu})_*\eta-(f_\eta)_*\eta-D_\eta L(\nu)\|_{C^1}
}{
\|\nu\|_{C^0}
}
\leq
C\|\nu\|_{C^0}
\longrightarrow0
\qquad\text{as }\|\nu\|_{C^0}\to0.
$$
Thus \(\mathbb{L}\) is Fréchet differentiable at \(\eta\), with
$$
D_\eta \mathbb{L}(\nu)
=
-(f_\eta)_*
\left(
\frac{\eta H_\nu}{f'_\eta}
\right)'.
$$
\end{proof}
\begin{assumption}[\textbf{Boundedness}]{\label{Boundedness}}
The operator
\[
D_{\mu^{*}}\mathbb{L} : V_0 \to C^2
\] is bounded.
\end{assumption}
\begin{proof}
    Since \(f_{\mu^*} \in C^4\), the fixed point density belongs to \(C^3(\mathbb{T}, \mathbb{R})\), the rest follows from $\eqref{Exactdiffformula}$.
\end{proof}

\begin{proof}[\textbf{Proof of lemma \ref{different frechet}}]

From assumption $\ref{Differentiability}$ it follows that $\mathbb{L}$ is Fréchet differentiable at $\eta$ 
and its derivative
\[
D_\eta\mathbb{L}:V_0\to V_1
\]
satisfies
\[
\|D_\eta\mathbb{L}\|_{V_0\to V_1}
\leq C\|h\|_{C^2}\|\eta\|_{C^2}.
\]
It remains to prove that the self-consistent transfer operator $\mathcal{F}$ is
Fréchet differentiable at $\eta$ from the strong space to the weak space.
We take as a candidate for the derivative
\[
D_\eta\mathcal{F}:V_1 \to V_1 ~~\text{such that}~~ D_\eta\mathcal{F}=(f_\eta)_*+D_\eta\mathbb{L}.
\]
Let $\nu \in V_1$. We write
\[
\begin{aligned}
\mathcal{F}(\eta+\nu)-\mathcal{F}(\eta)-D_\eta\mathcal{F}(\nu)
=&(f_{\eta+\nu})_*(\eta+\nu)
-(f_\eta)_*\eta
-(f_\eta)_*\nu
-D_\eta\mathbb{L}(\nu).
\end{aligned}
\]
Expanding the pushforward of $\eta + \nu$ by linearity and regrouping
\[
\mathcal{F}(\eta+\nu)-\mathcal{F}(\eta)-D_\eta\mathcal{F}(\nu)=
(f_{\eta+\nu})_*\eta-(f_\eta)_*\eta
-D_\eta\mathbb{L}(\nu)+(f_{\eta+\nu})_*\nu-(f_\eta)_*\nu .
\]
By the Fréchet differentiability of $\mathbb{L}$ from $V_0$ to $V_1$,
\[
\frac{
\|(f_{\eta+\nu})_*\eta-(f_\eta)_*\eta
-D_\eta\mathbb{L}(\nu)\|_{C^0}}
{\|\nu\|_{C^1}}
\to 0 .
\]
Moreover, from \eqref{regularity2}
\[
\frac{
\|(f_{\eta+\nu})_*\nu-(f_\eta)_*\nu\|_{C^0}}
{\|\nu\|_{C^1}}
\to 0 .
\]
Therefore,
\[
\lim_{\|\nu\|_{C^1}\to0}
\frac{
\|\mathcal{F}(\eta+\nu)-\mathcal{F}(\eta)
-D_\eta\mathcal{F}(\nu)\|_{C^0}}
{\|\nu\|_{C^1}}
=0,
\]
which proves that $\mathcal{F}$ is Fréchet differentiable at $\eta$
from the strong space $V_1$ to the weak space $V_0$\footnote{Although $\mathcal{F}$ is Fréchet differentiable at $\eta$ in the
strong-to-weak sense, its derivative is a bounded linear operator
\[
D_\eta\mathcal{F}:V_1\to V_1 .
\]}.
\end{proof}

\begin{proof}[\textbf{Proof of Lemma \ref{contractionfor theorem}}]
The result follows from \cite[Lemma 19]{castorrini2025differential}, once we
verify its assumptions uniformly for sufficiently small couplings. We first note that the sequential Lasota--Yorke inequalities in
Assumption~\ref{Sequential Lasota–Yorke Inequalities} hold uniformly for
sufficiently small couplings. Indeed, the proof of
Assumption~\ref{Sequential Lasota–Yorke Inequalities} uses only the uniform
expansion bound and uniform $C^2$ bounds on the family $\{f_\mu\}$. By taking
$\delta>0$ sufficiently small, these bounds can be chosen uniformly for all
$\|h\|_{C^3}\leq\delta$. Hence, the constants $A$, $B$, and
$\alpha<1$ in \eqref{lasota1} and \eqref{lasota2} are independent of $h$, of
the sequence $(\mu_i)$, and of its length.
\noindent The convergence-to-equilibrium estimate in
Assumption~\ref{Convergence to Equilibrium} can likewise be chosen uniformly
for $\|h\|_{C^3}\leq\delta$. 
Indeed, the uniform Lasota--Yorke inequality, together with the compact embedding \(C^1(\mathbb T)\hookrightarrow C^0(\mathbb T)\) and the uniform expansion of the family \(\{f_{\mu_h^*}\}\), gives a uniform spectral gap for the corresponding transfer operators on \(C^1(\mathbb T)\). Consequently, there exists a sequence \((a_t)_{t\geq1}\), independent of \(h\), with \(a_t\to0\), such that
$$
\|(f_{\mu_h^*}^t)_*\nu\|_{C^1}
\leq a_t\|\nu\|_{C^1},
\qquad \nu\in V_1,\quad \|h\|_{C^3}\leq\delta.
$$
Thus, the convergence-to-equilibrium condition also holds uniformly\footnote{For sufficiently small \(\|h\|_{C^3}\), the frozen maps \(f_{\mu_h^*}\) remain uniformly expanding and their Lasota--Yorke constants can be chosen independently of \(h\). Consequently, after reducing \(\delta\) if necessary, the decay constants \(C>0\) and \(0<\theta<1\) in \(\|(f_{\mu_h^*}^t)_*\nu\|_{C^1}\le C\theta^t\|\nu\|_{C^1}\) may also be chosen independently of \(h\).}
 over the
family of sufficiently small couplings.
It remains to verify the strong-to-weak closeness condition uniformly in $h$.
Since $\mu_0^h$ belongs to the local basin of attraction of $\mu_h^*$, we have
\[
\|\mu_j^h-\mu_h^*\|_{C^1}
\leq
C e^{-\gamma j}\|\mu_0-\mu_h^*\|_{C^1},
\]
where $\mu_j^h=F_h^j(\mu_0)$. For sufficiently small $\delta>0$, the
constants $C,\gamma>0$ can be chosen uniformly for
$\|h\|_{C^3}\leq\delta$. In particular,
\[
\sup_{\|h\|_{C^3}\leq\delta}
\|\mu_j^h-\mu_h^*\|_{C^1}
\longrightarrow0
\qquad\text{as }j\to\infty.
\]
By \eqref{regularity2}, for every $\mu\in V_1$,
\[
\|(f_{\mu_j^h})_*\mu-(f_{\mu_h^*})_*\mu\|_{C^0}
\leq
C_1\|h\|_{C^2}\|\mu\|_{C^1}
\|\mu_j^h-\mu_h^*\|_{C^0}.
\]
Consequently,
\[
\|(f_{\mu_j^h})_*-(f_{\mu_h^*})_*\|_{C^1\to C^0}
\leq
C_1\|h\|_{C^2}
\|\mu_j^h-\mu_h^*\|_{C^1}.
\]
Since $\|h\|_{C^2}\leq\|h\|_{C^3}\leq\delta$, it follows that
\[
\sup_{\|h\|_{C^3}\leq\delta}
\|(f_{\mu_j^h})_*-(f_{\mu_h^*})_*\|_{C^1\to C^0}
\longrightarrow0.
\]
Therefore, for every $\varepsilon>0$, there exists $J\in\mathbb N$,
independent of $h$, such that
\[
\|(f_{\mu_j^h})_*-(f_{\mu_h^*})_*\|_{C^1\to C^0}
\leq\varepsilon,
\qquad
j\geq J,\quad
\|h\|_{C^3}\leq\delta.
\]
Thus, the strong-to-weak closeness condition holds uniformly for all
sufficiently small couplings.
We can therefore apply \cite[Lemma 19]{castorrini2025differential} to the
tail sequence $\bigl((f_{\mu_j})_*\bigr)_{j\geq J}$. Since all three
assumptions have been verified with constants independent of $h$, there exist
constants $C,\lambda>0$, independent of $h$ and $j$, such that, for every
$t\in\mathbb N$,
\[
\left\|
\left(
f_{\mu_{j+t-1}}\circ\cdots\circ f_{\mu_j}
\right)_*\nu
\right\|_{C^1}
\leq
Ce^{-\lambda t}\|\nu\|_{C^1},
\qquad
j\geq J,\quad \nu\in V_1.
\]
Choose $\tau\in\mathbb N$ sufficiently large that
\[
Ce^{-\lambda\tau}<1,
\]
and set
\[
q_0:=Ce^{-\lambda\tau}<1.
\]
Taking $t=\tau$ in the above estimate, we obtain, for every $j\geq J$,
\[
\left\|
\left(
f_{\mu_{j+\tau-1}}\circ\cdots\circ f_{\mu_j}
\right)_*
\right\|_{V_1\to V_1}
\leq q_0.
\]
Hence,
\[
\sup_{j\geq J}
\left\|
\left(
f_{\mu_{j+\tau-1}}\circ\cdots\circ f_{\mu_j}
\right)_*
\right\|_{V_1\to V_1}
\leq q_0<1.
\]
This proves the lemma.
\end{proof}

\subsection{Examples of observables satisfying the sublevel estimate}

We verify that the sublevel-set condition in Assumption~\ref{sameLyapunov}
is satisfied by some class of smooth observables. The key point is that the condition prevents the partial derivatives of the observable from remaining close to zero on sets of too large measure.
\begin{example}
Typical examples satisfying the above sublevel condition on $\mathbb{T}^N$ are
\[
\Psi_1(x)
=
\frac{1}{N}\sum_{j=1}^N
\bigl(1-\cos(2\pi x_j)\bigr)^2
\]
and
\[
\Psi_2(x)
=
\frac{1}{N}\sum_{j=1}^N
\sin(2\pi x_j).
\]
For $\Psi_1$, we have
\[
\partial_{x_i}\Psi_1(x)
=
\frac{4\pi}{N}
\bigl(1-\cos(2\pi x_i)\bigr)\sin(2\pi x_i).
\]
Near $x_i=0$, this derivative vanishes to order three, while near
$x_i=\frac12$ it vanishes to order one. Hence, for sufficiently small
$r>0$,
\[
m\left(
\left\{x\in\mathbb{T}^N:
|\partial_{x_i}\Psi_1(x)|\le r
\right\}
\right)
\le C_i r^{1/3}.
\]
Thus $\Psi_1$ satisfies the sublevel condition with
$\alpha_i=\frac13$.  For $\Psi_2$, we have
\[
\partial_{x_i}\Psi_2(x)
=
\frac{2\pi}{N}\cos(2\pi x_i).
\]
The zeros of $\cos(2\pi x_i)$ are simple, and consequently, for
sufficiently small $r>0$,
\[
m\left(
\left\{x\in\mathbb{T}^N:
|\partial_{x_i}\Psi_2(x)|\le r
\right\}
\right)
\le \widetilde C_i r.
\]
Thus $\Psi_2$ satisfies the sublevel condition with $\alpha_i=1$.
\end{example}
\begin{proposition}[Sublevel estimate for analytic observables]
Let 
$
\psi:\mathbb{T}^2\to\mathbb{R}
$
be a real analytic function, and assume that
\[
\partial_{x_i}\psi\not\equiv 0,
\qquad i=1,2 .
\]
Then, for each \(i=1,2\), there exist constants \(C_i>0\) and
\(\alpha_i>0\) such that
\[
m\left(
\left\{
\mathbf{x}\in\mathbb{T}^2:
|\partial_{x_i}\psi(\mathbf{x})|\le r
\right\}
\right)
\le C_i r^{\alpha_i}
\]
for all sufficiently small \(r>0\), where \(m\) denotes the Lebesgue
measure on \(\mathbb{T}^2\).
\end{proposition}
\begin{proof}
Fix \(i\in\{1,2\}\) and define
\[
g_i(\mathbf{x})=\partial_{x_i}\psi(\mathbf{x}).
\]
Since \(\psi\) is real analytic, \(g_i\) is also real analytic on the
compact analytic manifold \(\mathbb{T}^2\). Moreover, by assumption,
\(g_i\not\equiv0\).
Let
\[
Z_i=\{\mathbf{x}\in\mathbb{T}^2:g_i(\mathbf{x})=0\}
\]
be the zero set of \(g_i\). Since \(g_i\) is a non-trivial real analytic
function, \(Z_i\) is a proper analytic subset of \(\mathbb{T}^2\).
The Łojasiewicz inequality implies that there exist constants
\(C>0\) and \(\theta>0\) such that, for all
\(\mathbf{x}\in\mathbb{T}^2\),
\[
\operatorname{dist}(\mathbf{x},Z_i)^\theta
\le C |g_i(\mathbf{x})|.
\]
Therefore, if
$
|g_i(\mathbf{x})|\le r,
$
then
\[
\operatorname{dist}(\mathbf{x},Z_i)
\le C^{1/\theta}r^{1/\theta}.
\]
Consequently,
\[
\{\mathbf{x}:|g_i(\mathbf{x})|\le r\}
\subset
\left\{
\mathbf{x}:
\operatorname{dist}(\mathbf{x},Z_i)
\le C^{1/\theta}r^{1/\theta}
\right\}.
\]
By the standard volume estimates for tubular neighbourhoods of compact
real analytic sets, there exist constants \(C'>0\) and \(\beta>0\) such
that
\[
m\left(
\left\{
\mathbf{x}:
\operatorname{dist}(\mathbf{x},Z_i)\le \varepsilon
\right\}
\right)
\le C'\varepsilon^\beta
\]
for all sufficiently small \(\varepsilon>0\). Hence, taking
$
\varepsilon=C^{1/\theta}r^{1/\theta},
$
we obtain
\[
m\left(
\{\mathbf{x}:|g_i(\mathbf{x})|\le r\}
\right)
\le
C'(C^{1/\theta}r^{1/\theta})^\beta .
\]
Therefore,
\[
m\left(
\{\mathbf{x}:|g_i(\mathbf{x})|\le r\}
\right)
\le C_i r^{\alpha_i},
\]
where
$
\alpha_i=\frac{\beta}{\theta}>0.
$
This proves the desired sublevel estimate.
\end{proof}

\bibliographystyle{alpha}
\bibliography{sample}

\newcommand{\etalchar}[1]{$^{#1}$}
\begin{thebibliography}{TCG{\etalchar{+}}11}

\bibitem[CFVV99]{cencini1999macroscopic}
M~Cencini, Massimo Falcioni, D~Vergni, and A~Vulpiani.
\newblock Macroscopic chaos in globally coupled maps.
\newblock {\em Physica D: Nonlinear Phenomena}, 130(1-2):58--72, 1999.

\bibitem[CGT25]{castorrini2025differential}
Roberto Castorrini, Stefano Galatolo, and Matteo Tanzi.
\newblock The differential of self-consistent transfer operators and the local
  convergence to equilibrium of mean field strongly coupled dynamical systems.
\newblock {\em Journal of Nonlinear Science}, 35(4):78, 2025.

\bibitem[CGT26]{castorrini2026stability}
Roberto Castorrini, Stefano Galatolo, and Matteo Tanzi.
\newblock Stability of fixed points for nonlinear selfconsistent transfer
  operators via cone contractions: R. castorrini et al.
\newblock {\em Journal of Statistical Physics}, 193(2):27, 2026.

\bibitem[Gal22a]{galatolo2022self}
Stefano Galatolo.
\newblock Self-consistent transfer operators: Invariant measures, convergence
  to equilibrium, linear response and control of the statistical properties: S.
  galatolo.
\newblock {\em Communications in Mathematical Physics}, 395(2):715--772, 2022.

\bibitem[Gal22b]{galatolo2022statisticalpropertiesdynamicsintroduction}
Stefano Galatolo.
\newblock Statistical properties of dynamics. introduction to the functional
  analytic approach, 2022.

\bibitem[Kan89]{kaneko1989chaotic}
Kunihiko Kaneko.
\newblock Chaotic but regular posi-nega switch among coded attractors by
  cluster-size variation.
\newblock {\em Physical Review Letters}, 63(3):219, 1989.

\bibitem[Kan90a]{kaneko1990clustering}
Kunihiko Kaneko.
\newblock Clustering, coding, switching, hierarchical ordering, and control in
  a network of chaotic elements.
\newblock {\em Physica D: Nonlinear Phenomena}, 41(2):137--172, 1990.

\bibitem[Kan90b]{kaneko1990globally}
Kunihiko Kaneko.
\newblock Globally coupled chaos violates the law of large numbers but not the
  central-limit theorem.
\newblock {\em Physical review letters}, 65(12):1391, 1990.

\bibitem[Kan92]{kaneko1992mean}
Kunihiko Kaneko.
\newblock Mean field fluctuation of a network of chaotic elements: Remaining
  fluctuation and correlation in the large size limit.
\newblock {\em Physica D: Nonlinear Phenomena}, 55(3-4):368--384, 1992.

\bibitem[Kan95]{kaneko1995remarks}
Kunihiko Kaneko.
\newblock Remarks on the mean field dynamics of networks of chaotic elements.
\newblock {\em Physica D: Nonlinear Phenomena}, 86(1-2):158--170, 1995.

\bibitem[KY10]{koiller2010coupled}
Jos{\'e} Koiller and Lai-Sang Young.
\newblock Coupled map networks.
\newblock {\em Nonlinearity}, 23(5):1121--1141, 2010.

\bibitem[OY03]{ott2003learning}
William Ott and James~A Yorke.
\newblock Learning about reality from observation.
\newblock {\em SIAM Journal on Applied Dynamical Systems}, 2(3):297--322, 2003.

\bibitem[PK94a]{pikovsky1994collective}
Arkady~S Pikovsky and J{\"u}rgen Kurths.
\newblock Collective behavior in ensembles of globally coupled maps.
\newblock {\em Physica D: Nonlinear Phenomena}, 76(4):411--419, 1994.

\bibitem[PK94b]{pikovsky1994globally}
Arkady~S Pikovsky and J{\"u}rgen Kurths.
\newblock Do globally coupled maps really violate the law of large numbers?
\newblock {\em Physical review letters}, 72(11):1644, 1994.

\bibitem[SK97]{shibata1997heterogeneity}
Tatsuo Shibata and Kunihiko Kaneko.
\newblock Heterogeneity-induced order in globally coupled chaotic systems.
\newblock {\em EPL (Europhysics Letters)}, 38(6):417--422, 1997.

\bibitem[SK98]{shibata1998collective}
Tatsuo Shibata and Kunihiko Kaneko.
\newblock Collective chaos.
\newblock {\em Physical review letters}, 81(19):4116, 1998.

\bibitem[Tan22]{tanzi2022meanfieldcoupledsystemsselfconsistent}
Matteo Tanzi.
\newblock Mean-field coupled systems and self-consistent transfer operators: A
  review, 2022.

\bibitem[TC13]{takeuchi2013collective}
Kazumasa~A Takeuchi and Hugues Chat{\'e}.
\newblock Collective lyapunov modes.
\newblock {\em Journal of Physics A: Mathematical and Theoretical},
  46(25):254007, 2013.

\bibitem[TCG{\etalchar{+}}11]{takeuchi2011extensive}
Kazumasa~A Takeuchi, Hugues Chat{\'e}, Francesco Ginelli, Antonio Politi, and
  Alessandro Torcini.
\newblock Extensive and subextensive chaos in globally coupled dynamical
  systems.
\newblock {\em Physical review letters}, 107(12):124101, 2011.

\bibitem[TGC09]{takeuchi2009lyapunov}
Kazumasa~A Takeuchi, Francesco Ginelli, and Hugues Chat{\'e}.
\newblock Lyapunov analysis captures the collective dynamics of large chaotic
  systems.
\newblock {\em Physical review letters}, 103(15):154103, 2009.

\bibitem[Via14]{viana2014lectures}
Marcelo Viana.
\newblock {\em Lectures on Lyapunov Exponents}, volume 145 of {\em Cambridge
  Studies in Advanced Mathematics}.
\newblock Cambridge University Press, 2014.

\bibitem[VLP21]{velasco2021nonuniversal}
David Velasco, Juan~M L{\'o}pez, and Diego Paz{\'o}.
\newblock Nonuniversal large-size asymptotics of the lyapunov exponent in
  turbulent globally coupled maps.
\newblock {\em arXiv preprint arXiv:2110.01949}, 2021.

\bibitem[WH89]{wiesenfeld1989attractor}
Kurt Wiesenfeld and Peter Hadley.
\newblock Attractor crowding in oscillator arrays.
\newblock {\em Physical Review Letters}, 62(12):1335, 1989.

\end{thebibliography}
\end{document}